\documentclass[letterpaper, 11pt,  reqno]{amsart}
\usepackage{graphicx} 
\usepackage{amsmath,amssymb,amscd,amsthm,amsxtra,esint}
\usepackage{color}
\usepackage[margin=1.1in,marginparwidth=1.5cm, marginparsep=0.5cm]{geometry} 
\usepackage{tikz-cd}
\tikzset{smalltext/.style={"\textup{\small #1}" description}}

\usepackage[markup=default]{changes}
\definechangesauthor[name={Reviewer 1}, color=orange]{R1} 
\definechangesauthor[name={Reviewer 2}, color=blue]{R2} 
\definechangesauthor[name={Reviewer 3}, color=green]{R3}

\usepackage[implicit=true]{hyperref}

\allowdisplaybreaks[4]

\definecolor{gr}{rgb}   {0.,   0.69,   0.23 }
\definecolor{bl}{rgb}   {0.,   0.5,   1. }
\definecolor{mg}{rgb}   {0.85,  0.,    0.85}
\definecolor{yl}{rgb}   {0.8,  0.7,   0.}
\definecolor{or}{rgb}  {0.7,0.2,0.2}

\DeclareMathOperator{\Law}{Law}

\newcommand{\noi}{\noindent}
\newcommand{\R}{\mathbb{R}}
\newcommand{\T}{\mathbb{T}}
\newcommand{\Z}{\mathbb{Z}}
\newcommand{\N}{\mathbb{N}}

\newcommand{\E}{\mathbb{E}}
\newcommand{\PP}{\mathbb{P}}
\newcommand{\D}{\mathcal{D}}
\renewcommand{\H}{\mathcal{H}}
\newcommand{\F}{\mathcal{F}}

\newcommand{\C}{\mathcal{C}}

\newcommand{\A}{\mathcal{A}}

\newcommand{\al}{\alpha}
\newcommand{\be}{\beta}
\newcommand{\ta}{\theta}
\newcommand{\s}{\sigma}
\newcommand{\eps}{\varepsilon}
\newcommand{\om}{\omega}
\newcommand{\Om}{\Omega}
\newcommand{\Dl}{\Delta}
\newcommand{\dl}{\delta}
\newcommand{\ld}{\lambda}
\newcommand{\nb}{\nabla}

\newcommand{\dt}{\partial_t}
\newcommand{\ind}{\mathbf 1}
\newcommand{\ft}{\widehat}

\newcommand{\les}{\lesssim}
\newcommand{\ges}{\gtrsim}
\newcommand{\cj}{\overline}

\newcommand{\uu}{\mathbf{u}}
\newcommand{\vv}{\mathbf{v}}
\newcommand{\ww}{\mathbf{w}}

\newcommand{\HLSe}{$\textup{HLSM}_{\eps, N}$ }
\newcommand{\PLS}{$\textup{PLSM}_N$ }
\newcommand{\NLWe}{$\textup{SdNLW}_\eps$ }

\newcommand{\1}{\hspace{0.5mm}\text{I}\hspace{0.5mm}}
\newcommand{\II}{\text{I \hspace{-2.8mm} I} }
\newcommand{\III}{\text{I \hspace{-2.9mm} I \hspace{-2.9mm} I}}
\newcommand{\jb}[1]{\langle #1 \rangle}

\newcommand{\wick}[1]{:\!{#1}\!:}
\newcommand{\deff}{\stackrel{\textup{def}}{=}}

\newtheorem{theorem}{Theorem}[section]

\newtheorem{lemma}[theorem]{Lemma}
\newtheorem{proposition}[theorem]{Proposition}

\newtheorem{remark}[theorem]{Remark}

\newtheorem*{ackno}{Acknowledgements}

\numberwithin{equation}{section}
\numberwithin{theorem}{section}

\makeatletter
\@namedef{subjclassname@2020}{%
  \textup{2020} Mathematics Subject Classification}
\makeatother

\title[S-K approximation of the HLSM and its mean-field limit]{On the Smoluchowski-Kramers approximation \\
for the hyperbolic $O(N)$ linear sigma model \\
and its mean-field limit}

\author[A.~Czernik, R.~Liu, and S.~Liu]{Alexander Czernik, Ruoyuan Liu, and Shao Liu}

\address{
Alexander Czernik, Mathematical Institute\\
University of Bonn\\
Endenicher Allee 60\\
53115\\
Bonn\\
Germany}

\email{s6alczer@uni-bonn.de}

\address{
Ruoyuan Liu, Mathematical Institute\\
University of Bonn\\
Endenicher Allee 60\\
53115\\
Bonn\\
Germany}

\email{ruoyuanl@math.uni-bonn.de}

\address{
Shao Liu, Mathematical Institute\\
University of Bonn\\
Endenicher Allee 60\\
53115\\
Bonn\\
Germany}

\email{shaoliu@math.uni-bonn.de}


\begin{document}

\baselineskip = 14pt


\begin{abstract} 
We study the hyperbolic $O(N)$ linear sigma model, i.e.~a system of $N$ interacting stochastic damped nonlinear wave equations (SdNLW) with coupled cubic nonlinearities, posed on the two-dimensional torus and indexed by a parameter $\eps > 0$. We show that as $\eps$ goes to zero (Smoluchowski-Kramers approximation) and $N$ goes to infinity (mean-field limit), each component of the solution to the SdNLW system converges to the solution to the stochastic nonlinear heat equation (SNLH) with a mean-field nonlinearity. We prove such convergence via two regimes: first with $\eps$ going to zero to obtain the parabolic $O(N)$ linear sigma model, i.e.~a system of $N$ coupled SNLH, and then with $N$ going to infinity; or first with $N$ going to infinity for each component to obtain the mean-field SdNLW and then with $\eps$ going to zero. As a result, we obtain a commutative diagram regarding the convergence from the hyperbolic $O(N)$ linear sigma model to the mean-field SNLH.
\end{abstract}


\maketitle

\tableofcontents

\section{Introduction}

\subsection{Hyperbolic $O(N)$ linear sigma model and the Smoluchowski-Kramers approximation}

In this paper, we study the following coupled system of stochastic damped nonlinear wave equations (SdNLW) on the two-dimensional torus $\T^2 = (\R / 2 \pi \Z)^2$ with a parameter $0 < \eps \leq 1$:
\begin{align}
    (\eps^2 \dt^2 + \dt + 1 - \Dl) u_\eps^{N, j} = - \frac{1}{N} \sum_{k = 1}^N (u_\eps^{N, k})^2 u_\eps^{N, j} + \sqrt{2} \xi^j , \qquad j = 1 , \dots, N ,
\label{NLWNe0}
\end{align}

\noi
where $\xi^j$'s are independent real-valued space-time white noises on $\R_+ \times \T^2$. When $\eps = 1$, the system \eqref{NLWNe0} is also referred to as the hyperbolic $O(N)$ linear sigma model, whose study was initiated by Oh, the second author, and the third author in \cite{LLO}. For later convenience, we address the system \eqref{NLWNe0} by the abbreviation $\text{HLSM}_{\eps, N}$, which stands for the hyperbolic $O(N)$ linear sigma model indexed by a parameter $\eps$. Our goal is to show that each component $u_\eps^{N, j}$ of the solution  to \eqref{NLWNe0} converges globally in time to the solution $u^j$ of the stochastic nonlinear heat equation (SNLH) with a mean-field nonlinearity:
\begin{align}
    (\dt + 1 - \Dl) u^j = - \E [(u^j)^2] u^j + \sqrt{2} \xi^j .
\label{NLH0}
\end{align}

\noi
We exploit such convergence via the following two regimes:
\begin{enumerate}
    \item[I.] We first let $\eps \to 0$ to obtain $u^{N, j}$ with $1 \leq j \leq N$ satisfying the following coupled system of SNLH, which we also refer to as the parabolic $O(N)$ linear sigma model ($\text{PLSM}_N$):
    \begin{align}
        (\dt + 1 - \Dl) u^{N, j} = - \frac{1}{N} \sum_{k = 1}^N (u^{N, k})^2 u^{N, j} + \sqrt{2} \xi^j .
    \label{NLHN0}
    \end{align}

    \noi
    We then take $N \to \infty$ to obtain $u^j$ satisfying \eqref{NLH0}.
    
    \item[II.] We first let $N \to \infty$ to obtain $u_\eps^j$ satisfying the following mean-field SdNLW indexed by a parameter $0 < \eps \leq 1$ (mean-field $\text{SdNLW}_\eps$):
    \begin{align}
        (\eps^2 \dt^2 + \dt + 1 - \Dl) u_\eps^j = - \E [(u_\eps^j)^2] u_\eps^j + \sqrt{2} \xi^j .
    \label{NLWe0}
    \end{align}

    \noi
    We then take $\eps \to 0$ to obtain $u^j$ satisfying \eqref{NLH0}.
\end{enumerate}

\noi
As a result, we obtain the following commutative diagram for global-in-time dynamics:
\begin{figure}[htbp] 
\centering 
    \begin{tikzcd}[cells={nodes={draw}}, row sep=2cm, column sep=4.5cm]
    \textup{\HLSe \eqref{NLWNe0}}
    \arrow[r, smalltext=\normalsize $N \to \infty$] 
    \arrow[d, smalltext=\normalsize $\eps \to 0$]
    & \textup{Mean-field \NLWe \eqref{NLWe0}} \arrow[d, smalltext=\normalsize $\eps \to 0$]\\
    \textup{\PLS \eqref{NLHN0}}\arrow[r, smalltext=\normalsize $N \to \infty$] & \textup{Mean-field SNLH \eqref{NLH0}}
\end{tikzcd}
\caption{Commutative diagram between \HLSe and the mean-field SNLH.}
\label{fig:diagram} 
\end{figure}
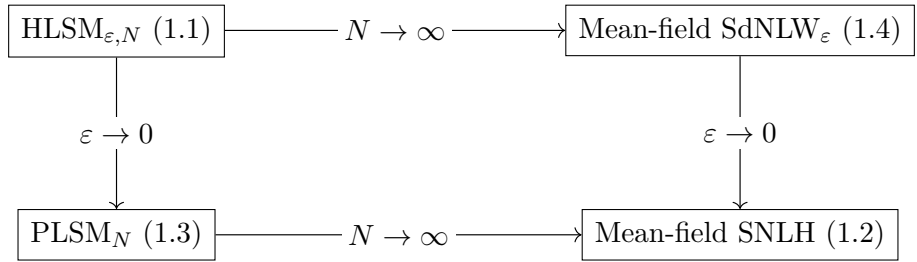

\noi
The process of taking $\eps \to 0$ is called the {\it Smoluchowski-Kramers approximation} \cite{Smo, Kra}, which is the main topic in this paper. The process of taking $N \to \infty$ is called the {\it mean-field limit}, which has been studied in \cite{SSZZ} for SNLH and \cite{LLO} for SdNLW.

Let us first discuss some background on the hyperbolic $O(N)$ linear sigma model and the Smoluchowski-Kramers approximation.
In Euclidean quantum field theory, models are described by probability measures on spaces of distributions (see \cite{GJ} for a more detailed introduction). The stochastic quantization paradigm, introduced by Parisi-Wu in \cite{PW}, studies such measures by constructing a stochastic partial differential equation (SPDE) whose unique invariant measure is the target quantum field theory measure. 
One of the most extensively studied examples is the scalar-valued $\Phi^4_2$-model (see \eqref{rhoN0} in Appendix~\ref{APP:inv} with $N = 1$), where the superscript ``4'' denotes the quartic interaction and the subscript ``2'' the spatial dimension. Its parabolic stochastic quantization is given by the SNLH \eqref{NLHN0} with $N = 1$, whose local well-posedness theory was established by Da Prato-Debussche in \cite{DPD} and global theory by Mourrat-Weber in \cite{MW17}; see also \cite{TW18} for a unique ergodicity result. While the scalar-valued case is fundamental, many physically relevant models involve fields taking values in an $N$-dimensional target space. A particular example is the $N$-component $\Phi^4_2$-model introduced by Wilson in \cite{Wil}, given by \eqref{rhoN0}, which is also known as the $O(N)$ linear sigma model. The term ``linear'' refers to the fact that the target space $\R^N$ is a linear space. The corresponding parabolic stochastic quantization is given by the \PLS \eqref{NLHN0}, which was studied by Shen-Smith-Zhu-Zhu in \cite{SSZZ}. 
See also \cite{SZZ} for the (spatial) 3-dimensional setting and also a nice survey paper \cite{Shen}.

In addition to the parabolic approach, there has been significant progress in the study of hyperbolic stochastic quantization, also known as canonical stochastic quantization or Hamiltonian stochastic quantization; see \cite{RSS}. For the scalar $\Phi^4_2$-model, the corresponding hyperbolic stochastic quantization is described by the SdNLW \eqref{NLWNe0} with $N = 1$ and $\eps = 1$. Local well-posedness for this equation was established by Gubinelli-Koch-Oh in \cite{GKO}, while global well-posedness was subsequently proved by Gubinelli-Koch-Oh-Tolomeo in \cite{GKOT}; see also \cite{Tol2} for a unique ergodicity result.
More recently, the hyperbolic framework has been extended beyond the scalar setting. In \cite{LLO}, Oh, the second author, and the third author initiated the study of the $N$-component model by introducing the \HLSe \eqref{NLWNe0} with $\eps = 1$, providing a hyperbolic stochastic quantization of the $O(N)$ linear sigma model.

In the present work, we investigate the convergence of the \HLSe \eqref{NLWNe0} to the \PLS \eqref{NLHN0} in the limit of a vanishing second-order (mass) term ($\eps \to 0$), alongside their mean-field convergence. This zero-mass limit is known as the Smoluchowski-Kramers approximation and has been studied across various finite and infinite dimensional settings; see \cite{CF06-1, CF06-2, CS14, CS16, CFS17, CS17, Sa19, CG20, CX22, CX23, CX24, CB25, CD25, CX25, GV25, Bla26, BL26, LLX26, XZ26}. In the context of stochastic quantization of the $\Phi^4_2$-model, such approximation has been studied by Fukuizumi-Hoshino-Inui in \cite{FHI} and by Zine in \cite{Zine}. 
From an analytical perspective, the approximation connects two fundamentally different analytical regimes: a parabolic dynamics driven by the strong instantaneous smoothing of the heat semigroup and a hyperbolic dynamics governed by dispersive and Hamiltonian effects. 
From a computational perspective, it links overdamped and underdamped Langevin dynamics, balancing the improved state-space exploration of the latter against the greater numerical stability of the former; see \cite{BKKLSW, DKB, Neal}. In particular, it provides a mathematical justification for treating $\eps$ as a tuning parameter without changing the underlying target quantum field theory measure.

The present work naturally extends \cite{Zine}, where the author treated the scalar case $N = 1$. The presence of the additional parameter $N$ leads to a natural question: do the limits $\eps \to 0$ and $N \to \infty$ commute? 
At first glance, the two limiting procedures appear largely independent, since $\eps$ enters only through the second-order time derivative, whereas $N$ appears only in the nonlinearity.
Indeed, we are able to prove the convergence in either order to the unique solution of the limiting equation (the mean-field SNLH \eqref{NLH0}); see Figure~\ref{fig:diagram}.
Nevertheless, one of the main findings of this paper is that, although both convergence routes ultimately lead to the same limiting equation, the analysis of the limit $\eps \to 0$ for global-in-time dynamics differs substantially depending on the order in which the limits are taken. We provide a more detailed discussion in Subsection~\ref{SUB:intro}. 
Roughly speaking, this distinction stems from the fact that the coupled systems \eqref{NLWNe0} and \eqref{NLHN0} are defined pathwise, for each $\om$ in a full-measure subset of $\Om$, whereas the mean-field equations \eqref{NLH0} and \eqref{NLWe0} involve averaging over the entire probability space.

\subsection{Main results on convergence of dynamics}
\label{SUB:intro}

Before stating the main results regarding the commutative diagram in Figure~\ref{fig:diagram}, let us take a closer look at the four equations  \eqref{NLWNe0}, \eqref{NLH0}, \eqref{NLHN0}, and \eqref{NLWe0}.

Let us first look at \HLSe \eqref{NLWNe0}. Given $j \in \N$ and $0 < \eps \leq 1$, we define $\Psi_\eps^j$ as the stochastic convolution satisfying the following linear stochastic damped wave equation with a parameter $\eps$ and zero initial data:
\begin{align}
\begin{cases}
    (\eps^2 \dt^2 + \dt + 1 - \Dl) \Psi_\eps^j = \sqrt{2} \xi^j \\
    (\Psi_\eps^j, \dt \Psi_\eps^j) |_{t = 0} = (0, 0) .
\end{cases}
\label{Psije}
\end{align}

\noi
As can be seen from \eqref{sigmaM} below, the frequency truncated version of $\Psi_\eps^j$ has variance that diverges as the truncation is removed.
This means that the stochastic convolution $\Psi_\eps^j$ is almost surely a distribution but in general not a function on $\T^2$, and so we need to make sense of its powers through renormalization. 

Let us proceed as in \cite{GKO, GKOT, Zine, LLO} with the following first order expansion (see \cite{McK, Bour96}), also known as the Da~Prato-Debussche trick in the parabolic setting (see \cite{DPD}):
\begin{align}
    u_\eps^{N, j} = \Psi_\eps^j + v_\eps^{N, j} , \qquad j = 1, \dots, N ,
\label{expNe}
\end{align}

\noi
where $v_\eps^{N, j}$ is a remainder term satisfying
\begin{align}
\begin{split}
    (\eps^2 \dt^2 + \dt + 1 - \Dl) v_\eps^{N, j} &= - \frac 1N \sum_{k = 1}^N \big( \wick{ (\Psi_\eps^k)^2 \Psi_\eps^j } + \wick{ (\Psi_\eps^k)^2 } v_\eps^{N, j} + 2 v_\eps^{N, k} \wick{ \Psi_\eps^k \Psi_\eps^j } \\
    &\quad + 2 \Psi_\eps^k v_\eps^{N, k} v_\eps^{N, j} + (v_\eps^{N, k})^2 \Psi_\eps^j + (v_\eps^{N, k})^2 v_\eps^{N, j} \big) , \qquad j = 1, \dots, N.
\end{split}
\label{NLWNev}
\end{align}

\noi
On the right-hand side of \eqref{NLWNev}, we have already applied Wick renormalization to the products and powers of $\Psi_\eps^j$'s. See \eqref{wick0} and \eqref{wick} for the precise definitions, but at the moment we make the following interpretation:
\begin{align*}
    \wick{ \Psi_\eps^k \Psi_\eps^j }  &=
    \begin{cases}
        \wick{(\Psi_\eps^j)^2} & \text{if } k = j \\
        \Psi_\eps^k \Psi_\eps^j & \text{if } k \neq j ,
    \end{cases} \\
    \wick{ (\Psi_\eps^k)^2 \Psi_\eps^j }  &=
    \begin{cases}
        \wick{(\Psi_\eps^j)^3} & \text{if } k = j \\
        \wick{(\Psi_\eps^k)^2} \Psi_\eps^j & \text{if } k \neq j ,
    \end{cases}
\end{align*}

\noi
where $\wick{(\Psi_\eps^j)^2}$ and $\wick{(\Psi_\eps^j)^3}$ denote the standard Wick powers.
Then, by setting
\begin{align*}
    \wick{ (u_\eps^{N, k})^2 u_\eps^{N, j} } \, \deff &\big( \wick{ (\Psi_\eps^k)^2 \Psi_\eps^j } + \wick{ (\Psi_\eps^k)^2 } v_\eps^{N, j} + 2 v_\eps^{N, k} \wick{ \Psi_\eps^k \Psi_\eps^j } \\
    &\quad + 2 \Psi_\eps^k v_\eps^{N, k} v_\eps^{N, j} + (v_\eps^{N, k})^2 \Psi_\eps^j + (v_\eps^{N, k})^2 v_\eps^{N, j} \big) ,
\end{align*}

\noi
we get the following renormalized version of \HLSe \eqref{NLWNe0}:
\begin{align}
    (\eps^2 \dt^2 + \dt + 1 - \Dl) u_\eps^{N, j} = - \frac{1}{N} \sum_{k = 1}^N \wick{ (u_\eps^{N, k})^2 u_\eps^{N, j} } + \, \sqrt{2} \xi^j , \qquad j = 1, \dots, N.
\label{NLWNeu}
\end{align}

\noi
As in \cite{GKO, GKOT, LLO}, we say that $\uu_\eps^N = (u_\eps^{N, j})_{1 \leq j \leq N}$ is a solution to the renormalized \HLSe \eqref{NLWNeu} if $u_\eps^{N, j}$ is given by \eqref{expNe} with $\vv_\eps^N = (v_\eps^{N, j})_{1 \leq j \leq N}$ being a solution to the perturbed \HLSe \eqref{NLWNev}.

In order to discuss the convergence problem later on, we need to know that the equation \eqref{NLWNeu} does have a solution. This is provided by the following proposition. For notations, we write
\begin{align*}
    \H^s (\T^2) = H^s (\T^2) \times H^{s - 1} (\T^2) ,
\end{align*}

\noi
where $H^s$ is the $L^2$-based Sobolev space with regularity $s \in \R$. Given $N \in \N$ and a function space $X$, we write $X^{\otimes N}$ as the $N$-fold product of $X$ with itself: $X \times \cdots \times X$.

\begin{proposition}[Pathwise global well-posedness for $\text{HLSM}_{\eps, N}$]
\label{PROP:GWP_NLW}
Let $0 < \eps \leq 1$, $N \in \N$, and $\frac 45 < s < 1$. Let $(\uu_0^N, \uu_1^N) = ((u_0^{N, j}, u_1^{N, j}))_{1 \leq j \leq N} \in \H^s (\T^2)^{\otimes N}$. Then, there exists a unique solution $(\uu_\eps^N, \dt \uu_\eps^N) = ((u_\eps^{N, j}, \dt u_\eps^{N, j}))_{1 \leq j \leq N}$ to the renormalized \HLSe \eqref{NLWNeu} with initial data $(\uu_0^N, \uu_1^N)$ almost surely in the space
\begin{align*}
    ((\Psi_\eps^j, \dt \Psi_\eps^j))_{1 \leq j \leq N} + C (\R_+ ; \H^s (\T^2))^{\otimes N} .
\end{align*}
\end{proposition}

The proof of Proposition~\ref{PROP:GWP_NLW} is similar to that of \cite[Theorem~1.1]{LLO}, which we now briefly describe. Our goal here is to prove global well-posedness of the perturbed \HLSe \eqref{NLWNev} for $(\vv_\eps^N, \dt \vv_\eps^N) = ((v_\eps^{N, j}, \dt v_\eps^{N, j}))_{1 \leq j \leq N}$ with initial data $(\uu_0^{N}, \uu_1^{N})$. In view of the one degree of smoothing from the linear damped wave dynamics from \eqref{defIe} (we allow dependence on $\eps$ here) and the regularity properties of the Wick products and powers of the stochastic convolution $\Psi_\eps^j$ in Lemma~\ref{LEM:sto}, we can prove pathwise local well-posedness for \eqref{NLWNev} using the same framework as in \cite[Proposition~3.1]{LLO}. Note that the proof does not use any auxiliary function space and so the uniqueness of a solution $(\vv_\eps^N, \dt \vv_\eps^N)$ holds in the entire space $C ([0, T] ; \H^s (\T^2))^{\otimes N}$ for some almost surely positive time $T$.

To achieve global well-posedness, we aim to use an energy functional to establish a global-in-time a priori bound for the solution, which allows us to iterate the local well-posedness argument. 
Given a tuple of space-time functions $(\ww^N, \dt \ww^N) = ((w^{N, j}, \dt w^{N, j}))_{1 \leq j \leq N}$, we define the energy functional $E_\eps^N (\ww^N, \dt \ww^N)$ by
\begin{align}
\begin{split}
    E_\eps^N (\ww^N, \dt \ww^N) &\deff \frac 12 \sum_{j = 1}^N \int_{\T^2} \big( (\eps \dt w^{N, j})^2 + (w^{N, j})^2 + |\nb w^{N, j}|^2 \big) dx \\
    &\quad + \frac{1}{4N} \int_{\T^2} \Big( \sum_{j = 1}^N (w^{N, j})^2 \Big)^2 dx .
\end{split}
\label{defEeN}
\end{align}

\noi
It is not hard to check that $E_\eps^N (\ww^N, \dt \ww^N)$ is conserved under the flow of the following coupled NLW system:
\begin{align}
    (\eps^2 \dt^2 + 1 - \Dl) w^{N, j} = - \frac{1}{N} \sum_{k = 1}^N (w^{N, k})^2 w^{N, j} , \qquad j = 1, \dots, N .
\label{NLWsys}
\end{align}

\noi
In fact, we can write \eqref{NLWsys} as the following Hamiltonian formulation:
\begin{align}
    \eps^2 \dt \begin{pmatrix}
        \ww^N \\ \dt \ww^N
    \end{pmatrix} 
    = \begin{pmatrix}
        0 & \textbf{Id}^N \\
        - \textbf{Id}^N & 0
    \end{pmatrix}
    \nb_{(L^2 (\T^2))^{\otimes 2N}} E_\eps^N (\ww^N, \dt \ww^N) ,
\label{Hamil}
\end{align}

\noi
where $\textbf{Id}^N$ denotes the $N \times N$ identity matrix and $\nb_{(L^2 (\T^2))^{\otimes 2N}}$ denotes the Fr\'echet derivative with respect to the norm $(L^2 (\T^2))^{\otimes 2N} = L^2 (\T^2) \times \cdots \times L^2 (\T^2)$.
As in the situation in \cite{GKOT}, there are two difficulties in proving pathwise global well-posedness for \eqref{NLWNev}. Firstly, due to the roughness of the stochastic objects in \eqref{NLWNev}, the local-in-time solution $(\vv_\eps^N (t), \dt \vv_\eps^N (t))$ does not belong to the energy space $(\H^1 (\T^2))^{\otimes N}$. Secondly, the energy $E_\eps^N (\vv_\eps^N, \dt \vv_\eps^N)$ is not conserved under the flow of the perturbed \HLSe \eqref{NLWNev}, even if we assume that $(\vv_\eps^N (t), \dt \vv_\eps^N (t)) \in (\H^1 (\T^2))^{\otimes N}$. The above issues have been solved by Gubinelli-Koch-Oh-Tolomeo in \cite{GKOT} via a hybrid approach combining the $I$-method \cite{CKSTT02} (for dealing with the first issue) and a Gronwall-type argument \cite{BTz14} (for dealing with the second issue). This hybrid approach was adapted to the vector-valued setting in \cite{LLO}. With minor modifications to allow for dependence on $\eps$, this argument also applies to proving pathwise global well-posedness for \eqref{NLWNev}.

\medskip
For the perturbed \HLSe \eqref{NLWNev}, we now take $\eps \to 0$ and $N \to \infty$ to formally obtain the limiting equation:
\begin{align}
    (\dt + 1 - \Dl) v^j = - 2 \E [\Psi_0^j v^j] \Psi_0^j - 2 \E [\Psi_0^j v^j] v^j - \E [(v^j)^2] \Psi_0^j - \E [(v^j)^2] v^j ,
\label{NLHv}
\end{align}

\noi
where $\Psi_0^j$ is the stochastic convolution satisfying the following linear stochastic heat equation with zero initial data:
\begin{align}
\begin{cases}
    (\dt + 1 - \Dl) \Psi_0^j = \sqrt{2} \xi^j \\
    \Psi_0^j |_{t = 0} = 0 .
\end{cases}
\label{Psij0}
\end{align}

\noi
By writing 
\begin{align}
    u^j = \Psi_0^j + v^j ,
\label{exp}
\end{align}

\noi
we can write out the following renormalized version of the mean-field SNLH \eqref{NLH0}:
\begin{align}
    (\dt + 1 - \Dl) u^j = - \E [(u^j)^2 - (\Psi_0^j)^2] u^j + \sqrt{2} \xi^j.
\label{NLHu}
\end{align}

\noi
As before, we say that $u^j$ is a solution to the renormalized mean-field SNLH \eqref{NLHu} if $u^j$ is given by \eqref{exp} with $v^j$ being a solution to the perturbed mean-field SNLH \eqref{NLHv}.

We also have the following well-posedness result for the equation \eqref{NLHu}.

\begin{proposition}[Global well-posedness for the mean-field SNLH]
\label{PROP:GWP_NLH}
Let $\frac 12 < s < 1$ and $u_0^j \in L^4 (\Omega; H^{s} (\T^2))$ given $j \in \N$. Then, for any $T > 0$, there exists a unique solution $u^j$ to the renormalized mean-field SNLH \eqref{NLHu} with initial data $u_0^j$ in the space
\begin{align*}
    \Psi_0^j + L^2 (\Om; C([0, T]; H^s (\T^2))) .
\end{align*}
\end{proposition}

Global well-posedness for the renormalized mean-field SNLH was established in \cite[Theorem~3.6]{SSZZ}, but in a different formulation from that of our Proposition~\ref{PROP:GWP_NLH}. The initial condition in \cite[Theorem~3.6]{SSZZ} is $u_0^j \in L^4 (\Om; L^4 (\T^2))$ and a unique global-in-time solution to the perturbed mean-field SNLH \eqref{NLHv} exists in the space
\begin{align*}
    L^2 (\Om; C ((0, T]; \C^\be (\T^2)) \cap C([0, T]; L^4 (\T^2)) ) 
\end{align*}

\noi
for any $T > 0$ and some sufficiently small $\be > 0$, where $\C^s (\T^2)$ denotes the Besov-H\"older space on $\T^2$ with regularity $s$ (see Subsection~\ref{SUB:sob}). In this paper, we choose to treat the wave equations and the heat equations in a uniform manner by putting the initial data and the solutions in $L^2$-based Sobolev spaces with more or less the same regularity. In order to obtain the desired regularity as claimed in Proposition~\ref{PROP:GWP_NLH}, we need some additional treatment based on the result in \cite{SSZZ}, which we present in Subsection~\ref{SUB:mfSNLH}.

\medskip
We now take a closer look at the two regimes of the convergence problem, starting with regime~I. For the perturbed \HLSe \eqref{NLWNev}, we fix $N \in \N$ and take $\eps \to 0$ to formally obtain the following perturbed $\text{PLSM}_N$:
\begin{align}
\begin{split}
    (\dt + 1 - \Dl) v^{N, j} = - \frac 1N &\sum_{k = 1}^N \big( \wick{ (\Psi_0^k)^2 \Psi_0^j } + \wick{ (\Psi_0^k)^2 } v^{N, j} + 2 v^{N, k} \wick{ \Psi_0^k \Psi_0^j } \\
    &\, + 2 \Psi_0^k v^{N, k} v^{N, j} + (v^{N, k})^2 \Psi_0^j + (v^{N, k})^2 v^{N, j} \big) , \qquad j = 1, \dots, N.
\end{split}
\label{NLHNv}
\end{align}

\noi
By writing
\begin{align}
    u^{N, j} = \Psi_0^j + v^{N, j} , \qquad j = 1, \dots, N ,
\label{expN}
\end{align}

\noi
and setting
\begin{align*}
    \sum_{k = 1}^N \wick{ (u^{N, k})^2 u^{N, j} } \, \deff &\sum_{k = 1}^N \big( \wick{ (\Psi_0^k)^2 \Psi_0^j } + \wick{ (\Psi_0^k)^2 } v^{N, j} + 2 v^{N, k} \wick{ \Psi_0^k \Psi_0^j } \\
    &\quad + 2 \Psi_0^k v^{N, k} v^{N, j} + (v^{N, k})^2 \Psi_0^j + (v^{N, k})^2 v^{N, j} \big) ,
\end{align*}

\noi
we get the following renormalized version of the $\text{PLSM}_N$:
\begin{align}
    (\dt + 1 - \Dl) u^{N, j} = - \frac{1}{N} \sum_{k = 1}^N \wick{ (u^{N, k})^2 u^{N, j} } + \, \sqrt{2} \xi^j , \qquad j = 1, \dots, N.
\label{NLHNu}
\end{align}

\noi
As before, we say that $\uu^N = (u^{N, j})_{1 \leq j \leq N}$ is a solution to the renormalized \PLS \eqref{NLHNu} if $u^{N, j}$ is given by \eqref{expN} with $\vv^N = (v^{N, j})_{1 \leq j \leq N}$ being a solution to \eqref{NLHNv}.

We are now ready to state our convergence result via regime I. Given $N \in \N$, a function space $X$, and $\mathbf{f} = (f^{j})_{1 \leq j \leq N} \in X^{\otimes N}$, we define the norm
\begin{align}
    \| \mathbf{f} \|_{\A_N X} = \| f^{j} \|_{\A_{N, j} X} \deff \bigg( \frac{1}{N} \sum_{j = 1}^N \| f^{j} \|_X^2 \bigg)^{\frac 12} .
\label{defANX}
\end{align}

\begin{theorem}[Convergence via regime I]
\label{THM:conv1}
Let $\frac 45 < s < 1$. Given $N \in \N$, let $(\uu_0^N, \uu_1^N) = ((u_0^{N, j}, u_1^{N, j}))_{1 \leq j \leq N}$ be random and belong to  $\H^s (\T^2)^{\otimes N}$ almost surely and let $u_0^j \in L^4 ( \Om; H^s (\T^2))$ given $j \in \N$. Assume that $\{ u_0^j \}_{j \in \N}$ are independent and identically distributed. Also, we assume that the following convergences hold in probability:
\begin{align*}
    &\text{For each $j \in \N$, } \| u_0^{N, j} - u_0^j \|_{H^s} \longrightarrow 0 \quad \text{as } N \to \infty; \\
    &\| u_0^{N, j} - u_0^j \|_{\A_{N, j} H^s} \longrightarrow 0 \quad \text{as } N \to \infty.
\end{align*}

\smallskip \noi
\textup{(i)~(Pathwise global well-posedness for $\text{PLSM}_N$)} There exists a unique solution $\uu^N = (u^{N, j})_{1 \leq j \leq N}$ to the renormalized \PLS \eqref{NLHNu} with initial data $\uu_0^N$ almost surely in the space
\begin{align*}
    (\Psi_0^j)_{1 \leq j \leq N} + C (\R_+; H^s (\T^2)^{\otimes N}) .
\end{align*}

\smallskip \noi
\textup{(ii)~(Convergence from \HLSe to $\text{PLSM}_N$)} 
Let $N \in \N$ be fixed. Given $0 < \eps \leq 1$, let $\uu_\eps^N = (u_\eps^{N, j})_{1 \leq j \leq N}$ be the solution to the renormalized \HLSe \eqref{NLWNeu} with initial data $(\uu_0^N, \uu_1^N)$ provided by Proposition~\ref{PROP:GWP_NLW}. Let $\uu^N = (u^{N, j})_{1 \leq j \leq N}$ be given by part \textup{(i)}. Then, $\{ \uu_\eps^N \}_{\eps \in (0, 1]}$ converges almost surely to $\uu^N$ in $C(\R_+; H^{- \sigma} (\T^2)^{\otimes N})$, endowed with the compact open topology in time, as $\eps \to 0$ for any $\sigma > 0$.

\smallskip \noi
\textup{(iii)~(Convergence from \PLS to the mean-field SNLH)} 
Let $\uu^N = (u^{N, j})_{1 \leq j \leq N}$ be given by part \textup{(i)}. Fix $j \in \N$ and let $u^j$ be the solution to the renormalized mean-field SNLH \eqref{NLHu} with initial data $u_0^j$ provided by Proposition~\ref{PROP:GWP_NLH}. Then, $\{ u^{N, j} \}_{N \in \N}$ converges in probability to $u^j$ in $C (\R_+; H^{- \sigma} (\T^2))$, endowed with the compact open topology in time, as $N \to \infty$ for any $\sigma > 0$.
\end{theorem}

For part (i) of Theorem~\ref{THM:conv1}, we note that pathwise global well-posedness for the \PLS was established in \cite[Lemma~2.2]{SSZZ}, but, once again, in a different formulation from that of ours. In \cite[Lemma~2.2]{SSZZ}, the initial condition is $\uu_0^N \in L^2 (\T^2)^{\otimes N}$ and each component $v^{N, j}$ of the solution $\vv^N$ to the perturbed \PLS \eqref{NLHNv} exists in the space
\begin{align*}
C ([0, T]; L^2 (\T^2)) \cap L^4 ([0, T]; L^4 (\T^2)) \cap L^2 ([0, T]; H^1 (\T^2))
\end{align*}

\noi
for any $T > 0$. In this paper, we need higher regularity of the solution to the perturbed \PLS \eqref{NLHNv}, and so we present the additional treatment in Subsection~\ref{SUB:SNLHsys}.

Part (iii) of Theorem~\ref{THM:conv1}, the convergence from the \PLS to the mean-field SNLH, has already been established in \cite[Theorem~4.1]{SSZZ}.\footnote{The convergence result in \cite[Theorem~4.1]{SSZZ} is based on the assumptions in \cite[Assumption~4.1]{SSZZ}, which involve an $L^2 (\Om)$-convergence of the initial data. However, we note that in establishing convergence of the solutions in probability, the authors did not use the assumption on the $L^2 (\Om)$-convergence of the initial data.} In fact, by using the first order expansions \eqref{expN} and \eqref{exp}, the authors in \cite{SSZZ} showed that each component $v^{N, j}$ of the solution $\vv^N$ to the perturbed \PLS \eqref{NLHNv} converges in probability to $v^j$ satisfying the perturbed mean-field SNLH \eqref{NLHv} in the space $C([0, T]; L^2 (\T^2))$ for any $T > 0$. Moreover, under an additional assumption that the initial data are exchangeable (i.e.~all permutations of the sequence of initial data have the same joint probability distribution), the authors in \cite{SSZZ} established a convergence result in the $L^2 (\Omega)$ sense. We refer the readers to \cite[Theorem~4.1]{SSZZ} for more details.

Part (ii) of Theorem~\ref{THM:conv1} is precisely the Smoluchowski-Kramers approximation for the hyperbolic $O(N)$ linear sigma model. When $N = 1$, this corresponds to the Smoluchowski-Kramers approximation for the cubic SdNLW studied in \cite{Zine}. We remark that the convergence result obtained in \cite{Zine} holds locally in time, whereas in this paper, we establish global-in-time convergence of coupled SdNLW systems by using a continuity argument (see Remark~\ref{RMK:Zine} below for a further comment). See Section~\ref{SEC:regI} for details.

\begin{remark} \rm
\label{RMK:Zine}
In \cite[Remark~1.8]{Zine}, the author claimed that the Smoluchowski-Kramers approximation for the cubic SdNLW holds globally in time as long as one can show global well-posedness for the cubic SdNLW indexed by the second-order (mass) term $\eps$ (i.e. \eqref{NLWNeu} with $N = 1$) for any $0 < \eps \leq 1$.
Such global well-posedness can be established using the globalization argument in \cite{GKOT}.
However, the approximation does not immediately follow due to the lack of a uniform-in-$\eps$ global-in-time bound for the solutions. Indeed, we are not able to establish such uniform-in-$\eps$ bound even with the tools provided in \cite{GKOT}; see the comments before Lemma~\ref{LEM:B4} below for the subtlety.
\end{remark}

We now look at regime II. For the equation \eqref{NLWNev}, we fix $0 < \eps \leq 1$ and take $N \to \infty$ to formally obtain the following equation:
\begin{align}
(\eps^2 \dt^2 + \dt + 1 - \Dl) v_\eps^j = - 2 \E [\Psi_\eps^j v_\eps^j] \Psi_\eps^j - 2 \E [\Psi_\eps^j v_\eps^j] v_\eps^j - \E [(v_\eps^j)^2] \Psi_\eps^j - \E [(v_\eps^j)^2] v_\eps^j .
\label{NLWev}
\end{align}

\noi
By writing
\begin{align}
    u_\eps^j = \Psi_\eps^j + v_\eps^j,
\label{expe}
\end{align}

\noi
we can write out the following renormalized version of the mean-field \NLWe \eqref{NLWe0}:
\begin{align}
    (\eps^2 \dt^2 + \dt + 1 - \Dl) u_\eps^j = - \E [ (u_\eps^j)^2 - (\Psi_\eps^j)^2 ] u_\eps^j + \sqrt{2} \xi^j .
\label{NLWeu}
\end{align}

\noi
Again, we say that $u_\eps^j$ is a solution to the renormalized mean-field \NLWe \eqref{NLWeu} if $u_\eps^j$ is given by \eqref{expe} with $v_\eps^j$ being a solution to the perturbed mean-field \NLWe \eqref{NLWev}.

We are now ready to state our convergence result via regime II.

\begin{theorem}[Convergence via regime II]
\label{THM:conv2}
Let $\frac 45 < s < 1$. Given $N \in \N$ and $j \in \N$, let $(\uu_0^N, \uu_1^N)$ be random and belong to $\H^s (\T^2)^{\otimes N}$ almost surely and let $(u_0^j, u_1^j) \in L^2 ( \Omega; \H^s (\T^2))$. Assume that $\{ ( u_0^j, u_1^j ) \}_{j \in \N}$ are independent and identically distributed. Also, we assume that the following convergences hold in probability:
\begin{align*}
    &\text{For each $j \in \N$, } \big\| ( u_0^{N, j} - u_0^j, u_1^{N, j} - u_1^j ) \big\|_{\H^s} \longrightarrow 0 \quad \text{as } N \to \infty; \\
    &\big\| ( u_0^{N, j} - u_0^j, u_1^{N, j} - u_1^j ) \big\|_{\A_{N, j} \H^s} \longrightarrow 0 \quad \text{as } N \to \infty.
\end{align*}

\smallskip \noi
\textup{(i)~(Global well-posedness for the mean-field $\textup{SdNLW}_\eps$)} Fix $j \in \N$. For any $0 < \eps \leq 1$ and $T > 0$, there exists a unique solution $(u_\eps^j, \dt u_\eps^j)$ to the renormalized mean-field \NLWe \eqref{NLWeu} with initial data $(u_0^j, u_1^j)$ in the space
\begin{align*}
(\Psi_\eps^j, \dt \Psi_\eps^j) + L^2 (\Om; C([0, T]; \H^s (\T^2))) .
\end{align*}

\smallskip \noi
\textup{(ii)~(Convergence from \HLSe to the mean-field $\textup{SdNLW}_\eps$)} Let $0 < \eps \leq 1$ be fixed. Let $\uu_\eps^N = (u_\eps^{N, j})_{1 \leq j \leq N}$ be the solution to the renormalized \HLSe \eqref{NLWNeu} with initial data $(\uu_0^N, \uu_1^N)$ provided by Proposition~\ref{PROP:GWP_NLW}. Fix $j \in \N$ and let $u_\eps^j$ be given by part \textup{(i)}. Then, $\{ u_\eps^{N, j} \}_{N \in \N}$ converges in probability to $u_\eps^j$ in $C(\R_+; H^{- \sigma} (\T^2))$, endowed with the compact open topology in time, as $N \to \infty$ for any $\s > 0$.

\smallskip \noi
\textup{(iii)~(Convergence from the mean-field \NLWe to the mean-field SNLH)} Fix $j \in \N$. Let $u_\eps^j$ be given by part \textup{(i)}. Assuming further that $u_0^j \in L^4 (\Om; H^s (\T^2))$, we let $u^j$ be the solution to the renormalized mean-field SNLH \eqref{NLHu} with initial data $u_0^j$ provided by Proposition~\ref{PROP:GWP_NLH}. Then, for any $T > 0$, $\{ u_\eps^j \}_{\eps \in (0, 1]}$ converges to $u^j$ in $L^2 (\Om; C([0, T]; H^{- \sigma} (\T^2)))$ as $\eps \to 0$ for any $\sigma > 0$.
\end{theorem}

The proof of part (i) of Theorem~\ref{THM:conv2}, global well-posedness for the mean-field $\text{SdNLW}_\eps$ \eqref{NLWeu}, follows from minor modifications of the proof for \cite[Theorem~1.2]{LLO}. Similar to Proposition~\ref{PROP:GWP_NLW} on pathwise global well-posedness for $\text{HLSM}_{\eps, N}$, we can first construct a unique local-in-time solution to the perturbed mean-field $\text{SdNLW}_\eps$ \eqref{NLWev} in $L^2 (\Om; C([0, T]; \H^s (\T^2)))$ for some $T > 0$ by using the one degree of smoothing from the linear damped wave dynamics from \eqref{defIe}. Then, we can prove an a priori bound for the solution and extend it to an arbitrarily large time interval by studying the following energy functional:
\begin{align}
E_\eps (w, \dt w) \deff \frac 12 \E \bigg[ \int_{\T^2} \big( (\eps \dt w)^2 + w^2 + |\nb w|^2 \big) dx \bigg] + \frac 14 \int_{\T^2} \big( \E [w^2] \big)^2 dx ,
\label{Eew}
\end{align}

\noi
which is conserved under the flow of the mean-field NLW:
\begin{align*}
(\eps^2 \dt^2 + 1 - \Dl) w = - \E [w^2] w .
\end{align*}

\noi
Once again, we need to use the hybrid approach from \cite{GKOT} combining the $I$-method \cite{CKSTT02} and a Gronwall-type argument \cite{BTz14}. 
While one can refer to \cite[Theorem~1.2]{LLO} for all the above steps that lead to global well-posedness for the mean-field $\text{SdNLW}_\eps$ \eqref{NLWeu}, we establish a stronger a priori bound for the solution to the perturbed mean-field $\text{SdNLW}_\eps$ \eqref{NLWev} in this paper. The main novelty here is that we are able to obtain an a priori bound uniformly in $0 < \eps \leq 1$. Such a uniform a priori bound plays a central role in proving convergence of the dynamics in part (iii) of Theorem~\ref{THM:conv2} to be mentioned below.

Part (ii) of Theorem~\ref{THM:conv2}, the mean-field convergence of $\text{HLSM}_{\eps, N}$, has essentially been established in \cite[Theorem~1.6]{LLO}.
To prove the mean-field convergence, we can establish law-of-large-numbers type lemmas using the second moment bound on the solution $v_\eps^j$ to the perturbed mean-field \NLWe \eqref{NLWev}. Since $\eps$ is fixed, we do not need to worry about dependence on $\eps$, and so we can proceed in exactly the same manner as the proof of \cite[Theorem~1.6]{LLO}. We omit details.

Part (iii) of Theorem~\ref{THM:conv2} corresponds to the Smoluchowski-Kramers approximation for the mean-field SdNLW. In contrast to the pathwise setting in part (ii) of Theorem~\ref{THM:conv1}, the solution to the mean-field equations are in the $L^2 (\Om)$ sense. This means that the continuity argument based on a pathwise solution theory used for proving the convergence via regime I is no longer suitable in proving part (iii) of Theorem~\ref{THM:conv2}. Nevertheless, as mentioned above, we can establish a uniform-in-$\eps$ bound for the solution to the perturbed mean-field \NLWe \eqref{NLWev}; see Proposition~\ref{PROP:NLW}. This uniform bound allows us to prove global-in-time convergence of the dynamics in a direct manner.  
Let us remark that in establishing the uniform-in-$\eps$ bound, we need to rely heavily on the fact that we are working with the mean-field nonlinearity. This is relevant in Lemma~\ref{LEM:B4} below. One can still prove a bound as in Lemma~\ref{LEM:B4} with the usual cubic nonlinearity but with more loss, which seems to be unacceptable in applying the Gronwall-type argument. 
See Section~\ref{SEC:regII} for more details. 
This also explains why we use two entirely different approaches in proving the two Smoluchowski-Kramers approximation results in Theorem~\ref{THM:conv1}~(ii) and Theorem~\ref{THM:conv2}~(iii).

\medskip
Before closing the introduction, we would like to remark that the coefficient $\sqrt{2}$ in front of the space-time white noises in the models \eqref{NLWNe0}, \eqref{NLH0}, \eqref{NLHN0}, and \eqref{NLWe0} does not play any role in the analysis, and one can replace it by any other constants. However, if one consider Gibbs measures and invariant Gibbs dynamics, then it is crucial to have the coefficient $\sqrt{2}$. 
We include a discussion on  invariant Gibbs dynamics in Appendix~\ref{APP:inv}, where we show the commutative diagram in Figure~\ref{fig:diagram} at Gibbs equilibrium in Theorem~\ref{THM:Gibbs}.
The proof follows essentially from same steps for proving Theorem~\ref{THM:conv1} and Theorem~\ref{THM:conv2} with only a few additional ingredients.

We end the introduction by stating several remarks.

\begin{remark} \rm
In this paper, we mainly work with the four equations \eqref{NLWNe0}, \eqref{NLH0}, \eqref{NLHN0}, and \eqref{NLWe0} in the defocusing case, i.e.~with a minus sign in front of the nonlinearity. In the focusing case, i.e.~with a plus sign in front of the nonlinearity, all the above global well-posedness results and global-in-time convergence results break down. The main issue comes from the fact that in the focusing setting, the potential energy terms (the quartic terms) in the energy functionals $E_\eps^N$ in \eqref{defEeN} and $E_\eps$ in \eqref{Eew} come with the minus signs. This means that the energy functionals cannot control the $H^1$-norm of the solution unless the initial data is sufficiently small. For general initial data, we expect that some finite-time blowup phenomenon may occur in this case. 

Nevertheless, all the well-posedness results and the convergence results, particularly the commutative diagram in Figure~\ref{fig:diagram}, hold locally in time in the focusing case, as we do not need to rely on the conservation of energy. The proof of the convergence results follow from minor modifications of local well-posedness arguments. See, for example, \cite[Theorem~1.5]{Zine} and \cite[Remark~1.12~(i)]{LLO}.
\end{remark}

\begin{remark} \rm
In this paper, we only consider the Smoluchowski-Kramers approximations and the mean-field limits on the two-dimensional torus $\T^2$. It would be of interest to investigate the convergence problems on the full space $\R^2$.
Indeed, the mean-field limit for the parabolic $O(N)$ linear sigma model \eqref{NLHN0} on $\R^2$ was proposed by Shen in \cite[Section~9]{Shen}, and the one for the hyperbolic $O(N)$ linear sigma model \eqref{NLWNe0} was proposed in \cite[Remark~1.12~(ii)]{LLO}. Moreover, the Smoluchowski-Kramers approximations for singular SPDEs on $\R^2$ are also open. The main difficulty in the infinite volume setting is the unboundedness (in $x$) of stochastic objects, even after frequency truncations. In the heat case, one needs the weighted function spaces as in \cite{MW17, GH19}; in the wave case, one may use the finite speed of propagation as in \cite{Tol, OTWZ}. However, as the Smoluchowski-Kramers approximations involve properties from both the heat and the wave equations, some new ingredients may be required.
\end{remark}

\begin{remark} \rm
It is known from \cite{ORTz} that the SdNLW \eqref{NLWNe0} with $N = 1$ and $\eps = 1$ is almost surely globally well-posed with respect to the Gibbsian initial data on a two-dimensional compact manifold. It would be of interest to study the Smoluchowski-Kramers approximations and the mean-field limits also on a compact manifold.
\end{remark}

\section{Function spaces and preliminary lemmas}

In this section, we introduce notations and recall some useful tools that will be applied in this paper. 

We denote by $C$ a constant that may vary from line to line. We also write $C (\cdot)$ to emphasize dependence of this constant on external parameters, such as $\om$, $N$, and $T$. For two positive quantities $a$ and $b$, we write $a \les b$ if $a \leq Cb$ for some constant $C > 0$ that is independent of the set where $a$ and $b$ are allowed to vary. We write $a \sim b$ if $a \les b$ and $b \les a$. We also use subscripts such as ``$\les_{s, t}$'' to denote dependence on external parameters. 

For any function $f$ on $\T^2$, we denote by $\ft f$ or $\F (f)$ the Fourier transform of $f$. For any function $g$ on $\Z^2$, we denote by $g^\vee$ the inverse Fourier transform of $g$. We will frequently use the Japanese bracket $\jb{\cdot} = (1 + |\cdot|^2)^{\frac 12}$.

For function spaces, we often use shorthand notations, such as $C_I H^s_x = C (I; H^s (\T^2))$ for some time interval $I \subset \R$ and $s \in \R$. If $I = [0, T]$ for some $T > 0$, we write $C_T H_x^s = C_{[0, T]} H_x^s$. Given a Banach space $X$ and $N \in \N$, we define the $\A_N X$-norm of $\mathbf{f} = (f^{j})_{1 \leq j \leq N} \in X^{\otimes N}$ as in \eqref{defANX}, namely the $\ell^2$-average of the $X$-norms of $f^{j}$'s. Similarly, we define the $\A_N^{(2)} X$-norm of $\mathbf{g} = (g^{j, k})_{1 \leq j, k \leq N} \in X^{\otimes N^2}$ by
\begin{align}
    \| \mathbf{g} \|_{\A_N^{(2)} X} = \big\| (g^{j, k})_{1 \leq j, k \leq N} \big\|_{\A_{N, j, k}^{(2)} X} \deff 
    \big\| (g^{j, k})_{1 \leq j, k \leq N} \big\|_{\A_{N, j} \A_{N, k} X} = \bigg( \frac{1}{N^2} \sum_{j, k = 1}^N \| g^{j, k} \|_X^2 \bigg)^{\frac 12} .
\label{defAN2X}
\end{align}

\subsection{Sobolev and Besov spaces}
\label{SUB:sob}
Given $s \in \R$ and $1 \leq p \leq \infty$, we denote by $W^{s, p} (\T^2)$  the $L^p$-based Sobolev space via the norm
\begin{align*}
    \| f \|_{W^{s, p}} \deff \| \jb{\nb}^s f \|_{L^p} = \| ( \jb{\cdot}^s \ft f )^\vee \|_{L^p} .
\end{align*}

\noi
When $p = 2$, we write $H^s (\T^2) = W^{s, 2} (\T^2)$. By Plancherel's identity, we have
\begin{align*}
    \| f \|_{H^s} = \| \jb{\nb}^s f \|_{L^2} = \| \jb{\cdot}^s \ft f (\cdot) \|_{\ell^2} .
\end{align*}

\noi
We also define the product space $\H^s (\T^2) \deff H^s (\T^2) \times H^{s - 1} (\T^2)$ via the norm
\begin{align*}
    \| (f, g) \|_{\H^s} \deff \| f \|_{H^s} + \| g \|_{H^{s - 1}} .
\end{align*}

\noi
For any $s_1 \leq s_2$ and $1 \leq p_1 \leq p_2 \leq \infty$, we have the embedding
\begin{align*}
    \| f \|_{W^{s_1, p_1}} \les \| f \|_{W^{s_2, p_2}} ,
\end{align*}

\noi
which we will use frequently without mentioning.

We record the following product estimates for Sobolev spaces. See \cite[Lemma~3.4]{GKO} and also \cite[Proposition~1 and Proposition~2]{BOZ} for the endpoint case.
\begin{lemma}
\label{LEM:prod}
Let $s > 0$.

\smallskip \noi
\textup{(i)} Let $1 < r, p_1, p_2, q_1, q_2 \leq \infty$ satisfying $\frac 1r = \frac{1}{p_1} + \frac{1}{q_1} = \frac{1}{p_2} + \frac{1}{q_2}$. Then, we have
\begin{align*}
    \| f g \|_{W^{s, r}} \les \| f \|_{W^{s, p_1}} \| g \|_{L^{q_1}} + \| f \|_{L^{p_2}} \| g \|_{W^{s, q_2}} .
\end{align*}

\smallskip \noi
\textup{(ii)} Let $1 < q, r < \infty$ and $1 < p \leq \infty$ satisfying $\frac{s}{2} \geq \frac{1}{p} + \frac{1}{q} - \frac{1}{r}$ and $q, r' \geq p'$, with $r'$ and $p'$ being the H\"older conjugates of $r$ and $p$, respectively. Then, we have
\begin{align*}
    \| f g \|_{W^{-s, r}} \les \| f \|_{W^{-s, p}} \| g \|_{W^{s, q}} .
\end{align*}
\end{lemma}

We also need the Besov-H\"older spaces. Let $\varphi : \R \to [0, 1]$ be a smooth even cutoff function supported on $[- \frac 85, \frac 85]$ such that $\varphi \equiv 1$ on $[- \frac 54, \frac 54]$.
Given $\xi \in \R^2$, we set $\varphi_1 (\xi) = \varphi (|\xi|)$ and, given a dyadic number $N \geq 2$, 
\begin{align*}
\varphi_N (\xi) = \varphi (\tfrac{|\xi|}{N}) - \varphi (\tfrac{2 |\xi|}{N}).
\end{align*}

\noi
Given dyadic $N \geq 1$, we define the Littlewood-Paley projector $\mathbf{P}_N$ as the Fourier multiplier operator with symbol $\varphi_N$.
Given $s \in \R$, we denote by $\C^s (\T^2)$ the Besov-H\"older space via the norm
\begin{align*}
\| f \|_{\C^s} \deff \sup_{\substack{N \geq 1 \\ \text{dyadic}}} N^s \| \mathbf{P}_N f \|_{L^\infty} .
\end{align*}


We have the following embedding results between the Besov-H\"older spaces and the $L^p$-based Sobolev spaces. For a proof, see \cite[Lemma~2.1]{GOTW}.
\begin{lemma}
\label{LEM:emb}
\textup{(i)} Let $s_1, s_2 \in \R$ be such that $s_1 < s_2$ and $1 \leq p \leq \infty$. Then, we have
\begin{align*}
\| f \|_{W^{s_1, p}} \les \| f \|_{\C^{s_2}} .
\end{align*}

\smallskip \noi
\textup{(ii)} Let $s_1, s_2 \in \R$ be such that $s_2 \geq s_1 + 1$. Then, we have
\begin{align*}
\| f \|_{\C^{s_1}} \les \| f \|_{H^{s_2}} .
\end{align*}

\smallskip \noi
\textup{(iii)} Let $s \in \R$. Then, we have
\begin{align*}
\| f \|_{\C^{s}} \les \| f \|_{W^{s, \infty}} .
\end{align*}
\end{lemma}

We also record the following product estimates for Besov-H\"older spaces. For a proof, see, for example, \cite[Lemma~2.1]{GOTW}.
\begin{lemma}
\label{LEM:prodC}
\textup{(i)} Let $s > 0$. Then, we have
\begin{align*}
\| f g \|_{\C^s} \les \| f \|_{\C^s} \| g \|_{\C^s} .
\end{align*}

\smallskip \noi
\textup{(ii)} Let $s_1, s_2 \in \R$ be such that $s_1 < 0 < s_2$ and $s_1 + s_2 > 0$. Then, we have
\begin{align*}
\| f g \|_{\C^{s_1}} \les \| f \|_{\C^{s_1}} \| g \|_{\C^{s_2}} .
\end{align*}
\end{lemma}

We now record the following estimate on the smoothing effect of the heat semigroup.  The proof follows from minor modifications of \cite[Lemma~A.5]{GIP}.
\begin{lemma}
\label{LEM:heat}
Let $s_1, s_2 \in \R$ be such that $s_1 \geq s_2$ and $t > 0$. Then, we have
\begin{align*}
\| e^{- t (1 - \Dl)} f \|_{H^{s_1}} \les t^{- \frac{s_1 - s_2}{2}} \| f \|_{H^{s_2}} 
\end{align*}

\noi
and 
\begin{align*}
\| e^{- t (1 - \Dl)} f \|_{\C^{s_1}} \les t^{- \frac{s_1 - s_2}{2}} \| f \|_{\C^{s_2}} .
\end{align*}
\end{lemma}

\subsection{Convergence of linear propagators}

In this subsection, we recall some known results on the convergence of the linear propagators of the damped wave equation and the heat equation. We follow the notations in \cite{Zine}.

Given $0 < \eps \leq 1$, we define $\mathbf{P}_\eps^{\text{lo}}$ and $\mathbf{P}_\eps^{\text{hi}}$ as the sharp frequency projections onto $\{ n \in \Z^2: \jb{n} \leq (2 \eps)^{-1} \}$ and $\{ n \in \Z^2 : \jb{n} > (2 \eps)^{-1} \}$, respectively. 
For any $n \in \Z^2$, we also define
\begin{align*}
    \ld_\eps (n) \deff \frac{\sqrt{1 - 4 \eps^2 \jb{n}^2}}{2 \eps^2} \quad \text{and} \quad \zeta_\eps (n) \deff \frac{\sqrt{4 \eps^2 \jb{n}^2 - 1}}{2 \eps^2} .
\end{align*}

\noi
We then define
\begin{align}
    \mathcal{D}_\eps (t) \deff e^{- \frac{t}{2 \eps^2}} \frac{\sinh ( t \ld_\eps (\nb))}{\ld_\eps (\nb)} \mathbf{P}_\eps^{\text{lo}} 
    + e^{- \frac{t}{2 \eps^2}} \frac{\sin (t \zeta_\eps (\nb))}{\zeta_\eps (\nb)} \mathbf{P}_\eps^{\text{hi}} ,
\label{defDe}
\end{align}

\noi
whose associated Fourier multiplier is given by
\begin{align}
\ft{\mathcal{D}}_\eps (t, n) \deff e^{- \frac{t}{2 \eps^2}} \frac{\sinh ( t \ld_\eps (n))}{\ld_\eps (n)} \ind_{\{ \jb{n} \leq (2 \eps)^{-1} \}} + e^{- \frac{t}{2 \eps^2}} \frac{\sin (t \zeta_\eps (n))}{\zeta_\eps (n)} \ind_{\{ \jb{n} > (2 \eps)^{-1} \}} .
\label{Den}
\end{align}

Given $0 < \eps \leq 1$, the solution to the inhomogeneous linear damped wave equation
\begin{align*}
\begin{cases}
    (\eps^2 \dt^2 + \dt + 1 - \Dl) u = F \\
    (u, \dt u)|_{t = 0} = (\phi_0, \phi_1)
\end{cases}
\end{align*}

\noi
is given by
\begin{align*}
    u (t) = P_\eps (t) (\phi_0, \phi_1) + \mathcal{I}_\eps (F) (t) ,
\end{align*}

\noi
where
\begin{align}
    P_\eps (t) (\phi_0, \phi_1) \deff (\eps^{-2} + \dt) \D_\eps (t) \phi_0 + \D_\eps (t) \phi_1
\label{defPe}
\end{align}

\noi
and
\begin{align}
    \mathcal{I}_{\eps} (F) (t) \deff \int_{0}^t \eps^{-2} \D_\eps (t - t') F (t') dt' .  
\label{defIe}
\end{align}

\noi
We also consider the linear heat equation
\begin{align*}
\begin{cases}
    (\dt + 1 - \Dl) u = F \\
    u|_{t = 0} = \phi_0 ,
\end{cases}
\end{align*}

\noi
whose solution is given by
\begin{align*}
    u(t) = P_0 (t) \phi_0 + \mathcal{I}_0 (F) (t) ,
\end{align*}

\noi
where
\begin{align}
    P_0 (t) \phi_0 \deff e^{- t (1 - \Dl)} \phi_0
\label{defP0}
\end{align}

\noi
and
\begin{align}
    \mathcal{I}_{0} (F) (t) \deff \int_{0}^t P_0 (t - t') F (t') dt' .
\label{defI0}
\end{align}

\noi
For later convenience, we alse write
\begin{align}
    P_0 (t) (\phi_0, \phi_1) = P_0 (t) \phi_0 .
\label{P001}
\end{align}

The following lemma was proved in \cite[Proposition~3.1 and Lemma~3.7~(iii)]{Zine}.
\begin{lemma}
\label{LEM:PIconv}
Let $s \in \R$, $0 \leq \eps \leq 1$, and $T > 0$. Then, we have the following bounds:
\begin{align}
    &\| P_\eps (\cdot) (\phi_0, \phi_1) \|_{C_T H_x^s} \les \| (\phi_0, \eps \phi_1) \|_{\H^s} , \label{Pe1} \\
    &\| \mathcal{I}_{\eps} (F) (\cdot) \|_{C_T H_x^s} \les T^{\frac 12} \| F \|_{L_T^\infty H_x^{s - 1}} , \nonumber
\end{align}

\noi
where the underlying constants are independent of $\eps$.
Moreover, for any $0 < \ta \ll 1$ sufficiently small, we have the difference estimates:
\begin{align}
    &\| (P_\eps - P_0) (\cdot) (\phi_0, \phi_1) \|_{C_T H_x^s} \les \eps^{\ta} \| (\phi_0, \eps \phi_1) \|_{\H^{s + 2 \ta}}, \label{Pe2} \\
    &\| (\mathcal{I}_{\eps} - \mathcal{I}_{0}) (F) (\cdot) \|_{C_T H_x^s} \les T^{\frac 12} \eps^\ta \| F \|_{L_T^\infty H_x^{s - 1 + 2 \ta}} . \nonumber
\end{align}
\end{lemma}

\begin{remark} \rm
Compared to \cite[Proposition~3.1]{Zine}, we have an extra $\eps$ factor on the right-hand side of the estimates \eqref{Pe1} and \eqref{Pe2}. This is obtained in \cite[Lemma~3.7~(iii)]{Zine} and plays a crucial role in establishing a uniform-in-$\eps$ global a priori bound for the mean-field \NLWe in Subsection~\ref{SUB:mfbdd} below.
\end{remark}

\subsection{$I$-method and $I$-operator}

As mentioned in the introduction, we will need the $I$-method to establish a global-in-time a priori bound for the solution of the mean-field $\text{SdNLW}_{\eps}$.

Let $0 < s < 1$. Given $M \in \N$, we define a smooth, radial, and non-increasing (in radial direction) function $m_{s, M} \in C^\infty (\R^2; [0, 1])$ satisfying
\begin{align*}
    m_{s, M} (\xi) =
    \begin{cases}
        1 & \text{if } |\xi| \leq M \\
        (\frac{M}{|\xi|})^{1 - s} & \text{if } |\xi| \geq 2 M .
    \end{cases}
\end{align*}

\noi
We then define $I_M$ (omitting the dependence on $s$) as the Fourier multiplier operator with symbol $m_{s, M}$:
\begin{align}
    \ft{I_M f} (n) = m_{s, M} (n) \ft f (n)
\label{defIop}
\end{align}

\noi
for any $n \in \Z^2$. Note that we can easily deduce from the definition that for any $\gamma \in \R$,
\begin{align}
    \| f \|_{H^{\gamma}} \les \| I_M f \|_{H^{\gamma + 1 - s}} \les M^{1 - s} \| f \|_{H^\gamma} . 
\label{Ibdd1}
\end{align}

\noi
Also, by using the Littlewood-Paley theorem, for any $s_0 \in \R$, $0 \leq s_1 \leq 1 - s$, and $1 < p < \infty$, we have
\begin{align}
    \| I_M f \|_{W^{s_0 + s_1, p}} \les M^{s_1} \| f \|_{W^{s_0, p}} . 
\label{Ibdd2}
\end{align}


Let us record some useful estimates for the $I$-operator. The following lemma was proved in \cite[Lemma~3.1]{GKOT}.
\begin{lemma}
\label{LEM:I_est1}
Let $\frac 23 \leq s < 1$ and $M \in \N$. Then, we have
\begin{align*}
    \big\| I_M (f^2 g) - (I_M f)^2 I_M g \big\|_{L^2} \les M^{- 3 s + 2} \| I_M f \|_{H^1}^2 \| I_M g \|_{H^1} 
\end{align*}

\noi
with the underlying constant independent of $M$.
\end{lemma}

We also have the following lemma, whose proof follows from \cite[Lemma~3.3]{GKOT}.
\begin{lemma}
\label{LEM:I_est2}
Let $\frac 23 \leq s < 1$ and $M \in \N$ with $M \geq 10$. Then, given $0 < \nu \leq 1 - s$, there exist small $\s_0 = \s_0 (\nu) > 0$ and large $p_0 = p_0 (\nu) > 1$ such that
\begin{align}
    \big\| I_M (f g h) - I_M f \, I_M g \, I_M h \big\|_{L^2} \les M^{- s + \frac 12 + \nu} \| I_M f \|_{H^1} \| I_M g \|_{H^1} \| h \|_{W^{- \s_0, p_0}} ,
\label{I_est3}
\end{align}

\noi
where the underlying constant is independent of $M$.
\end{lemma}

\subsection{On the stochastic convolution}
\label{SEC:sto}

In this subsection, we consider the stochastic convolutions $\Psi_\eps^j$ and $\Psi_0^j$ given by \eqref{Psije} and \eqref{Psij0}, respectively.

We recall that $\{\xi^j\}_{j \in \N}$ is a family of independent space-time white noises on $\R_+ \times \T^2$. From \eqref{Psije}, given any $0 < \eps \leq 1$, we may write
\begin{align*}
    \Psi_\eps^j (t) = \sqrt{2} \int_0^t \eps^{-2} \D_\eps (t - t') d W^j (t') ,
\end{align*}

\noi
where $\D_\eps (t)$ is defined in \eqref{defDe} and
\begin{align}
    W^j (t) \deff \frac{1}{2 \pi} \sum_{n \in \Z^2} B_n^j (t) e^{i n \cdot x}
\label{Wj}
\end{align}

\noi
is a cylindrical Wiener process on $L^2 (\T^2)$ with $B_n^j (t) = (2 \pi)^{-1} \langle \xi^j, \ind_{[0, t]} (t') e^{i n \cdot x} \rangle_{t', x}$, where $\langle \cdot, \cdot \rangle_{t, x}$ denotes the duality pairing on $\R_+ \times \T^2$. Note that the definition implies that $B_0^j$ is a standard real-valued Brownian motion and $\{ B_n^j \}_{n \in \Z^2 \setminus \{0\}}$ is a family of independent standard complex-valued Brownian motions conditioned such that $B_{- n}^j = \cj{B_n^j}$ for any $n \in \Z^2 \setminus \{0\}$. When $\eps = 0$, from \eqref{Psij0}, we write
\begin{align*}
    \Psi_0^j (t) = \sqrt{2} \int_0^t P_0 (t - t') d W^j (t') ,
\end{align*}

\noi
where $P_0$ is defined in \eqref{defP0}.

Given $M \in \N$, let $P_{\leq M}$ be the sharp frequency projection onto $\{ n \in \Z^2 : |n| \leq M \}$. Given $M \in \N$, $j \in \N$, and $0 \leq \eps \leq 1$, we define $\Psi_{\eps, M}^j \deff P_{\leq M} \Psi_\eps^j$. Note that for each $n \in \Z^2$, we have from \cite[(4.3)]{Zine} that
\begin{align}
    \E \Big[ | \ft{\Psi_\eps^j} (t, n) |^2 \Big] \les \jb{n}^{-2} 
\label{Psin}
\end{align}

\noi
with the underlying constant independent of $\eps$ and $t$. Also, a direct computation yields that when $0 < \eps \leq 1$, we have
\begin{align*}
\E \Big[ | \ft{\Psi_\eps^j} (t, n) |^2 \Big] 
&\geq \ind_{\{ \jb{n} > (2 \eps)^{-1} \}} \Big( \frac{1}{\jb{n}^2} - e^{- \frac{t}{\eps^2}} \frac{4 \eps^2}{4 \eps^2 \jb{n}^2 - 1} \\
&\quad - e^{- \frac{t}{\eps^2}} \frac{\sin (2 t \zeta_\eps (n))}{\jb{n}^2 \sqrt{4 \eps^2 \jb{n}^2 - 1}} + e^{- \frac{t}{\eps^2}} \frac{\cos (2 t \zeta_\eps (n))}{\jb{n}^2 (4 \eps^2 \jb{n}^2 - 1)} \Big) \\
&\ges_{\eps, t} \ind_{\{ \jb{n} \geq \eps^{-1} \}} \frac{1}{\jb{n}^2} 
\end{align*}

\noi
at least for $t \ges \eps^{2}$, and when $\eps = 0$, we have
\begin{align*}
\E \Big[ |\ft{\Psi_0^j} (t, n)|^2 \Big] = \frac{1 - e^{-2 t \jb{n}^2}}{\jb{n}^2} \ges_t \frac{1}{\jb{n}^2}
\end{align*}

\noi
for $t > 0$.
Thus, we have
\begin{align}
    \s_{\eps, M} (t) &\deff \E \big[ |\Psi_{\eps, M}^j (t) |^2 \big] \sim_{\eps, t} \log M ,
\label{sigmaM}
\end{align}

\noi
which diverges as $M \to \infty$. Given $k, j \in \N$, we define the Wick products
\begin{align}
\begin{split}
    \wick{ \Psi_{\eps, M}^k \Psi_{\eps, M}^j } &\deff
    \begin{cases}
        H_2 ( \Psi_{\eps, M}^j ; \s_{\eps, M} ) & \text{if } k = j \\
        \Psi_{\eps, M}^k \Psi_{\eps, M}^j & \text{if } k \neq j ,
    \end{cases} \\
    \wick{ (\Psi_{\eps, M}^k)^2 \Psi_{\eps, M}^j } &\deff
    \begin{cases}
        H_3 ( \Psi_{\eps, M}^j ; \s_{\eps, M} ) & \text{if } k = j \\
        H_2 ( \Psi_{\eps, M}^k ; \s_{\eps, M} ) \Psi_{\eps, M}^j & \text{if } k \neq j ,
    \end{cases}
\end{split}
\label{wick0}
\end{align}

\noi
where $H_2 (\cdot ; \s)$ and $H_3 (\cdot; \s)$ denote Hermite polynomials of degree 2 and 3, respectively, with variance parameter $\s > 0$. Then, the Wick products appearing earlier in \eqref{NLWNev} are defined by
\begin{align}
\begin{split}
    \wick{ \Psi_{\eps}^k \Psi_{\eps}^j } &\deff \lim_{M \to \infty} \wick{ \Psi_{\eps, M}^k \Psi_{\eps, M}^j } , \\
    \wick{ (\Psi_{\eps}^k)^2 \Psi_{\eps}^j } &\deff \lim_{M \to \infty} \wick{ (\Psi_{\eps, M}^k)^2 \Psi_{\eps, M}^j } .
\end{split}
\label{wick}
\end{align}

\noi
We will see in Lemma~\ref{LEM:sto} below that the limits in \eqref{wick} exist almost surely in the space $C(\R_+; W^{- \s, \infty})$ for any $\s > 0$.




\medskip
We now state and prove the following lemma regarding the regularity properties of the above stochastic objects.

\begin{lemma}
\label{LEM:sto}
Let $T \geq 1$ and $k, j \in \N$. Given $0 \leq \eps \leq 1$ and $M \in \N$, we consider the following stochastic terms $Z_{\eps, M}$ and $Z_\eps$ and their degree $r \in \N$:
\begin{itemize}
    \item[$\bullet$] $Z_{\eps, M} = \Psi_{\eps, M}^j$, $Z_\eps = \Psi^j_\eps$, $r = 1$;
    \item[$\bullet$] $Z_{\eps, M} = \,\wick{\Psi_{\eps, M}^k \Psi_{\eps, M}^j}$, $Z_\eps = \, \wick{ \Psi_\eps^k \Psi_\eps^j }$, $r = 2$;
    \item[$\bullet$] $Z_{\eps, M} = \, \wick{(\Psi_{\eps, M}^k)^2 \Psi_{\eps, M}^j}$, $Z_\eps = \, \wick{(\Psi_\eps^k)^2 \Psi_\eps^j}$, $r = 3$.
\end{itemize}

\noi
Then, the following properties hold.

\smallskip \noi
\textup{(i)} 
For any $0 \leq \eps \leq 1$ and $\sigma > 0$, $Z_{\eps, M}$ converges almost surely to some limit denoted by $Z_\eps$ in $C([0, T]; W^{- \sigma, \infty} (\T^2))$ as $M \to \infty$.

\smallskip \noi
\textup{(ii)}
For any $0 \leq \eps \leq 1$, $\s > 0$, and time interval $[t_0, t_0 + 1] \subset [0, T]$, we have the tail estimate
\begin{align*}
    \PP \Big( \| Z_\eps \|_{C_{[t_0, t_0 + 1]} W_x^{- \s, \infty}} > \ld \Big) \leq C \exp ( - c \ld^{\frac{2}{r}} )
\end{align*}

\noi
for any $\ld > 0$ and some constants $C, c > 0$ independent of $\eps$ and $t_0$, and also the moment bound
\begin{align*}
    \E \Big[ \| Z_\eps \|_{C_{[t_0, t_0 + 1]} W_x^{- \s, \infty}}^p \Big] \leq C (p, r)
\end{align*}

\noi
for any $1 \leq p < \infty$ and some constant $C(p, r) > 0$ independent of $\eps_1$, $\eps_2$, and $t_0$.

\smallskip \noi
\textup{(iii)} For any $0 \leq \eps_1 < \eps_2 \leq 1$, $0 < \gamma < \s$, and time interval $[t_0, t_0 + 1] \subset [0, T]$, we have the tail estimate
\begin{align*}
    \PP \Big( (\eps_2 - \eps_1)^{- \gamma} \| Z_{\eps_2} - Z_{\eps_1}  \|_{C_{[t_0, t_0 + 1]} W_x^{- \s, \infty}} > \ld \Big) \leq C \exp ( - c \ld^{\frac 2r} )
\end{align*}

\noi
for any $\ld > 0$ and some constants $C, c > 0$ independent of $\eps_1$, $\eps_2$, and $t_0$, and also the moment bound
\begin{align*}
    \E \Big[ (\eps_2 - \eps_1)^{- p \gamma} \| Z_{\eps_2} - Z_{\eps_1} \|_{C_{[t_0, t_0 + 1]} W_x^{- \s, \infty}}^p \Big] \leq C (p, r)
\end{align*}

\noi
for any $1 \leq p < \infty$ and some constant $C(p, r) > 0$ independent of $\eps$ and $t_0$.
Consequently, there exists $\al > 0$ small such that $\eps \mapsto Z_\eps$ is almost surely $\al$-H\"older continuous on $[0, 1]$ with values in $C ([0, T]; W^{- \s, \infty} (\T^2))$.
\end{lemma}

\begin{proof}
The proof follows essentially from \cite[Proposition~4.1]{Zine}. The key estimates are
\begin{align*}
\E \Big[ | \ft{\Psi_\eps^j} (t, n) |^2 \Big] \les \jb{n}^{-2}
\end{align*}

\noi
and
\begin{align*}
\E \Big[ \big| \ft{\Psi_{\eps + h_1}^j} (t + h_2, n) - \ft{\Psi_\eps^j} (t, n) \big|^2 \Big] \les ( |h_1|^{\gamma_1} + |h_2|^{\gamma_1} ) \jb{n}^{-2 + \gamma_2}
\end{align*}

\noi
for any $0 \leq \eps \leq 1$, $t \geq 0$, $n \in \Z^2$, and $h_1, h_2 \in \R$ satisfying $0 \leq \eps + h_1 \leq 1$ and $t + h_2 \geq 0$, where $\gamma_1, \gamma_2 > 0$ are arbitrarily small. These estimates, along with the property of Wick products (see, for example, \cite[Lemma~2.5]{GKO}), the Wiener chaos estimate (see, for example, \cite[Lemma~2.3]{GKO}), a Kolmogorov continuity criterion argument (see, for example, \cite[Appendix]{OOTz}), Chebyshev's inequality (see \cite[Lemma~4.5]{Tz10}), and the Garsia-Rodemich-Rumsey inequality (see \cite[Lemma~2.2]{GKOT}), give the almost sure convergence in part (i) and the desired tail estimates in part (ii) and part (iii). The moment bounds in part (ii) and part (iii) then follow from the layer cake representation and the tail estimates. The H\"older continuity with respect to $\eps$ in part (iii) follows from the moment bound in part (iii) and the Kolmogorov continuity criterion (see \cite[Theorem~3.3]{DPZ}).
\end{proof}

We will also need the following lemma. Let us recall the $I$-operator in \eqref{defIop}. 
\begin{lemma}
\label{LEM:stoI}
Let $j \in \N$, $0 < \eps \leq 1$, $t \in \R_+$, $x \in \T^2$, $0 < s < 1$, and $M \in \N$ with $M \geq 2$. 
Then, $I_M \Psi_\eps^j (t, x)$ is a mean-zero Gaussian random variable with variance bounded by $C(s) \log M$ for some constant $C (s) > 0$ independent of $j$, $\eps$, $t$, $x$, and $M$.
\end{lemma}

\begin{proof}
It is not hard to see that $I_M \Psi_\eps^j (t, x)$ is a mean-zero Gaussian random variable. From \eqref{Psin} and the definition \eqref{defIop}, we have
\begin{align*}
    \E \big[ |I_M \Psi_\eps^j (t, x)|^2 \big] \les \sum_{\substack{n \in \Z^2 \\ |n| \les M}} \frac{1}{\jb{n}^{2}} + \sum_{\substack{n \in \Z^2 \\ |n| \ges M}} \frac{M^{2 - 2s}}{\jb{n}^{4 - 2 s}} \les_s \log M ,
\end{align*}

\noi
which gives the desired variance bound. 
\end{proof}

\section{Global well-posedness for the \PLS and the mean-field SNLH}

In this section, we prove global well-posedness of the \PLS \eqref{NLHNu} as stated in Theorem~\ref{THM:conv1}~(i) and the mean-field SNLH \eqref{NLHu} as stated in Proposition~\ref{PROP:GWP_NLH}.

\subsection{Global well-posedness for the $\text{PLSM}_N$}
\label{SUB:SNLHsys}

Let us consider the perturbed \PLS \eqref{NLHNv}, whose Duhamel formulation is given by
\begin{align}
\begin{split}
    v^{N, j} (t) = P_0 (t) u_0^{N, j} - \frac{1}{N} &\sum_{k = 1}^N \mathcal{I}_0  \big( \wick{ (\Psi_0^k)^2 \Psi_0^j } + \wick{ (\Psi_0^k)^2 } v^{N, j} + 2 v^{N, k} \wick{ \Psi_0^k \Psi_0^j } \\
    &+ 2 \Psi_0^k v^{N, k} v^{N, j} + (v^{N, k})^2 \Psi_0^j + (v^{N, k})^2 v^{N, j} \big) (t) , \quad j = 1, \dots, N,
\end{split}
\label{NLHNvDuh}
\end{align}

\noi
where $P_0$ is defined in \eqref{defP0}, $\mathcal{I}_0$ is defined in \eqref{defI0}, and $\uu_0^N = (u_0^{N, j})_{1 \leq j \leq N}$ is the initial data for \eqref{NLHNv}. The system \eqref{NLHNvDuh} is known to be globally well-posed by \cite[Lemma~2.2]{SSZZ}, but does not satisfy the regularity requirement in our setting. Our goal is to prove the following proposition on global well-posedness of the system \eqref{NLHNvDuh}, which implies Theorem~\ref{THM:conv1}~(i) on pathwise global well-posedness for the \PLS \eqref{NLHNu}.
\begin{proposition}
\label{PROP:NLHN}
Let $N \in \N$, $\frac 12 < s < 1$, and $T \geq 1$. Let $\uu_0^N = (u_0^{N, j})_{1 \leq j \leq N} \in H^s (\T^2)^{\otimes N}$. Then, almost surely, there exists a unique solution $\vv^N = ( v^{N, j} )_{1 \leq j \leq N} \in C ([0, T]; H^s (\T^2)^{\otimes N})$ to the perturbed \PLS \eqref{NLHNvDuh} with initial data $\uu_0^N$. Moreover, we have the bound
\begin{align}
    \| \vv^N \|_{\A_N C_T H_x^s} \leq C (\om, N, T, \| \uu_0^N \|_{\A_N H^s})
\label{vNTbdd}
\end{align}

\noi
for some constant $C(\om, N, T, \| \uu_0^N \|_{\A_N H^s}) > 0$.
\end{proposition}

\begin{proof}
Let us define 
\begin{align*}
\mathbf{\Psi}_0^N = \big( \Psi_0^j, \wick{\Psi_0^k \Psi_0^j} , \wick{(\Psi_0^k)^2 \Psi_0^j} \big)_{1 \leq j, k \leq N}
\end{align*}

\noi
and, given $\s \in \R$, define the norm
\begin{align*}
    \| \mathbf{\Psi}_0^N \|_{\mathcal{Z}^{\s, N}_T} &\deff \| \Psi_0^j \|_{\A_{N, j} C_T W_x^{\s, \infty}} + \| \wick{(\Psi_0^j)^2} \|_{\A_{N, j} C_T W_x^{\s, \infty}} \\
    &\quad + \| \wick{\Psi_0^k \Psi_0^j} \|_{\A_{N, j, k}^{(2)} C_T W_x^{\s, \infty}} + \| \wick{(\Psi_0^k)^2 \Psi_0^j} \|_{\A_{N, j, k}^{(2)} C_T W_x^{\s, \infty}} ,
\end{align*}

\noi
where $\A_N$ and $\A_N^{(2)}$ are the $\ell^2$-averages defined in \eqref{defANX} and \eqref{defAN2X}, respectively. From Lemma~\ref{LEM:sto}, we know that for any $\ta > 0$,
\begin{align}
\| \mathbf{\Psi}_0^N \|_{\mathcal{Z}^{- \ta, N}_T} \leq C(\om, N, T)
\label{Psi0N_bdd}
\end{align}

\noi
for some constant $C(\om, N, T) > 0$.

Given $T_0 > 0$, we define the space $\mathcal{X}_{T_0}^{s}$ via the norm
\begin{align*}
\| u \|_{\mathcal{X}_{T_0}^{s}} \deff \| u \|_{C_{T_0} H_x^{s}} + \sup_{0 \leq t \leq T_0} t^{\frac 34} \| u (t) \|_{\C^{s + \frac 12}} .
\end{align*}
We first show that there exists $T_1 = T_1 (\om, N, T) > 0$ such that a unique solution $\vv^N$ to the equation \eqref{NLHNvDuh} exists in the space $C ([0, T_1]; H^{s} (\T^2)^{\otimes N}) \cap C ((0, T_1]; \C^{s + \frac 12} (\T^2)^{\otimes N})$ endowed with the norm $\A_N \mathcal{X}_{T_1}^s$. 
Note that by interpolation and the embedding in Lemma~\ref{LEM:emb}~(ii), we have
\begin{align}
\begin{split}
\sup_{0 \leq t \leq T_0} t^{\frac{3}{10}} \| u (t) \|_{\C^{s - \frac 25}} 
&= \sup_{0 \leq t \leq T_0} \| u (t) \|_{\C^{s - 1}}^{\frac 35} \big( t^{\frac 34} \| u (t) \|_{\C^{s + \frac 12}} \big)^{\frac 25} \\
&\les \sup_{0 \leq t \leq T_0} \| u (t) \|_{H^{s}}^{\frac 35} \big( t^{\frac 34} \| u (t) \|_{\C^{s + \frac 12}} \big)^{\frac 25} \\
&\leq \| u \|_{\mathcal{X}_{T_0}^s} .
\end{split}
\label{Cinterp}
\end{align}

\noi
Let $\ta > 0$. For any $T_0 > 0$, from Minkowski's integral inequality and Lemma~\ref{LEM:heat}, we have
\begin{align}
\begin{split}
\| \mathcal{I}_0 (F) (t) \|_{C_{T_0} H_x^s} 
&\leq \bigg\| \int_0^t \| P_0 (t - t') F (t') \|_{H_x^s} dt' \bigg\|_{C_{T_0}} \\
&\les \bigg\| \int_0^t (t - t')^{- \frac{s + \ta}{2}} dt' \bigg\|_{C_{T_0}} \| F \|_{C_{T_0} H_x^{- \ta}} \\
&\les T_0^{\frac{2 - s - \ta}{2}} \| F \|_{C_{T_0} H_x^{- \ta}}
\end{split}
\label{I0C1}
\end{align}

\noi
provided that $s + \ta < 2$ and
\begin{align}
\| \mathcal{I}_0 (F) (t) \|_{\C^{s + \frac 12}} \les \int_0^t (t - t')^{- \frac{2s + 2 \ta + 1}{4}} \| F (t') \|_{\C^{- \ta}} dt' .
\label{I0C2}
\end{align}

\noi
We define $\Gamma^{N, j} [\vv^N]$ to be the right-hand side of \eqref{NLHNvDuh}. Let us first consider the $H^s$-norm.
For any $0 < T_1 \leq 1$, by the product estimates in Lemma~\ref{LEM:prod}, H\"older's inequalities, and Sobolev's inequalities, we have
\begin{align}
\| \wick{ (\Psi_0^k)^2} v^{N, j} \|_{C_{T_1} H_x^{- \ta}} 
\les \| \wick{ (\Psi_0^k)^2} \|_{C_{T_1} W_x^{- \ta, \infty}} \| v^{N, j} \|_{C_{T_1} H_x^{\ta}} , 
\label{GWPh3}
\end{align}
\begin{align}
\| v^{N, k} \wick{ \Psi_0^k \Psi_0^j} \|_{C_{T_1} H_x^{- \ta}} 
\les \| v^{N, k} \|_{C_{T_1} H_x^{\ta}} \| \wick{ \Psi_0^k \Psi_0^j} \|_{C_{T_1} W_x^{- \ta, \infty}} ,
\label{GWPh4}
\end{align}
\begin{align}
\begin{split}
\| \Psi_0^k v^{N, k} v^{N, j} \|_{C_{T_1} H_x^{- \ta}} 
&\les \| \Psi_0^k \|_{C_{T_1} W_x^{- \ta, \infty}} \| v^{N, k} \|_{C_{T_1} W_x^{\ta, 4}} \| v^{N, j} \|_{C_{T_1} W_x^{\ta, 4}} \\
&\les \| \Psi_0^k \|_{C_{T_1} W_x^{- \ta, \infty}} \| v^{N, k} \|_{C_{T_1} H_x^{\ta + \frac 12}} \| v^{N, j} \|_{C_{T_1} H_x^{\ta + \frac 12}} ,
\end{split}
\label{GWPh5}
\end{align}
\begin{align}
\begin{split}
\| (v^{N, k})^2 \Psi_0^j \|_{C_{T_1} H_x^{- \ta}} 
&\les \| v^{N, k} \|_{C_{T_1} W_x^{\ta, 4}}^2 \| \Psi_0^j \|_{C_{T_1} W_x^{- \ta, \infty}} \\
&\les \| v^{N, k} \|_{C_{T_1} H_x^{\ta + \frac 12}}^2 \| \Psi_0^j \|_{C_{T_1} W_x^{- \ta, \infty}} ,
\end{split}
\label{GWPh6}
\end{align}
\begin{align}
\begin{split}
\| (v^{N, k})^2 v^{N, j} \|_{C_{T_1} H_x^{- \frac 12}} 
&\les \| (v^{N, k})^2 v^{N, j} \|_{C_{T_1} L_x^{\frac 43}} \\
&\les \| v^{N, k} \|_{C_{T_1} L_x^4}^2 \| v^{N, j} \|_{C_{T_1} L_x^4} \\
&\les \| v^{N, k} \|_{C_{T_1} H_x^{\frac 12}}^2 \| v^{N, j} \|_{C_{T_1} H_x^{\frac 12}} .  
\end{split}
\label{GWPh7}
\end{align}

\noi
Thus, using Lemma~\ref{LEM:heat}, combining \eqref{I0C1} (with $\ta = \frac 12$ for the $(v^{N, k})^2 v^{N, j}$ term and $\ta > 0$ sufficiently small for other terms), \eqref{GWPh3}, \eqref{GWPh4}, \eqref{GWPh5}, \eqref{GWPh6}, and \eqref{GWPh7} along with the fact that $\frac 12 < s < 1 < \frac 32$, and applying the Cauchy-Schwarz inequalities in $k$ and Young's inequalities, we obtain
\begin{align}
\big\| \Gamma^{N, j} [\vv^N] \big\|_{\A_{N, j} C_{T_1} H_x^s} \les \| \uu_0^N \|_{\A_N H^s} + T_1^{\frac{3 - 2 s}{4}} \Big( 1 + \| \mathbf{\Psi}_0^N \|_{\mathcal{Z}^{- \ta, N}_T}^3 + \| \vv^N \|_{\A_N C_{T_1} H_x^s}^3 \Big) .
\label{GWPh8}
\end{align}
 
\noi
For the $\C^{s + \frac 12}$-norm, by the product estimates in Lemma~\ref{LEM:prodC} along with Lemma~\ref{LEM:emb}~(iii), \eqref{Cinterp}, and the fact that $\frac 12 < s < 1$ with $\ta > 0$ being sufficiently small, we have
\begin{align}
\begin{split}
\| \wick{ (\Psi_0^k)^2} (t') v^{N, j} (t') \|_{\C^{- \ta}} 
&\les \| \wick{ (\Psi_0^k)^2} (t') \|_{\C^{- \ta}} \| v^{N, j} (t') \|_{\C^{2 \ta}} \\
&\leq (t')^{- \frac{3}{10}} \| \wick{ (\Psi_0^k)^2} \|_{C_{T_1} W_x^{- \ta, \infty}} \| v^{N, j} \|_{\mathcal{X}_{T_1}^s} , 
\end{split}
\label{GWPh11}
\end{align}
\begin{align}
\begin{split}
\| v^{N, k} (t') \wick{ \Psi_0^k \Psi_0^j} (t') \|_{\C^{- \ta}} 
&\les \| v^{N, k} (t') \|_{\C^{2 \ta}} \| \wick{ \Psi_0^k \Psi_0^j} (t') \|_{\C^{- \ta}} \\
&\leq (t')^{- \frac{3}{10}} \| v^{N, k} \|_{\mathcal{X}_{T_1}^s} \| \wick{ \Psi_0^k \Psi_0^j} \|_{C_{T_1} W_x^{- \ta, \infty}} ,
\end{split}
\label{GWPh12}
\end{align}
\begin{align}
\begin{split}
\| \Psi_0^k (t') v^{N, k} (t') v^{N, j} (t') \|_{\C^{- \ta}} 
&\les \| \Psi_0^k (t') \|_{\C^{- \ta}} \| v^{N, k} (t') \|_{\C^{2 \ta}} \| v^{N, j} (t') \|_{\C^{2 \ta}} \\
&\les (t')^{- \frac 35} \| \Psi_0^k \|_{C_{T_1} W_x^{- \ta, \infty}} \| v^{N, k} \|_{\mathcal{X}_{T_1}^s} \| v^{N, j} \|_{\mathcal{X}_{T_1}^s} ,
\end{split}
\label{GWPh13}
\end{align}
\begin{align}
\begin{split}
\| (v^{N, k} (t'))^2 \Psi_0^j (t') \|_{\C^{- \ta}} 
&\les \| v^{N, k} (t') \|_{\C^{2 \ta}}^2 \| \Psi_0^j (t') \|_{\C^{- \ta}} \\
&\les (t')^{- \frac 35} \| v^{N, k} \|_{\mathcal{X}_{T_1}^s}^2 \| \Psi_0^j \|_{C_{T_1} W_x^{- \ta, \infty}} ,
\end{split}
\label{GWPh14}
\end{align}
\begin{align}
\begin{split}
\| (v^{N, k} (t'))^2 v^{N, j} (t') \|_{\C^{- \ta}} 
&\les \| v^{N, k} (t') \|_{\C^{\ta}}^2 \| v^{N, j} (t') \|_{\C^{\ta}} \\
&\les (t')^{- \frac{9}{10}} \| v^{N, k} \|_{\mathcal{X}_{T_1}^s}^2 \| v^{N, j} \|_{\mathcal{X}_{T_1}^s}   
\end{split}
\label{GWPh15}
\end{align}

\noi
for any $0 \leq t' \leq T_1$. Thus, using Lemma~\ref{LEM:heat} and Lemma~\ref{LEM:emb}~(ii), combining  \eqref{I0C2}, \eqref{GWPh11}, \eqref{GWPh12}, \eqref{GWPh13}, \eqref{GWPh14}, and \eqref{GWPh15}, and applying the Cauchy-Schwarz inequalities in $k$ and Young's inequalities, we obtain
\begin{align}
\begin{split}
\Big\| &\sup_{0 \leq t \leq T_1} t^{\frac 34} \big\| \Gamma^{N, j} [\vv^N] (t) \big\|_{\C^{s + \frac 12}} \Big\|_{\A_{N, j}} \\
&\les \| \uu_0^N \|_{\A_N H^{s}} + T_1^{\frac{1}{10} + \frac{1 - s - \ta}{2}} \Big( 1 + \| \mathbf{\Psi}_0^N \|_{\mathcal{Z}_T^{- \ta, N}}^3 + \| \vv^N \|_{\A_N \mathcal{X}_{T_1}^s}^3 \Big) .
\end{split}
\label{GWPh16}
\end{align}

\noi
Here, we note that the power $\frac 34$ of $t$ in the definition of the $\mathcal{X}_{T_1}^s$-norm is chosen precisely to obtain the bound on $P_0 (\cdot) u_0^{N, j}$. Also, here we used the fact that $s + \ta < \frac 32$, which is satisfied under the assumptions that $s < 1$ and $\ta$ is sufficiently small. From \eqref{GWPh8} and \eqref{GWPh16}, we get
\begin{align*}
\big\| \Gamma^{N, j} [\vv^N] \big\|_{\A_{N, j} \mathcal{X}_{T_1}^s} 
\les \| \uu_0^N \|_{\A_N H^{s}} + T_1^{\frac{1}{10} + \frac{1 - s - \ta}{2}} \Big( 1 + \| \mathbf{\Psi}_0^N \|_{\mathcal{Z}_T^{- \ta, N}}^3 + \| \vv^N \|_{\A_N \mathcal{X}_{T_1}^s}^3 \Big) .
\end{align*}

\noi
Using similar steps, we obtain the following difference estimate:
\begin{align*}
\big\| &\Gamma^{N, j} [\vv_1^N] - \Gamma^{N, j} [\vv_2^N] \big\|_{\A_{N, j} \mathcal{X}_{T_1}^s} \\
&\les T_1^{\frac{1}{10} + \frac{1 - s - \ta}{2}} \| \vv_1^N - \vv_2^N \|_{\A_N \mathcal{X}_{T_1}^s} \Big( 1 + \| \mathbf{\Psi}_0^N \|_{\mathcal{Z}_T^{- \ta, N}}^2 + \| \vv_1^N \|_{\A_N \mathcal{X}_{T_1}^s}^2 + \| \vv_2^N \|_{\A_N \mathcal{X}_{T_1}^s}^2 \Big) .
\end{align*}

\noi
Therefore, thanks to \eqref{Psi0N_bdd}, we can use a standard contraction argument to obtain a unique solution $\vv^N$ to \eqref{NLHNvDuh} in $C ([0, T_1]; H^{s} (\T^2)^{\otimes N}) \cap C ((0, T_1]; \C^{s + \frac 12} (\T^2)^{\otimes N})$ for some $T_1 = T_1 (\om, N, T) > 0$ sufficiently small.

Given any $0 < T_0 < T$ and $\vv^N (T_0) \in \C^{s + \frac 12} (\T^2)^{\otimes N}$, by using a slight variant of \cite[Theorem~6.1]{MW17} (valid since $s + \frac 12 > 1$ given $s > \frac 12$), we know that there exists a unique solution $\vv^N$ to \eqref{NLHNvDuh} in the space $C ([T_0, T]; \C^{s + \frac 12} (\T^2)^{\otimes N})$. Thus, from the above local well-posedness result and the embedding in Lemma~\ref{LEM:emb}~(i), we know that a solution exists in the space $C ([0, T]; H^s (\T^2)^{\otimes N})$.

It remains to show that the solution $\vv^N$ constructed above is unique in $C ([0, T]; H^s (\T^2)^{\otimes N})$. Let $\vv_1^N$ and $\vv_2^N$ be two solutions to \eqref{NLHNvDuh} in $C([0, T]; H^s (\T^2)^{\otimes N})$ with the common initial data $\uu_0^N$. 
For any $T_0 \geq 0$ and $T' > 0$ such that $T_0 + T' \leq T$, by using similar steps that lead to \eqref{GWPh8}, we obtain
\begin{align}
\begin{split}
&\| \vv_1^N - \vv_2^N \|_{\A_N C_{[T_0, T_0 + T']} H_x^s} \\
&\les \| \vv_1^N (T_0) - \vv_2^N (T_0) \|_{\A_N H^s} + (T')^{\frac{3 - 2s}{4}} \| \vv_1^N - \vv_2^N \|_{\A_N C_{[T_0, T_0 + T']} H_x^s} \\
&\qquad  
\times \Big( 1 + \| \mathbf{\Psi}_0^N \|_{\mathcal{Z}_T^{- \ta, N}}^2 + \| \vv_1^N \|_{\A_N C_{T} H_x^s}^2 + \| \vv_2^N \|_{\A_N C_{T} H_x^s}^2 \Big) .
\end{split}
\label{vdiff}
\end{align}

\noi
Then, thanks to \eqref{Psi0N_bdd}, we get $\vv_1^N = \vv_2^N$ on each time interval $[T_0, T_0 + T']$ with sufficiently small  $T' > 0$ independent of $T_0$, which gives the uniqueness result.
Thus, we have finished the proof.
\end{proof}

\begin{remark} \rm
The condition $s < 1$ in Proposition~\ref{PROP:NLHN} is by no means to be sharp. One can obtain global well-posedness of \eqref{NLHNvDuh} for some range of $s \geq 1$ by introducing the helper norm
\begin{align*}
\| u \|_{\mathcal{Y}_{T_0}^s} \deff \| u \|_{C_{T_0} H_x^s} + \sup_{0 \leq t \leq T_0} t^{\frac{\gamma + 1}{2}} \| u (t) \|_{\C^{s + \gamma}}
\end{align*}

\noi
and optimizing the choice of $\gamma$. However, since our main goal is to establish Smoluchowski-Kramers approximations, we choose not to pursue this point.
\end{remark}

\subsection{Global well-posedness for the mean-field SNLH}
\label{SUB:mfSNLH}

Let us consider the perturbed mean-field SNLH \eqref{NLHv} and drop the superscript $j$ in this subsection for simplicity. The Duhamel formulation of \eqref{NLHv} is given by
\begin{align}
v (t) = P_0 (t) u_0 - \mathcal{I}_0 \big( 2 \E [\Psi_0 v] \Psi_0 + 2 \E [\Psi_0 v] v + \E [v^2] \Psi_0 + \E [v^2] v \big) (t) ,
\label{NLHvDuh}
\end{align}

\noi
where $P_0$ is defined in \eqref{defP0}, $\mathcal{I}_0$ is defined in \eqref{defI0}, and $u_0$ is the initial data for \eqref{NLHv}. The equation \eqref{NLHvDuh} is known to be globally well-posed by \cite[Theorem~3.6]{SSZZ}, but we need better regularity of the solution. Our goal is to prove the following proposition on global well-posedness of the equation \eqref{NLHvDuh}, which implies Proposition~\ref{PROP:GWP_NLH} on global well-posedness for the mean-field SNLH \eqref{NLHu}.
\begin{proposition}
\label{PROP:NLH}
Let $\frac 12 < s < 1$ and $T > 0$. Let $u_0 \in L^4 (\Om; H^s (\T^2))$. Then, there exists a unique solution $v \in L^2 (\Om; C ([0, T]; H^s (\T^2)))$ to the perturbed mean-field SNLH \eqref{NLHvDuh} with initial data $u_0$. Moreover, we have the bound
\begin{align}
    \| v \|_{L^2_\om C_T H^s_x} \leq C (T, \| u_0 \|_{L^4_\om H^s})
\label{vTbdd}
\end{align}

\noi
for some constant $C (T, \| u_0 \|_{L^4_\om H^s}) > 0$.
\end{proposition}

\begin{proof}
By Sobolev's inequality, we know that $u_0 \in L^4 (\Om; L^4 (\T^2))$. Let $\ta > 0$ be small. From \cite[Theorem~3.6]{SSZZ}, we know that there exists a unique solution $v$ to the equation \eqref{NLHvDuh} in the space 
\begin{align*}
L^2 (\Om; C ((0, T]; \C^{4 \ta} (\T^2)) \cap C([0, T]; L^4 (\T^2)) ) .
\end{align*}

\noi
In particular, we have
\begin{align}
\Big\| \sup_{0 \leq t \leq T} t^{2 \ta + \frac 14} \| v (t) \|_{\C_x^{4 \ta}} \Big\|_{L_\om^2} + \| v \|_{L_\om^2 C_T L_x^4} \leq C (T, \| u_0 \|_{L^4_\om H^s}) 
\label{L4bdd}
\end{align}

\noi
for some constant $C (T, \| u_0 \|_{L^4_\om H^s}) > 0$.

We now show that the solution $v$ lies in $L^2 (\Om; C([0, T]; H^s (\T^2)))$. 
By using \eqref{NLHvDuh}, Lemma~\ref{LEM:heat}, \eqref{I0C1}, and Minkowski's integral inequality, we have
\begin{align}
\begin{split}
\| v \|_{L^2_\om C_{T} H_x^s} 
&\les \| u_0 \|_{L_\om^2 H^s} + T^{\frac{3 - 2s}{4}} \big\| \E [v^2] v \big\|_{L_\om^2 C_{T} H_x^{- \frac 12}} \\
&\quad + \bigg\| \int_0^t (t - t')^{- \frac{s + \ta}{2}} \big\| \E [\Psi_0 (t') v (t')] \Psi_0 (t') \big\|_{H_x^{- \ta}} dt' \bigg\|_{L_\om^2 C_{T}} \\
&\quad + \bigg\| \int_0^t (t - t')^{- \frac{s + \ta}{2}} \big\| \E [\Psi_0 (t') v (t')] v (t') \big\|_{H_x^{- \ta}} dt' \bigg\|_{L_\om^2 C_{T}} \\
&\quad + \bigg\| \int_0^t (t - t')^{- \frac{s + \ta}{2}}  \big\| \E [v (t')^2] \Psi_0 (t') \big\|_{H_x^{- \ta}}  dt' \bigg\|_{L_\om^2 C_{T}} .
\end{split}
\label{heat0}
\end{align}

\noi
By Sobolev's inequality, H\"older's inequality, Minkowski's integral inequality, and Sobolev's inequality, we get
\begin{align}
\begin{split}
\big\| \E [v^2] v \big\|_{L_\om^2 C_{T} H_x^{- \frac 12}}
&\les \big\| \E [v^2] v \big\|_{L_\om^2 C_{T} L_x^{\frac 43}} \\
&\leq \E \big[ \| v^2 \|_{C_{T} L_x^2} \big] \| v \|_{L_\om^2 C_{T} L_x^4} \\
&\leq \| v \|_{L_\om^2 C_{T} L_x^4}^3 .
\end{split}
\label{heat1}
\end{align}

\noi
By proceeding as in \cite[Lemma~3.1]{SSZZ}, we write
\begin{align}
\E [\Psi_0 v] \Psi_0 = \E [ v' \Psi_0' \Psi_0 | \Psi_0 ] ,
\label{cond0}
\end{align}

\noi
where $(\Psi_0', v')$ is an independent copy of $(\Psi_0, v)$.
By \eqref{cond0}, Jensen's inequality, the product estimate in Lemma~\ref{LEM:prod}~(ii), and the embedding in Lemma~\ref{LEM:emb}~(i), we get
\begin{align}
\begin{split}
\big\| \E [\Psi_0 (t') v (t')] \Psi_0 (t') \big\|_{H_x^{- \ta}} 
&\leq \E \big[ \| v' (t') \Psi_0' (t') \Psi_0 (t') \|_{H_x^{- \ta}} | \Psi_0 \big]  \\
&\les \E \big[ \| v' (t') \|_{H_x^{\ta}} \| \Psi_0' (t') \Psi_0 (t') \|_{W_x^{- \ta, \infty}} | \Psi_0 \big]  \\
&\leq \E \big[ \| v' (t') \|_{\C_x^{4 \ta}} \| \Psi_0' (t') \Psi_0 (t') \|_{W_x^{- \ta, \infty}} | \Psi_0 \big] .
\end{split}
\label{heat2}
\end{align}

\noi
By the product estimate in Lemma~\ref{LEM:prod}~(ii), Minkowski's integral inequality, Lemma~\ref{LEM:prod}~(ii) again, the Cauchy-Schwarz inequality in $\om$, interpolation, and the embedding in Lemma~\ref{LEM:emb}~(i), we get
\begin{align}
\begin{split}
\big\| &\E [\Psi_0 (t') v (t')] v (t') \big\|_{H_x^{- \ta}} \\
&\les \E \big[ \| \Psi_0 (t') v (t') \|_{W_x^{- \ta, 4}} \big] \| v (t') \|_{W_x^{\ta, 4}} \\
&\les \E \big[ \| \Psi_0 (t') \|_{W_x^{- \ta, \infty}} \| v (t') \|_{W_x^{\ta, 4}} \big] \| v (t') \|_{W_x^{\ta, 4}} \\
&\les \| \Psi_0 (t') \|_{L_\om^2 W_x^{- \ta, \infty}} \| v (t') \|_{L_\om^2  W_x^{\ta, 4}} \| v (t') \|_{W_x^{\ta, 4}} \\
&\les \| \Psi_0 (t') \|_{L_\om^2 W_x^{- \ta, \infty}} \| v (t') \|_{L_\om^2 L_x^4}^{\frac 12} \| v (t') \|_{L_\om^2 \C_x^{4 \ta}}^{\frac 12} \| v (t') \|_{L_x^{4}}^{\frac 12} \| v (t') \|_{\C_x^{4 \ta}}^{\frac 12}  .
\end{split}
\label{heat3}
\end{align}

\noi
By the product estimate in Lemma~\ref{LEM:prod}~(ii), Minkowski's integral inequality, the product estimate in Lemma~\ref{LEM:prod}~(i), interpolation, and the embedding in Lemma~\ref{LEM:emb}~(i), we get
\begin{align}
\begin{split}
\big\| \E [v (t')^2] \Psi_0 (t') \big\|_{H_x^{- \ta}}
&\les \E \big[ \| v (t')^2 \|_{H_x^\ta} \big] \| \Psi_0 (t') \|_{W_x^{- \ta, \infty}}  \\
&\les \| v (t') \|_{L_\om^2 W_x^{\ta, 4}}^2 \| \Psi_0 (t') \|_{W_x^{- \ta, \infty}} \\
&\les \| v (t') \|_{L_\om^2 L_x^4} \| v (t') \|_{L_\om^2 \C_x^{4 \ta}} \| \Psi_0  (t')\|_{W_x^{- \ta, \infty}} .
\end{split}
\label{heat4}
\end{align}

\noi
Combining \eqref{heat0}, \eqref{heat1}, \eqref{heat2} (for which we need to apply conditional H\"older's inequality), \eqref{heat3}, and \eqref{heat4} along with Minkowski's integral inequalities and applying the moment bound in Lemma~\ref{LEM:sto} (which also applies to $\Psi_0' \Psi_0$), we obtain
\begin{align*}
\| v \|_{L_\om^2 C_{T} H_x^s} 
&\les C \| u_0 \|_{L_\om^2 H_x^s} + T^{\frac{3 - 2s}{4}} \| v \|_{L_\om^2 C_T L_x^4}^3 \\
&\quad + T^{\frac 34 - \frac{s + 5 \ta}{2}} \Big\| \sup_{0 \leq t \leq T} t^{2 \ta + \frac 14} \| v (t) \|_{\C_x^{4 \ta}} \Big\|_{L_\om^2} \big( 1 + \| v \|_{L_\om^2 C_T L_x^4} \big) .
\end{align*}

\noi
Thus, by using \eqref{L4bdd}, we know that $v \in L^2 (\Om; C ([0, T]; H^s (\T^2)))$ and we obtain the desired bound \eqref{vTbdd}.

It remains to show that the solution $v$ is unique in $L^2 (\Om; C([0, T]; H^s (\T^2)))$. Let $v_1$ and $v_2$ be two solutions to \eqref{NLHvDuh} in $L^2 (\Om; C([0, T]; H^s (\T^2)))$ with the common initial data $u_0$. Similar to \eqref{vdiff} in the proof of Proposition~\ref{PROP:NLHN}, for any $T_0 \geq 0$ and $T' > 0$ such that $T_0 + T' \leq T$, we obtain
\begin{align*}
&\| v_1 - v_2 \|_{L_\om^2 C_{[T_0, T_0 + T']} H_x^s} \\
&\les \| v_1 (T_0) - v_2 (T_0) \|_{L_\om^2 H_x^s} \\
&\quad + (T')^{\frac{3 - 2s}{4}} \| v_1 - v_2 \|_{L_\om^2 C_{[T_0, T_0 + T']} H_x^s} \Big(1 + \| v_1 \|_{L_\om^2 C_{T} H_x^s}^2 + \| v_2 \|_{L_\om^2 C_{T} H_x^s}^2 \Big) ,
\end{align*}

\noi
where the only modifications we need to do is to replace the $\A_N$-norm by the $L_\om^2$-norm, use the identity \eqref{cond0}, and apply the moment bound in Lemma~\ref{LEM:sto}.
\end{proof}

\section{Convergence via regime I}
\label{SEC:regI}

In this section, we prove Theorem~\ref{THM:conv1}, the convergence of the \HLSe \eqref{NLWNeu} to the mean-field SNLH \eqref{NLHu} via regime I.

We first note that part (i) of Theorem~\ref{THM:conv1}, pathwise global well-posedness for \PLS \eqref{NLHNu}, has already been established via Proposition~\ref{PROP:NLHN}. Also, as mentioned in the introduction, part (iii) of Theorem~\ref{THM:conv1}, the convergence from \PLS \eqref{NLHNu} to the mean-field SNLH \eqref{NLHu}, has been covered by \cite[Theorem~4.1]{SSZZ}. Thus, we focus on proving part (ii) of Theorem~\ref{THM:conv1}, the convergence of the \HLSe \eqref{NLWNeu} to the \PLS \eqref{NLHNu} as $\eps \to 0$.

To prove the theorem, we need to show the convergence of $\vv_\eps^N = (v_\eps^{N, j})_{1 \leq j \leq N}$ satisfying the perturbed \HLSe
\begin{align}
\begin{split}
    v_\eps^{N, j} (t) &= P_\eps (t) (u_0^{N, j}, u_1^{N, j}) - \frac{1}{N} \sum_{k = 1}^N \mathcal{I}_\eps \big( \wick{ (\Psi_\eps^k)^2 \Psi_\eps^j } + \wick{ (\Psi_\eps^k)^2 } v_\eps^{N, j} \\
    &\qquad + 2 v_\eps^{N, k} \wick{ \Psi_\eps^k \Psi_\eps^j } + 2 \Psi_\eps^k v_\eps^{N, k} v_\eps^{N, j} + (v_\eps^{N, k})^2 \Psi_\eps^j + (v_\eps^{N, k})^2 v_\eps^{N, j} \big) (t)
\end{split}
\label{NLWNevDuh2}
\end{align}

\noi
to $\vv^N = (v^{N, j})_{1 \leq j \leq N}$ satisfying the perturbed \PLS
\begin{align}
\begin{split}
    v^{N, j} (t) &= P_0 (t) (u_0^{N, j}, u_1^{N, j}) - \frac{1}{N} \sum_{k = 1}^N \mathcal{I}_0 \big( \wick{ (\Psi_0^k)^2 \Psi_0^j } + \wick{ (\Psi_0^k)^2 } v^{N, j} \\
    &\qquad + 2 v^{N, k} \wick{ \Psi_0^k \Psi_0^j } + 2 \Psi_0^k v^{N, k} v^{N, j} + (v^{N, k})^2 \Psi_0^j + (v^{N, k})^2 v^{N, j} \big) (t) ,
\end{split}
\label{NLHNvDuh2}
\end{align}

\noi
where $P_\eps$ is defined in \eqref{defPe}, $\mathcal{I}_\eps$ is defined in \eqref{defIe}, $P_0$ is defined in \eqref{defP0} (see also \eqref{P001}), $\mathcal{I}_0$ is defined in \eqref{defI0}, and $(\uu_0^N, \uu_1^N) = ((u_0^{N, j}, u_1^{N, j}))_{1 \leq j \leq N}$ is the initial data for both systems. 
Our goal is to prove the following proposition, which, together with \eqref{expNe}, \eqref{expN}, and the convergence of $\Psi_\eps^j$ to $\Psi_0^j$ from Lemma~\ref{LEM:sto}~(iii), implies Theorem~\ref{THM:conv1}~(ii).
\begin{proposition}
\label{PROP:convv1}
Let $N \in \N$, $\frac 45 < s < 1$, $T \geq 1$, and $(\uu_0^N, \uu_1^N) = ((u_0^{N, j}, u_1^{N, j}))_{1 \leq j \leq N}$ be random and belong to $\H^s (\T^2)^{\otimes N}$. Given any $0 < \eps \leq 1$, let $\vv_\eps^N = (v_\eps^{N, j})_{1 \leq j \leq N} \in C ([0, T]; H^s (\T^2)^{\otimes N})$ be the solution to the perturbed \HLSe \eqref{NLWNevDuh2} with initial data $(\uu_0^N, \uu_1^N)$ guaranteed by Proposition~\ref{PROP:GWP_NLW}. Let $\vv^N = (v^{N, j})_{1 \leq j \leq N} \in C ([0, T]; H^s (\T^2)^{\otimes N})$ be the solution to the perturbed \PLS \eqref{NLHNvDuh2} guaranteed by Theorem~\ref{THM:conv1}~\textup{(i)}. Then, for any $s' < s$, we have almost surely
\begin{align*}
    \vv_\eps^N \longrightarrow \vv^N \quad \text{in } C ([0, T]; H^{s'} (\T^2)^{\otimes N})
\end{align*}

\noi
as $\eps \to 0$.
\end{proposition}

\begin{proof}
We may assume that $\frac 45 < s' < s < 1$. Let us define 
\begin{align*}
    \mathbf{\Psi}_\eps^N = \big( \Psi_\eps^j, \wick{\Psi_\eps^k \Psi_\eps^j}, \wick{(\Psi_\eps^k)^2 \Psi_\eps^j} \big)_{1 \leq j, k \leq N}    
\end{align*}

\noi
and, given $\s \in \R$, recall the norm
\begin{align*}
    \| \mathbf{\Psi}_\eps^N \|_{\mathcal{Z}^{\s, N}_T} &\deff \| \Psi_\eps^j \|_{\A_{N, j} C_T W_x^{\s, \infty}} + \| \wick{(\Psi_\eps^j)^2} \|_{\A_{N, j} C_T W_x^{\s, \infty}} \\
    &\quad + \| \wick{\Psi_\eps^k \Psi_\eps^j} \|_{\A_{N, j, k}^{(2)} C_T W_x^{\s, \infty}} + \| \wick{(\Psi_\eps^k)^2 \Psi_\eps^j} \|_{\A_{N, j, k}^{(2)} C_T W_x^{\s, \infty}} ,
\end{align*}

\noi
where $\A_N$ and $\A_N^{(2)}$ are the $\ell^2$-averages defined in \eqref{defANX} and \eqref{defAN2X}, respectively. From Lemma~\ref{LEM:sto}, we know that
\begin{align}
    \| \mathbf{\Psi}_\eps^N \|_{\mathcal{Z}^{s - 1, N}_T} \leq C(\om, N, T)
\label{PsiN_bdd}
\end{align}

\noi
and
\begin{align}
    \| \mathbf{\Psi}_\eps^N - \mathbf{\Psi}_0^N \|_{\mathcal{Z}^{s - 1, N}_T} \leq C(\om, N, T) \eps^{\al}
\label{PsiN_diff}
\end{align}

\noi
for some $\al > 0$ and constant $C(\om, N, T) \geq 1$ independent of $\eps$.

Following \cite{OOT1, OOT2}, we let $\ld \geq 1$ be a large number to be chosen later and introduce the following norm with an exponentially decaying time weight given any $\s \in \R$:
\begin{align}
    \| u \|_{S_T^{\s, \ld}} \deff \| e^{- \ld t} u \|_{C_T H_x^\s} .
\label{STsl}
\end{align}

\noi
Note that we have
\begin{align}
    \| u \|_{S_T^{\s, \ld}} \leq \| u \|_{C_T H_x^\s} \leq e^{\ld T} \| u \|_{S_T^{\s, \ld}} .
\label{SlT_bdd}
\end{align}

\noi
Also, given any $\s \in \R$, from Minkowski's integral inequality and Lemma~\ref{LEM:heat}, we have
\begin{align}
\begin{split}
    \| \mathcal{I}_0 (F) (t) \|_{S_T^{\s, \ld}} &= \bigg\| \int_0^t e^{- \ld (t - t')} P_0 (t - t') ( e^{- \ld t'} F (t') ) dt' \bigg\|_{C_T H_x^\s} \\
    &\leq \bigg\| \int_0^t e^{- \ld (t - t')} \big\| P_0 (t - t') ( e^{- \ld t'} F (t') ) \big\|_{H_x^\s} dt' \bigg\|_{C_T} \\
    &\les \bigg\| \int_0^t e^{- \ld (t - t')} |t - t'|^{- \frac 12} dt' \bigg\|_{C_T} \| F \|_{S_T^{\s - 1, \ld}} \\
    &\les \ld^{- \frac 12} \| F \|_{S_T^{\s - 1, \ld}} .
\end{split}
\label{conve3h}
\end{align}

From \eqref{NLWNevDuh2} and \eqref{NLHNvDuh2}, we have
\begin{align}
    \| v_\eps^{N, j} - v^{N, j} \|_{S_T^{s', \ld}} \leq \1^j + \II^j + \III^j_1 + \III^j_2 + \III^j_3 + \III^j_4 + \III^j_5 + \III^j_6 ,
\label{conveN0}
\end{align}

\noi
where
\begin{align*}
    \1^j &\deff \big\| (P_\eps - P_0) (\cdot) ( u_0^{N, j}, u_1^{N, j} ) \big\|_{S_T^{s', \ld}} , \\
    \II^j &\deff \frac 1N \sum_{k = 1}^N \Big\| (\mathcal{I}_{\eps} - \mathcal{I}_{0}) \big( \wick{ (\Psi_\eps^k)^2 \Psi_\eps^j } + \wick{ (\Psi_\eps^k)^2 } v_\eps^{N, j}  \\
    &\quad + 2 v_\eps^{N, k} \wick{ \Psi_\eps^k \Psi_\eps^j } + 2 \Psi_\eps^k v_\eps^{N, k} v_\eps^{N, j} + (v_\eps^{N, k})^2 \Psi_\eps^j + (v_\eps^{N, k})^2 v_\eps^{N, j} \big) \Big\|_{S_T^{s', \ld}} , \\
    \III^j_1 &\deff \frac 1N \sum_{k = 1}^N \big\| \mathcal{I}_{0} \big( \wick{ (\Psi_\eps^k)^2 \Psi_\eps^j } - \wick{ (\Psi_0^k)^2 \Psi_0^j } \big) \big\|_{S_T^{s', \ld}} , \\
    \III^j_2 &\deff \frac 1N \sum_{k = 1}^N \big\| \mathcal{I}_{0} \big( \wick{ (\Psi_\eps^k)^2 } v_\eps^{N, j} - \wick{ (\Psi_0^k)^2 } v^{N, j} \big) \big\|_{S_T^{s', \ld}} , \\
    \III^j_3 &\deff \frac 2N \sum_{k = 1}^N \big\| \mathcal{I}_{0} \big( v_\eps^{N, k} \wick{ \Psi_\eps^k \Psi_\eps^j } - \, v^{N, k} \wick{ \Psi_0^k \Psi_0^j } \big) \big\|_{S_T^{s', \ld}} , \\
    \III^j_4 &\deff \frac 2N \sum_{k = 1}^N \big\| \mathcal{I}_{0} \big( \Psi_\eps^k v_\eps^{N, k} v_\eps^{N, j} - \Psi_0^k v^{N, k} v^{N, j} \big) \big\|_{S_T^{s', \ld}} , \\
    \III^j_5 &\deff \frac 1N \sum_{k = 1}^N \big\| \mathcal{I}_{0} \big( (v_\eps^{N, k})^2 \Psi_\eps^j - (v^{N, k})^2 \Psi_0^j \big) \big\|_{S_T^{s', \ld}} , \\
    \III^j_6 &\deff \frac 1N \sum_{k = 1}^N \big\| \mathcal{I}_{0} \big( (v_\eps^{N, k})^2 v_\eps^{N, j} - (v^{N, k})^2 v^{N, j} \big) \big\|_{S_T^{s', \ld}} .
\end{align*}

\noi
For $\1^j$, we use \eqref{SlT_bdd} and Lemma~\ref{LEM:PIconv} to obtain
\begin{align}
    \1^j \leq \big\| (P_\eps - P_0) (\cdot) (u_0^{N, j}, u_1^{N, j}) \big\|_{C_T H^{s'}} \les  \eps^{\frac{s - s'}{2}} \big\| (u_0^{N, j}, u_1^{N, j}) \big\|_{\H^s} . 
\label{conveN1}
\end{align}

\noi
For $\II^j$, by using \eqref{SlT_bdd} and Lemma~\ref{LEM:PIconv}, we have
\begin{align}
\begin{split}
    \II^j &\les \eps^{\frac{s - s'}{2}} \frac{T^{\frac 12}}{N} \sum_{k = 1}^N \Big( \| \wick{ (\Psi_\eps^k)^2 \Psi_\eps^j } \|_{C_T H_x^{s - 1}} + \| \wick{ (\Psi_\eps^k)^2 } v_\eps^{N, j} \|_{C_T H_x^{s - 1}} \\
    &\qquad \qquad + \| v_\eps^{N, k} \wick{ \Psi_\eps^k \Psi_\eps^j } \|_{C_T H_x^{s - 1}} + \| \Psi_\eps^k v_\eps^{N, k} v_\eps^{N, j} \|_{C_T H_x^{s - 1}} \\
    &\qquad \qquad + \| (v_\eps^{N, k})^2 \Psi_\eps^j \|_{C_T H_x^{s - 1}} + \| (v_\eps^{N, k})^2 v_\eps^{N, j} \|_{C_T H_x^{s - 1}} \Big) .
\end{split}
\label{conveN2-0}
\end{align}

\noi
By the product estimates in Lemma~\ref{LEM:prod}, H\"older's inequalities, and Sobolev's inequalities, we get
\begin{align}
\begin{split}
    \| \wick{ (\Psi_\eps^k)^2 } v_\eps^{N, j} \|_{C_T H_x^{s - 1}} &\les \| \wick{ (\Psi_\eps^k)^2 } \|_{C_T W_x^{s - 1, \infty}} \| v_\eps^{N, j} \|_{C_T H_x^{1 - s}} \\
    &\leq \| \wick{ (\Psi_\eps^k)^2 } \|_{C_T W_x^{s - 1, \infty}} \| v_\eps^{N, j} \|_{C_T H_x^{s'}} ,
\end{split}
\label{conveN2-1}
\end{align}
\begin{align}
\begin{split}
    \| v_\eps^{N, k} \wick{ \Psi_\eps^k \Psi_\eps^j } \|_{C_T H_x^{s - 1}} &\les \| v_\eps^{N, k} \|_{C_T H_x^{1 - s}} \| \wick{ \Psi_\eps^k \Psi_\eps^j } \|_{C_T W_x^{s - 1, \infty}} \\
    &\leq \| v_\eps^{N, k} \|_{C_T H_x^{s'}} \| \wick{ \Psi_\eps^k \Psi_\eps^j } \|_{C_T W_x^{s - 1, \infty}} ,
\end{split}
\label{conveN2-2}
\end{align}
\begin{align}
\begin{split}
    \| \Psi_\eps^k v_\eps^{N, k} v_\eps^{N, j} \|_{C_T H_x^{s - 1}} &\les \| \Psi_\eps^k \|_{C_T W_x^{s - 1, \infty}} \| v_\eps^{N, k} \|_{C_T W_x^{1 - s, 4}} \| v_\eps^{N, j} \|_{C_T W_x^{1 - s, 4}} \\
    &\les \| \Psi_\eps^k \|_{C_T W_x^{s - 1, \infty}} \| v_\eps^{N, k} \|_{C_T H_x^{s'}} \| v_\eps^{N, j} \|_{C_T H_x^{s'}} ,
\end{split}
\label{conveN2-3}
\end{align}
\begin{align}
\begin{split}
    \| (v_\eps^{N, k})^2 \Psi_\eps^{j} \|_{C_T H_x^{s - 1}} &\les \| v_\eps^{N, k} \|_{C_T W_x^{1 - s, 4}}^2 \| \Psi_\eps^j \|_{C_T W_x^{s - 1, \infty}} \\
    &\les \| v_\eps^{N, k} \|_{C_T H_x^{s'}}^2 \| \Psi_\eps^j \|_{C_T W_x^{s - 1, \infty}} ,
\end{split}
\label{conveN2-4}
\end{align}
\begin{align}
\begin{split}
    \| (v_\eps^{N, k})^2 v_\eps^{N, j} \|_{C_T H_x^{s - 1}} &\leq \| v_\eps^{N, k} \|_{C_T L_x^6}^2 \| v_\eps^{N, j} \|_{C_T L_x^6} \\
    &\les \| v_\eps^{N, k} \|_{C_T H_x^{s'}}^2 \| v_\eps^{N, j} \|_{C_T H_x^{s'}} ,
\end{split}
\label{conveN2-5}
\end{align}

\noi
where we used $1 - s \leq s'$, $\frac{s' - (1 - s)}{2} \geq \frac 14$, and $\frac{s'}{2} \geq \frac 13$ given $\frac 45 < s' < s < 1$. Combining \eqref{conveN2-0}, \eqref{conveN2-1}, \eqref{conveN2-2}, \eqref{conveN2-3}, \eqref{conveN2-4}, and \eqref{conveN2-5} and using the Cauchy-Schwarz inequalities in $k$ followed by Young's inequalities, we obtain
\begin{align}
    \frac 1N \sum_{j = 1}^N (\II^j)^2 \les T \eps^{s - s'} \Big( 1 + \| \mathbf{\Psi}_\eps^N \|_{\mathcal{Z}_T^{s - 1, N}}^6 + \| \vv_\eps^{N} \|_{\A_N C_T H_x^{s'}}^6 \Big) .
\label{conveN2}
\end{align}

We now estimate $\III^j_1$, $\III^j_2$, $\III^j_3$, $\III^j_4$, $\III^j_5$, and $\III^j_6$. 
By using \eqref{conve3h}, \eqref{SlT_bdd}, and similar steps to those in \eqref{conveN2-1}, \eqref{conveN2-2}, \eqref{conveN2-3}, \eqref{conveN2-4}, and \eqref{conveN2-5}, we obtain
\begin{align*}
    \III^j_1 &\les \frac{1}{N} \sum_{k = 1}^N \big\| \wick{(\Psi_\eps^k)^2 \Psi_\eps^j} -  \wick{(\Psi_0^k)^2 \Psi_0^j} \big\|_{C_T W_x^{s - 1, \infty}} , \\
    \III^j_2 &\les \frac{1}{N} \sum_{k = 1}^N \Big( \ld^{- \frac 12} \| v_\eps^{N, j} - v^{N, j} \|_{S_T^{s', \ld}} \| \wick{ (\Psi_\eps^k)^2 } \|_{C_T W_x^{s - 1, \infty}} \\
    &\qquad + \| v^{N, j} \|_{C_T H_x^{s'}} \| \wick{(\Psi_\eps^k)^2} -  \wick{(\Psi_0^k)^2} \|_{C_T W_x^{s - 1, \infty}} \Big) , \\
    \III^j_3 &\les \frac{1}{N} \sum_{k = 1}^N \Big( \ld^{- \frac 12} \| v_\eps^{N, k} - v^{N, k} \|_{S_T^{s', \ld}} \| \wick{\Psi_\eps^k \Psi_\eps^j} \|_{C_T W_x^{s - 1, \infty}} \\
    &\qquad + \| v^{N, k} \|_{C_T H_x^{s'}} \| \wick{\Psi_\eps^k \Psi_\eps^j} - \wick{\Psi_0^k \Psi_0^j} \|_{C_T W_x^{s - 1, \infty}} \Big) , \\
    \III^j_4 &\les \frac{1}{N} \sum_{k = 1}^N \Big( \| \Psi_\eps^k - \Psi_0^k \|_{C_T W_x^{s - 1, \infty}} \| v_\eps^{N, k} \|_{C_T H_x^{s'}} \| v_\eps^{N, j} \|_{C_T H_x^{s'}} \\
    &\qquad + \ld^{- \frac 12} \| \Psi_0^k \|_{C_T W_x^{s - 1, \infty}} \| v_\eps^{N, k} - v^{N, k} \|_{S_T^{s', \ld}}  \| v_\eps^{N, j} \|_{C_T H_x^{s'}} \\
    &\qquad + \ld^{- \frac 12} \| \Psi_0^k \|_{C_T W_x^{s - 1, \infty}} \| v^{N, k} \|_{C_T H_x^{s'}} \| v_\eps^{N, j} - v^{N, j} \|_{S_T^{s', \ld}} \Big) , \\
    \III^j_5 &\les \frac{1}{N} \sum_{k = 1}^N \Big( \| \Psi_\eps^j - \Psi_0^j \|_{C_T W_x^{s - 1, \infty}} \| v_\eps^{N, k} \|_{C_T H_x^{s'}}^2  \\
    &\qquad + \ld^{- \frac 12} \| \Psi_0^j \|_{C_T W_x^{s - 1, \infty}} \| v_\eps^{N, k} - v^{N, k} \|_{S_T^{s', \ld}} \big( \| v_\eps^{N, k} \|_{C_T H_x^{s'}} + \| v^{N, k} \|_{C_T H_x^{s'}} \big) \Big) , \\
    \III^j_6 &\les \frac{\ld^{-\frac 12}}{N} \sum_{k = 1}^N \Big( \| v_\eps^{N, k} - v^{N, k} \|_{S_T^{s', \ld}} \big( \| v_\eps^{N, k} \|_{C_T H_x^{s'}} + \| v^{N, k} \|_{C_T H_x^{s'}} \big) \| v_\eps^{N, j} \|_{C_T H_x^{s'}} \\
    &\qquad + \| v^{N, k} \|_{C_T H_x^{s'}}^2 \| v_\eps^{N, j} - v^{N, j} \|_{S_T^{s', \ld}}  \Big) ,
\end{align*}

\noi
provided that $1 - s' \leq s'$, $\frac{s' - (1 - s')}{2} \geq \frac 14$, and $\frac{s'}{2} \geq \frac 13$ which are satisfied under the assumption $\frac 45 < s' < s < 1$. 
Thus, from the Cauchy-Schwarz inequalities in $k$ and Young's inequalities, we obtain
\begin{align}
\begin{split}
    &\frac{1}{N} \sum_{j = 1}^N \big( (\III^j_1)^2 + (\III^j_2)^2 + (\III^j_3)^2 + (\III^j_4)^2 + (\III^j_5)^2 + (\III^j_6)^2 \big) \\
    &\les \Big( \ld^{- 1} \| \vv_\eps^{N} - \vv^N \|_{\A_N S_T^{s', \ld}}^2 + \| \mathbf{\Psi}_\eps^N - \mathbf{\Psi}_0^N \|_{\mathcal{Z}_T^{s - 1, N}}^2 \Big) \\
    &\quad \times \Big( 1 + \| \mathbf{\Psi}_\eps^N \|_{\mathcal{Z}_T^{s - 1, N}}^4 + \| \mathbf{\Psi}_0^N \|_{\mathcal{Z}_T^{s - 1, N}}^4 + \| \vv_\eps^N \|_{\A_N C_T H_x^{s'}}^4 + \| \vv^N \|_{\A_N C_T H_x^{s'}}^4 \Big) .
\end{split}
\label{conveN3}
\end{align}

\noi
Combining \eqref{conveN0}, \eqref{conveN1}, \eqref{conveN2}, and \eqref{conveN3} and using \eqref{PsiN_bdd} and \eqref{PsiN_diff}, we obtain
\begin{align}
\begin{split}
    \| &\vv_\eps^{N} - \vv^N \|_{\A_N S_T^{s', \ld}}^2 \\
    &\leq C \eps^{s - s'} \big\| (\uu_0^{N}, \uu_1^{N}) \big\|_{\A_N \H^s}^2 + C C(\om, N, T)^6 \eps^{s - s'} \Big( 1 + \| \vv_\eps^N \|_{\A_N C_T H_x^{s'}}^6 \Big) \\
    &\quad + C C(\om, N, T)^6 \eps^{2 \al} \Big( 1 + \| \vv_\eps^N \|_{\A_N C_T H_x^{s'}}^4 + \| \vv^N \|_{\A_N C_T H_x^{s'}}^4 \Big) \\
    &\quad + C C(\om, N, T)^4 \ld^{- 1} \| \vv_\eps^{N} - \vv^N \|_{\A_N S_T^{s', \ld}}^2 \Big( 1 + \| \vv_\eps^N \|_{\A_N C_T H_x^{s'}}^4 + \| \vv^N \|_{\A_N C_T H_x^{s'}}^4 \Big) 
\end{split}
\label{veN_diff1}
\end{align}

\noi
for some constant $C \geq 1$ that varies from line to line.

From the bound \eqref{vNTbdd} in Proposition~\ref{PROP:NLHN}, we know that
\begin{align}
    \| \vv^N \|_{\A_N C_T H_x^{s'}} \leq C(\om, N, T, \| \uu_0^N \|_{\A_N H^s}) 
\label{vNTbdd1}
\end{align}

\noi
for some constant $C(\om, N, T, \| \uu_0^N \|_{\A_N H^s}) > 0$.
Let us assume that 
\begin{align}
    \| \vv_\eps^N \|_{\A_N C_T H_x^{s'}} \leq 4 C(\om, N, T, \| ( \uu_0^N, \uu_1^N ) \|_{\A_N \H^s} ) ,
\label{veNTbdd1}
\end{align}

\noi
where the constant satisfies $C(\om, N, T, \| ( \uu_0^N, \uu_1^N ) \|_{\A_N \H^s} ) \geq C(\om, N, T, \| \uu_0^N \|_{\A_N H^s})$.
Then, from \eqref{veN_diff1}, \eqref{vNTbdd1}, and \eqref{veNTbdd1}, we obtain
\begin{align}
\begin{split}
    \| \vv_\eps^{N} - \vv^N \|_{\A_N S_T^{s', \ld}}^2 &\leq C C(\om, N, T, \| ( \uu_0^N, \uu_1^N ) \|_{\A_N \H^s} )^{6} \eps^{2 \be} \\
    &\quad + C C(\om, N, T, \| ( \uu_0^N, \uu_1^N ) \|_{\A_N \H^s} )^{4} \ld^{- 1} \| \vv_\eps^{N} - \vv^N \|_{\A_N S_T^{s', \ld}}^2 
\end{split}
\label{veN_diff2}
\end{align}

\noi
for $\be = \min (\frac{s - s'}{2}, \al) > 0$. Using \eqref{SlT_bdd}, \eqref{veN_diff2}, and taking $\ld = \ld (\om, N, T, \| ( \uu_0^N, \uu_1^N ) \|_{\A_N \H^s}) > 0$ to be sufficiently large, we get
\begin{align}
\begin{split}
    \| \vv_\eps^{N} - \vv^N \|_{\A_N C_T H_x^{s'}} &\leq e^{\ld T} \| \vv_\eps^{N} - \vv^N \|_{\A_N S_T^{s', \ld}} \\
    &\leq 2 C e^{\ld T} C(\om, N, T, \| ( \uu_0^N, \uu_1^N) \|_{\A_N \H^s} )^{3}  \eps^{\be} .
\end{split}
\label{veN_diff3}
\end{align}

In view of \eqref{veN_diff3}, we know that the desired almost sure convergence of $\vv_\eps^N$ to $\vv^N$ follows once we show the bound \eqref{veNTbdd1} for all $0 < \eps \leq \eps_0$ for some small $\eps_0 > 0$. We assume for the sake of contradiction that there exists $0 < T_* < T$ such that $T_*$ is the largest number such that
\begin{align*}
\| \vv_\eps^N \|_{\A_N C_{T_*} H_x^{s'}} \leq 2 C(\om, N, T, \| ( \uu_0^N, \uu_1^N ) \|_{\A_N \H^s}) .
\end{align*}

\noi
Then, by continuity, we know that there exists $T_0 > 0$ such that $0 < T_* + T_0 \leq T$ and
\begin{align*}
    \| \vv_\eps^N \|_{\A_N C_{T_* + T_0} H_x^{s'}} \leq 4 C(\om, N, T, \| ( \uu_0^N, \uu_1^N ) \|_{\A_N \H^s} ) .
\end{align*}

\noi
This satisfies the assumption \eqref{veNTbdd1}, so that the above argument that leads to \eqref{veN_diff3} gives
\begin{align}
    \| \vv_\eps^N - \vv^N \|_{\A_N C_{T_* + T_0} H_x^{s'}} \leq 2 C e^{\ld T} C(\om, N, T, \| ( \uu_0^N, \uu_1^N ) \|_{\A_N \H^s})^3 \eps^{\be} .
\label{veN_diff4}
\end{align}

\noi
Thus, by setting $\eps_0 = \eps_0 (\om, N, T, s, s', \al, \| ( \uu_0^N, \uu_1^N ) \|_{\A_N \H^s}) > 0$ to be sufficiently small, we obtain from \eqref{veN_diff4} and \eqref{vNTbdd1} that for any $0 < \eps \leq \eps_0$,
\begin{align*}
    \| \vv_\eps^N \|_{\A_N C_{T_* + T_0} H_x^{s'}} &\leq \| \vv_\eps^N - \vv^N \|_{\A_N C_{T_* + T_0} H_x^{s'}} + \| \vv^N \|_{\A_N C_T H_x^{s'}} \\
    &\leq 2 C(\om, N, T, \| ( \uu_0^N, \uu_1^N ) \|_{\A_N \H^s} ) ,
\end{align*}

\noi
which contradicts the assumption on $T_*$. Thus, we must have $T_* = T$, so that the assumption \eqref{veNTbdd1} must hold. This finishes the proof of the convergence.
\end{proof}

\section{Convergence via regime II}
\label{SEC:regII}

In this section, we prove Theorem~\ref{THM:conv2}, convergence of \HLSe \eqref{NLWNeu} to the mean-field SNLH \eqref{NLHu} via regime II.

As mentioned earlier, part (i) and part (ii) of Theorem~\ref{THM:conv2}, global well-posedness for the mean-field \NLWe and mean-field convergence of the $\text{HLSM}_{\eps, N}$, have been essentially established by \cite[Theorem~1.2]{LLO} and \cite[Theorem~1.6]{LLO}, respectively. Thus, we focus on proving part (iii) of Theorem~\ref{THM:conv2}, convergence of the mean-field \NLWe \eqref{NLWeu} to the mean-field SNLH \eqref{NLHu}.
As mentioned in the introduction, we first establish a uniform-in-$\eps$ bound for the solution in Subsection~\ref{SUB:mfbdd} and then show the global-in-time convergence of the dynamics in Subsection~\ref{SUB:conv2}.

\subsection{Uniform a priori bound for the mean-field $\text{SdNLW}_\eps$}
\label{SUB:mfbdd}

In this section, we establish a uniform-in-$\eps$ a priori bound for $v_\eps^j$ satisfying the perturbed mean-field \NLWe \eqref{NLWev}.
Let us drop the superscript $j$ in this subsection for simplicity. The Duhamel formulation of \eqref{NLWev} is given by
\begin{align}
    v_\eps (t) = P_\eps (t) (u_0, u_1) - \mathcal{I}_\eps \big( 2 \E [ \Psi_\eps v_\eps ] \Psi_\eps + 2 \E [\Psi_\eps v_\eps] v_\eps + \E [(v_\eps)^2] \Psi_\eps + \E [(v_\eps)^2] v_\eps \big) (t) ,
\label{NLWveDuh}
\end{align}

\noi
where $P_\eps$ is defined in \eqref{defPe}, $\mathcal{I}_\eps$ is defined in \eqref{defIe}, and $(u_0, u_1)$ is the initial data for \eqref{NLWev}. Our goal is to prove the following global-in-time a priori bound for $v_\eps$ uniformly in $\eps$. 
\begin{proposition}
\label{PROP:NLW}
Let $0 < \eps \leq 1$, $\frac 45 < s < 1$, and $T \geq 1$. Let $(u_0, u_1) \in L^2 ( \Omega; \H^s (\T^2) )$ and let $v_\eps$ be the unique global-in-time solution to the equation \eqref{NLWveDuh} in $L^2 (\Om; C([0, T]; H^s (\T^2)))$ with initial data $(u_0, u_1)$. Then, we have the bound
\begin{align}
    \| v_\eps \|_{L^2_\om C_T H_x^s} \leq C (T, s, \| ( u_0, u_1 ) \|_{L^2_\om \H_x^s} ) 
\label{vebdd_goal}
\end{align}

\noi
for some constant $C (T, s, \| ( u_0, u_1 ) \|_{L^2_\om \H_x^s})$ independent of $\eps$.
\end{proposition}

We recall the $I_M$-operator defined in \eqref{defIop}. Recalling the energy functional $E_\eps$ in \eqref{Eew}, in view of \eqref{Ibdd1}, we see that it suffices to provide an upper bound for the following energy:
\begin{align}
\begin{split}
    E_{\eps, M} (t) &= E_\eps ( I_M v_\eps (t), \dt I_M v_\eps (t) ) \\
    &=\frac 12 \E \bigg[ \int_{\T^2} (\eps \dt I_M v_\eps (t))^2 + (I_M v_\eps (t))^2 + |\nb I_M v_\eps (t) |^2 dx \bigg] \\
    &\quad + \frac 14 \int_{\T^2} \E [ (I_M v_\eps (t))^2 ]^2 dx .
\end{split}
\label{defEe}
\end{align}

From \eqref{NLWev}, given $0 \leq t_0 \leq t \leq T$, we compute that
\begin{align}
\begin{split}
E_{\eps, M} &(t) - E_{\eps, M} (t_0) + \E \bigg[ \int_{t_0}^t \int_{\T^2} (\dt I_M v_\eps)^2 dx dt' \bigg] \\
&= \E \bigg[ \int_{t_0}^t \int_{\T^2} \dt I_M v_\eps \Big( \E [ (I_M v_\eps)^2 ] I_M v_\eps - I_M \big( \E [v_\eps^2] v_\eps \big) \Big) dx dt' \bigg] \\
&\quad - \E \bigg[ \int_{t_0}^t \int_{\T^2} \dt I_M v_\eps \Big( I_M \big( \E [v_\eps^2] \Psi_\eps \big) + 2 I_M \big( \E [\Psi_\eps v_\eps] v_\eps \big) \Big) dx dt' \bigg] \\
&\quad - 2 \E \bigg[ \int_{t_0}^t \int_{\T^2} (\dt I_M v_\eps) I_M \big( \E [ \Psi_\eps v_\eps ] \Psi_\eps \big) dx dt' \bigg] \\
&\deff A_1 - A_2 - A_3 .
\end{split}
\label{defB123}
\end{align}

\noi
We also define
\begin{align}
    A_4 \deff \E \bigg[ \int_{t_0}^t \int_{\T^2} \dt I_M v_\eps \Big( \E [ (I_M v_\eps)^2 ] I_M \Psi_\eps + 2 \E [ I_M \Psi_\eps I_M v_\eps ] I_M v_\eps \Big) dx dt' \bigg] .
\label{defB4}
\end{align}

In the following series of lemmas, we estimate the terms $A_1$, $A_2$, $A_3$, and $A_4$.

\begin{lemma}
\label{LEM:B1}
Let $\frac 23 \leq s < 1$, $M \in \N$, and $0 \leq t_0 \leq t \leq T$. Then, for $A_1$ defined in \eqref{defB123}, we have
\begin{align*}
    |A_1| \leq \frac 14 \E \bigg[ \int_{t_0}^t \int_{\T^2} (\dt I_M v_\eps)^2 dx dt' \bigg] + C(s) M^{- 6 s + 4} \int_{t_0}^t E_{\eps, M} (t')^3 dt' 
\end{align*}

\noi
for some constant $C(s) > 0$ independent of $\eps$.
\end{lemma}

\begin{proof}
By Minkowski's integral inequality and Lemma~\ref{LEM:I_est1}, we have
\begin{align}
\begin{split}
    \E &\Big[ \big\| \E [ (I_M v_\eps)^2 ] I_M v_\eps - I_M \big( \E [v_\eps^2] v_\eps \big) \big\|_{L_x^2}^2 \Big] \\
    &\leq \int_\Om \bigg( \int_\Om \big\| (I_M v_\eps (\om'))^2 I_M v_\eps (\om) - I_M \big( v_\eps (\om')^2 v_\eps (\om) \big) \big\|_{L^2_x} d \om' \bigg)^2 d \om \\
    &\les_s \int_\Om \bigg( \int_\Om M^{-3s + 2} \| I_M v_\eps (\om') \|_{H^1_x}^2 \| I_M v_\eps (\om) \|_{H^1_x} d \om' \bigg)^2 d \om \\
    &= M^{-6s + 4} \| I_M v_\eps \|_{L_\om^2 H_x^1}^6 .
\end{split}
\label{B1-1}
\end{align}

\noi
Thus, by Cauchy's inequality, \eqref{B1-1}, and \eqref{defEe}, we obtain
\begin{align*}
    |A_1| &\leq \frac 14 \E \bigg[ \int_{t_0}^t \int_{\T^2} (\dt I_M v_\eps)^2 dx dt' \bigg] + \int_{t_0}^t \E \Big[ \big\| \E [ (I_M v_\eps)^2 ] I_M v_\eps - I_M \big( \E [v_\eps^2] v_\eps \big) \big\|_{L_x^2}^2 \Big] dt' \\
    &\leq \frac 14 \E \bigg[ \int_{t_0}^t \int_{\T^2} (\dt I_M v_\eps)^2 dx dt' \bigg] + C (s) M^{- 6 s + 4} \int_{t_0}^t E_{\eps, M} (t')^3 dt' 
\end{align*}

\noi
for some constant $C (s) > 0$, as desired.
\end{proof}

\begin{lemma}
\label{LEM:B24}
Let $\frac 23 \leq s < 1$, $M \in \N$ with $M \geq 10$, $0 < \nu \leq 1 - s$, and $0 \leq t_0 \leq t \leq T$. Let $\s_0 = \s_0 (\nu) > 0$ and $p_0 = p_0 (\nu) > 1$ be given as in Lemma~\ref{LEM:I_est2}. Then, for $A_2$ defined in \eqref{defB123} and $A_4$ defined in \eqref{defB4}, we have
\begin{align*}
    |A_2 - A_4| \leq \frac 14 \E \bigg[ \int_{t_0}^t \int_{\T^2} (\dt I_M v_\eps)^2 dx dt' \bigg] + C(\nu, s) M^{- 2s + 1 + 2 \nu} \int_{t_0}^t E_{\eps, M} (t')^2 dt' 
\end{align*}

\noi
for some constant $C (\nu, s) > 0$ independent of $\eps$.
\end{lemma}

\begin{proof}
By Minkowski's integral inequality, \eqref{I_est3} in Lemma~\ref{LEM:I_est2},
and the Cauchy-Schwarz inequality in $\om'$, we have
\begin{align}
\begin{split}
    \E &\Big[ \big\| I_M \big( \E [ \Psi_\eps v_\eps ] v_\eps \big) - \E [ I_M \Psi_\eps  I_M v_\eps ] I_M v_\eps \big\|_{L_x^2}^2 \Big] \\
    &\leq \int_\Om \bigg( \int_\Om \big\| I_M \big( \Psi_\eps (\om') v_\eps (\om') v_\eps (\om) \big) - I_M \Psi_\eps (\om') I_M v_\eps (\om') I_M v_\eps (\om) \big\|_{L_x^2} d \om' \bigg)^2 d \om \\
    &\les M^{- 2s + 1 + 2 \nu} \int_\Om \bigg( \int_\Om \| \Psi_\eps (\om') \|_{W_x^{- \s_0, p_0}} \| I_M v_\eps (\om') \|_{H_x^1} \| I_M v_\eps (\om) \|_{H_x^1} d \om' \bigg)^2 d \om \\
    &\les M^{- 2s + 1 + 2 \nu} \| I_M v_\eps \|_{L_\om^2 H_x^1}^4 \| \Psi_\eps \|_{L_\om^2 W_x^{- \s_0, p_0}}^2 .
\end{split}
\label{B24-1}
\end{align}

\noi
Using similar steps, we obtain
\begin{align}
    \E \Big[ \big\| I_M \big( \E [ v_\eps^2 ] \Psi_\eps \big) - \E [ ( I_M v_\eps )^2 ] I_M \Psi_\eps \big\|_{L_x^2}^2 \Big] \les M^{- 2s + 1 + 2 \nu} \| I_M v_\eps \|_{L_\om^2 H_x^1}^4 \| \Psi_\eps \|_{L_\om^2 W_x^{- \s_0, p_0}}^2 .
\label{B24-2}
\end{align}

\noi
Thus, by Cauchy's inequality, \eqref{B24-2}, \eqref{B24-1}, \eqref{defEe}, and the moment bound in Lemma~\ref{LEM:sto}, we get
\begin{align*}
    |&A_2 - A_4| \\
    &\leq \frac 18 \E \bigg[ \int_{t_0}^t \int_{\T^2} (\dt I_M v_\eps)^2 dx dt' \bigg] 
    + 2 \int_{t_0}^t \E \Big[ \big\| I_M \big( \E [ v_\eps^2 ] \Psi_\eps \big) - \E [ ( I_M v_\eps )^2 ] I_M \Psi_\eps \big\|_{L_x^2}^2 \Big] dt' \\
    &\quad + \frac 18 \E \bigg[ \int_{t_0}^t \int_{\T^2} (\dt I_M v_\eps)^2 dx dt' \bigg] 
    + 8 \int_{t_0}^t \E \Big[ \big\| I_M \big( \E [ \Psi_\eps v_\eps ] v_\eps \big) - \E [ I_M \Psi_\eps I_M v_\eps ] I_M v_\eps \big\|_{L_x^2}^2 \Big] dt' \\
    &\leq \frac 14 \E \bigg[ \int_{t_0}^t \int_{\T^2} (\dt I_M v_\eps)^2 dx dt' \bigg] + C_1(\nu, s) M^{- 2s + 1 + 2 \nu} \| \Psi_\eps \|_{L_\om^2 C_{[t_0, t]} W_x^{- \s_0, p_0}}^2 \int_{t_0}^t E_{\eps, M} (t')^2 dt' \\
    &\leq \frac 14 \E \bigg[ \int_{t_0}^t \int_{\T^2} (\dt I_M v_\eps)^2 dx dt' \bigg] + C_2(\nu, s) M^{- 2s + 1 + 2 \nu} \int_{t_0}^t E_{\eps, M} (t')^2 dt'
\end{align*}

\noi
for some constants $C_1 (\nu, s), C_2 (\nu, s) > 0$, as desired.
\end{proof}

\begin{lemma}
\label{LEM:B3}
Let $\frac 23 \leq s < 1$, $0 < \gamma \leq 1 - s$, $M \in \N$, and $0 \leq t_0 \leq t \leq T$. Then, for $A_3$ defined in \eqref{defB123}, we have
\begin{align*}
    |A_3| \leq \frac 14 \E \bigg[ \int_{t_0}^t \int_{\T^2} (\dt I_M v_\eps)^2 dx dt' \bigg] + C(\gamma) M^{2 \gamma} \int_{t_0}^t E_{\eps, M} (t')^{1 - \frac{s}{2} + \frac{\gamma}{2}} dt'
\end{align*}

\noi
for some constant $C (\gamma) > 0$ independent of $\eps$.
\end{lemma}

\begin{proof}
We proceed as in \cite[Lemma~3.1]{SSZZ} by writing
\begin{align*}
    \E [\Psi_\eps v_\eps] \Psi_\eps = \E [ v_\eps' \Psi_\eps' \Psi_\eps | \Psi_\eps ] ,
\end{align*}

\noi
where $(\Psi_\eps', v_\eps')$ is an independent copy of $(\Psi_\eps, v_\eps)$. In particular, the product $\Psi_\eps' \Psi_\eps$ makes sense without renormalization and Lemma~\ref{LEM:sto} applies to $\Psi_\eps' \Psi_\eps$. Thus, by \eqref{Ibdd2}, Jensen's inequality, the product estimate in Lemma~\ref{LEM:prod}~(ii), and conditional H\"older's inequality, we have
\begin{align}
\begin{split}
    \big\| I_M \big( \E [\Psi_\eps v_\eps] \Psi_\eps \big) \big\|_{L_\om^2 L_x^2} &= \big\| I_M \big( \E [v_\eps' \Psi_\eps' \Psi_\eps | \Psi_\eps ] \big) \big\|_{L_\om^2 L_x^2} \\
    &\les M^\gamma \big\| \E \big[ \| v_\eps' \Psi_\eps' \Psi_\eps \|_{H_x^{- \gamma}} | \Psi_\eps \big] \big\|_{L_\om^2} \\
    &\les_\gamma M^\gamma \big\| \E \big[ \| v_\eps' \|_{H_x^\gamma} \| \Psi'_\eps \Psi_\eps \|_{W_x^{- \gamma, \infty}} | \Psi_\eps \big] \big\|_{L_\om^2} \\
    &\leq M^\gamma \| v_\eps \|_{L_\om^2 H_x^\gamma} \| \Psi_\eps' \Psi_\eps \|_{L_\om^2 W_x^{- \gamma, \infty}} .
\end{split}
\label{B3-1}
\end{align}

\noi
For the term $\| v_\eps \|_{L_\om^2 H_x^\gamma}$, we use \eqref{Ibdd1}, interpolation, H\"older's inequality in $x$, and \eqref{defEe} to obtain
\begin{align}
\begin{split}
    \| v_\eps \|_{L_\om^2 H_x^\gamma} &\les \| I_M v_\eps \|_{L_\om^2 H_x^{\gamma + 1 - s}} \\
    &\les \| I_M v_\eps \|_{L_\om^2 H_x^1}^{\gamma + 1 - s} \| I_M v_\eps \|_{L_\om^2 L_x^2}^{s - \gamma} \\
    &\leq \| I_M v_\eps \|_{L_\om^2 H_x^1}^{\gamma + 1 - s} \| I_M v_\eps \|_{L_x^4 L_\om^2}^{s - \gamma} \\
    &\les E_{\eps, M}^{\frac 12 - \frac s4 + \frac{\gamma}{4}} .
\end{split}
\label{B3-2}
\end{align}

\noi
Thus, by Cauchy's inequality, \eqref{B3-1}, \eqref{B3-2}, and the moment bound in Lemma~\ref{LEM:sto}, we get
\begin{align*}
    |A_3| &\leq \frac 14 \E \bigg[ \int_{t_0}^t \int_{\T^2} (\dt I_M v_\eps)^2 dx dt' \bigg] + 4 \int_{t_0}^t \big\| I_M \big( \E [\Psi_\eps v_\eps] \Psi_\eps \big) \big\|_{L_\om^2 L_x^2}^2 dt' \\
    &\leq \frac 14 \E \bigg[ \int_{t_0}^t \int_{\T^2} (\dt I_M v_\eps)^2 dx dt' \bigg] + C_1 (\gamma) M^{2 \gamma} \| \Psi_\eps' \Psi_\eps \|_{L_\om^2 C_{[t_0, t]} W_x^{- \gamma, \infty}}^2 \int_{t_0}^t E_{\eps, M} (t')^{1 - \frac s2 + \frac{\gamma}{2}} dt' \\
    &\leq \frac 14 \E \bigg[ \int_{t_0}^t \int_{\T^2} (\dt I_M v_\eps)^2 dx dt' \bigg] + C_2 (\gamma) M^{2 \gamma}  \int_{t_0}^t E_{\eps, M} (t')^{1 - \frac s2 + \frac{\gamma}{2}} dt'
\end{align*}

\noi
for some constants $C_1 (\gamma), C_2 (\gamma) > 0$, as desired.
\end{proof}

In the following lemma, we rely crucially on the fact that we are working with the mean-field nonlinearity. Specifically, we ensure a loss not worse than $\log M$ in our estimates.
If the bound were to grow faster than $\log M$, the subsequent Gronwall-type argument would fail, yielding a singular estimate.
Fortunately, the mean-field nonlinearity guarantees that the norm of $I_M \Psi_\eps$ remains well-behaved, growing at most like $\log M$. 
One can compare this situation to \cite{GKOT}, where the authors also ensure a loss of $\log M$ when working with the usual cubic nonlinearity; however, in our setting, we must apply an additional Cauchy-Schwarz inequality in order to guarantee uniformity in $\eps$.
Because of this additional estimate, we are unable to obtain a $\log M$ bound in the $N$-component setting. Consequently, for the $N$-component system, the subsequent proof breaks down, preventing us from finding a uniform-in-$\eps$ bound for the solution.

\begin{lemma}
\label{LEM:B4}
Let $0 < s < 1$, $M \in \N$ with $M \geq 2$, and $0 \leq t_0 \leq t \leq T$. Then, for $A_4$ defined in \eqref{defB4}, we have
\begin{align*}
    |A_4| \leq \frac 14 \E \bigg[ \int_{t_0}^t \int_{\T^2} (\dt I_M v_\eps)^2 dx dt' \bigg] + C(s) \log M \int_{t_0}^t E_{\eps, M} (t') dt'
\end{align*}

\noi
for some constant $C(s) > 0$ independent of $\eps$.
\end{lemma}

\begin{proof}
By the Cauchy-Schwarz inequality in $\om$, H\"older's inequality in $x$, Cauchy's inequality, and Lemma~\ref{LEM:stoI}, we obtain
\begin{align*}
    |A_4| &\leq 3 \int_{t_0}^t \int_{\T^2} \| \dt I_M v_\eps \|_{L_\om^2} \| I_M v_\eps \|_{L_\om^2}^2 \| I_M \Psi_\eps \|_{L_\om^2} dx dt' \\
    &\leq 3 \int_{t_0}^t \| \dt I_M v_\eps \|_{L_x^2 L_\om^2} \big\| \E [(I_M v_\eps)^2] \big\|_{L_x^{2}} \| I_M \Psi_\eps \|_{L_x^{\infty} L_\om^2} dt' \\
    &\leq \frac 14 \E \bigg[ \int_{t_0}^t \int_{\T^2} (\dt I_M v_\eps)^2 dx dt' \bigg] + C (s) \log M \int_{t_0}^t E_{\eps, M} (t') dt'
\end{align*}

\noi
for some constant $C (s) > 0$, as desired.
\end{proof}

Gathering \eqref{defB123}, Lemma~\ref{LEM:B1}, Lemma~\ref{LEM:B24}, Lemma~\ref{LEM:B3}, and Lemma~\ref{LEM:B4}, we obtain that given $0 \leq t_0 \leq t \leq T$ with $t - t_0 \leq 1$,
\begin{align}
\begin{split}
    &E_{\eps, M} (t) - E_{\eps, M} (t_0) \\
    &\leq C(s) M^{- 6 s + 4} \int_{t_0}^t E_{\eps, M} (t')^3 dt' + C (\nu, s) M^{- 2 s + 1 + 2 \nu} \int_{t_0}^t E_{\eps, M} (t')^2 dt' \\
    &\quad + C (s) \log M \int_{t_0}^t E_{\eps, M} (t') dt' + C (\gamma) M^{2 \gamma} \int_{t_0}^t E_{\eps, M} (t')^{1 - \frac{s}{2} + \frac{\gamma}{2}} dt'
\end{split}
\label{Ediff}
\end{align}

\noi
for any $\frac 23 \leq s < 1$, $M \in \N$ with $M \geq 10$, $0 < \nu \leq 1 - s$, and $0 < \gamma \leq 1 - s$ and some constants $C (s) > 0$, $C(\nu, s) > 0$, $C(\gamma) > 0$ independent of $\eps$.

Let us now show the following iterative lemma, for which the idea goes back to \cite{Tol, GKOT}.

\begin{lemma}
\label{LEM:ite}
Let $0 < \eps \leq 1$, $\frac 23 < s < 1$, $0 < \be < \al < -2 + 3 s$, $M \in \N$ with $M \geq 10$, and $T \geq 1$. Then, there exists $0 < \tau_* = \tau_* (s, \al, \be) \leq 1$ such that if 
\begin{align}
    E_{\eps, M} (t_0) \leq M^\be
\label{Ee_cond}
\end{align}

\noi
for some $0 \leq t_0 < T$, then we have
\begin{align}
    E_{\eps, M} (t) \leq M^\al
\label{Ee_goal}
\end{align}

\noi
for any $t$ satisfying $t_0 \leq t \leq \min (T, t_0 + \tau_*)$.
\end{lemma}

\begin{proof}
By replacing $E_{\eps, M}$ with $E_{\eps, M} + 1$, we may assume that $E_{\eps, M} (t) \geq 1$ for any $t \geq 0$. From \eqref{Ee_cond} and continuity, we know that there exists $t_1$ with $t_0 < t_1 \leq T$ such that
\begin{align}
    E_{\eps, M} (t) \leq 2 M^\al
\label{Ee1}
\end{align}

\noi
for any $t_0 \leq t \leq t_1$. We also assume that $t_1 - t_0 \leq 1$.
From \eqref{Ee1} and \eqref{Ediff} with $\nu = \frac{1 - s}{2}$ and $\gamma = \min(\frac{\be}{6}, -\frac 23 + s, 1 - s)$, we get
\begin{align}
\begin{split}
    E_{\eps, M} (t) - E_{\eps, M} (t_0) &\les_{s, \al, \be} M^{- 6 s + 4} \int_{t_0}^t E_{\eps, M} (t')^3 dt' + M^{- 3 s + 2} \int_{t_0}^t E_{\eps, M} (t')^2 dt' \\
    &\quad + \log M \int_{t_0}^t E_{\eps, M} (t') dt' + M^{\frac{\be}{3}} \int_{t_0}^t E_{\eps, M} (t')^{\frac 23} dt'
\end{split}
\label{Ee2}
\end{align}

\noi
for any $t_0 \leq t \leq t_1$. 

Given $t_0 \leq t \leq t_1$, we now define
\begin{align}
    F_{\eps, M} (t) = \max_{t_0 \leq \tau \leq t} E_{\eps, M} (\tau) - E_{\eps, M} (t_0) + M^\be .
\label{defFt}
\end{align}

\noi
From the definition and the fact that $\be < \al$, we have
\begin{align}
\begin{split}
    &E_{\eps, M} (t) \leq F_{\eps, M} (t), \\
    &M^\be \leq F_{\eps, M} (t) \leq 3 M^\al, \\ 
    &\log F_{\eps, M} (t) \sim \log M ,
\end{split}
\label{Ft1}
\end{align}

\noi
which together with $0 < \be < \al \leq - 2 + 3 s$ implies that
\begin{align}
\begin{split}
    &M^{- 6 s + 4} F_{\eps, M} (t)^3 \leq M^{-2 \al} F_{\eps, M} (t)^3 \leq 9 F_{\eps, M} (t) , \\
    &M^{- 3 s + 2} F_{\eps, M} (t)^2 \leq M^{- \al} F_{\eps, M} (t)^2 \leq 3 F_{\eps, M} (t), \\
    &M^{\frac{\be}{3}} F_{\eps, M} (t)^{\frac 23} \leq F_{\eps, M} (t) .
\end{split}
\label{Ft2}
\end{align}

\noi
Thus, from \eqref{Ee2}, \eqref{defFt}, \eqref{Ft1}, and \eqref{Ft2}, we obtain
\begin{align}
\begin{split}
    F_{\eps, M} (t) - F_{\eps, M} (t_0) &= \max_{t_0 \leq \tau \leq t} \big( E_{\eps, M} (\tau) - E_{\eps, M} (t_0) \big) \\
    &\les_{s, \al, \be} \int_{t_0}^t F_{\eps, M} (t') (1 + \log M) dt' \\
    &\les \int_{t_0}^t F_{\eps, M} (t') (1 + \log F_{\eps, M} (t')) dt' .
\end{split}
\label{Ft3}
\end{align}

\noi
Note that the ODE
\begin{align*}
    G (t) = G(0) + C \int_0^t G (t') (1 + \log G (t')) dt'
\end{align*}

\noi
for some constant $C > 0$ has an explicit solution
\begin{align*}
    G (t) = \exp \big( e^{C t} (1 + \log G (0)) - 1 \big) .
\end{align*}

\noi
Thus, from \eqref{Ft3} and comparison, we get
\begin{align}
    F_{\eps, M} (t) \leq \exp \big( e^{C(s, \al, \be) (t - t_0)} (1 + \log F_{\eps, M} (t_0)) - 1 \big)
\label{Ft4}
\end{align}

\noi
for some constant $C (s, \al, \be) > 0$.

Since $F_{\eps, M} (t_0) = M^\be$, we know from \eqref{Ft1} and \eqref{Ft4} that 
\begin{align}
    E_{\eps, M} (t) \leq F_{\eps, M} (t) \leq \exp \big( e^{C (s, \al, \be) (t - t_0)} (1 + \be \log M) - 1 \big) .
\label{Ft5}
\end{align}

\noi
Since $\be < \al$, we can take $\tau_* = \tau_* (s, \al, \be) > 0$ sufficiently small such that for all $t_0 \leq t \leq \min (t_1, t_0 + \tau_*)$, we have
\begin{align*}
    e^{C (s, \al, \be) (t - t_0)} (1 + \be \log M) - 1 \leq \al \log M ,
\end{align*}

\noi
so that we get from \eqref{Ft5} that
\begin{align*}
    E_{\eps, M} (t) \leq M^\al ,
\end{align*}

\noi
which improves the bound \eqref{Ee1}. By using a standard continuity argument, we conclude the desired bound \eqref{Ee_goal} for all $t_0 \leq t \leq \min (T, t_0 + \tau_*)$.
\end{proof}

We are now ready to prove the a priori bound in Proposition~\ref{PROP:NLW}. The proof is similar to that in \cite{Tol, GKOT}, but for completeness we write out details.

\begin{proof}[Proof of Proposition~\ref{PROP:NLW}]
Let $M_0 \in \N$ be a large number to be chosen later. Given $0 < \be < \al < -2 + 3s$ and an integer $k \geq 0$, we define
\begin{align}
    M_k \deff M_0^{\ld^k}
\label{defMk}
\end{align}

\noi
for $\ld > 1$ in such a way that
\begin{align}
    M_{k + 1}^{2 (1 - s)} M_k^\al + M_k^{2 \al} \leq c M_{k + 1}^\be 
\label{Mk_cond}
\end{align}

\noi
for some small $c > 0$ to be chosen later. Note that for \eqref{Mk_cond} to hold, we require $\be > 2 (1 - s)$, and so the parameters $\al$ and $\be$ satisfying the above conditions exist given $\frac 45 < s < 1$. Suppose that for some $t \geq 0$ and some integer $k \geq 0$, we have
\begin{align}
    E_{\eps, M_k} (t) \leq M_k^\al .
\label{Eek_cond}
\end{align}

\noi
Then, by \eqref{defEe}, Minkowski's integral inequality, Sobolev's inequality, \eqref{Ibdd1}, and \eqref{defEe} again, we have
\begin{align*}
    E_{\eps, M_{k + 1}} (t) 
    &\leq \frac 12 \| I_{M_{k + 1}} v_\eps (t) \|_{L^2_\om H^1_x}^2 + \frac 12 \eps^2 \| \dt I_{M_{k + 1}} v_\eps (t) \|_{L^2_\om L^2_x}^2 + \frac 14 \| I_{M_{k + 1}} v_\eps (t) \|_{L^2_\om L_x^4}^4 \\
    &\leq \frac{C}{2} M_{k + 1}^{2 (1 - s)} \| v_\eps (t) \|_{L^2_\om H_x^{s}}^2 + \frac{C}{2} M_{k + 1}^{2 (1 - s)} \eps^2 \| \dt v_\eps (t) \|_{L^2_\om H_x^{s - 1}}^2 + \frac C4 \| v_\eps (t) \|_{L_\om^2 H_x^{s}}^4 \\
    &\leq \frac{C^2}{2} M_{k + 1}^{2 (1 - s)} \| I_{M_k} v_\eps (t) \|_{L^2_\om H_x^{1}}^2 + \frac{C^2}{2} M_{k + 1}^{2 (1 - s)} \eps^2 \| \dt I_{M_k} v_\eps (t) \|_{L^2_\om L_x^2}^2 \\
    &\quad + \frac{C^2}{4} \| I_{M_k} v_\eps (t) \|_{L_\om^2 H_x^{1}}^4 \\
    &\leq C^2 M_{k + 1}^{2 (1 - s)} E_{\eps, M_k} (t) + C^2 E_{\eps, M_k} (t)^2
\end{align*}

\noi
for some constant $C > 0$, so that from \eqref{Eek_cond} and \eqref{Mk_cond} with $c > 0$ sufficiently small such that $C^2 c \leq 1$, we get
\begin{align}
    E_{\eps, M_{k + 1}} (t) \leq M_{k + 1}^\be .
\label{Eek_cond2}
\end{align}

We now perform an iterative argument. Given initial data $(u_0, u_1) \in L^2 (\Om; \H^s (\T^2))$, from \eqref{defEe}, Minkowski's integral inequality, and Sobolev's inequality and the fact that $\be > 2 (1 - s)$, we can let $M_0 = M_0 (s, \| (u_0, u_1) \|_{L^2_\om \H_x^s}) \in \N$ be large enough such that
\begin{align}
\begin{split}
    E_{\eps, M_0} (0) &\leq \frac 12 \| I_{M_{0}} u_0 \|_{L^2_\om H^1_x}^2 + \frac 12 \eps^2 \| I_{M_{0}} u_1 \|_{L^2_\om L^2_x}^2 + \frac 14 \| I_{M_{0}} u_0 \|_{L^2_\om L_x^4}^4 \\
    &\leq \frac{C}{2} M_{0}^{2 (1 - s)} \| u_0 \|_{L^2_\om H_x^{s}}^2 + \frac{C}{2} M_{0}^{2 (1 - s)} \eps^2 \| u_1 \|_{L^2_\om H_x^{s - 1}}^2 + \frac C4 \| u_0 \|_{L_\om^2 H_x^{s}}^4 \\
    &\leq M_0^\be .
\end{split}
\label{Eet0}
\end{align}

\noi
From \eqref{Eet0} and Lemma~\ref{LEM:ite}, there exists $0 < \tau_* = \tau_* (s) \leq 1$ independent of $M_0$ such that
\begin{align*}
E_{\eps, M_0} (t) \leq M_0^\al 
\end{align*}

\noi
for all $0 \leq t \leq \tau_*$. Then, from \eqref{Eek_cond} and \eqref{Eek_cond2}, we get
\begin{align}
    E_{\eps, M_1} (t) \leq M_1^\be
\label{Eet2}
\end{align}

\noi
for all $0 \leq t \leq \tau_*$. From \eqref{Eet2} and Lemma~\ref{LEM:ite} again, we get
\begin{align*}
E_{\eps, M_1} (t) \leq M_1^{\al}
\end{align*}

\noi
for all $0 \leq t \leq 2 \tau_*$. We now iterate the above argument $\ell = [\frac{T}{\tau_*}] + 1$ times, so that we obtain
\begin{align*}
E_{\eps, M_\ell} (t) \leq M_\ell^\al 
\end{align*}

\noi
for all $0 \leq t \leq T$. Thus, from \eqref{Ibdd1}, \eqref{defEe}, and \eqref{defMk}, we obtain
\begin{align}
    \sup_{0 \leq t \leq T} \big\| \big( v_\eps (t), \eps \dt v_\eps (t) \big) \big\|_{L_\om^2 \H_x^s}^2 \leq 2 \sup_{0 \leq t \leq T} E_{\eps, M_\ell} (t) \leq 2 M_\ell^\al = 2 M_0^{\al \ld^\ell} .
\label{vebdd1}
\end{align}

\noi
The desired bound \eqref{vebdd_goal} then follows from \eqref{vebdd1} and a contraction argument for \eqref{NLWveDuh} similar to that in \cite[Proposition~4.1]{LLO} along with the linear estimates in Lemma~\ref{LEM:PIconv} and the uniform-in-$\eps$ moment bound in Lemma~\ref{LEM:sto}.
\end{proof}



\subsection{Proof of the convergence}
\label{SUB:conv2}

In this subsection, we prove part (iii) in Theorem~\ref{THM:conv2}, convergence of the mean-field \NLWe \eqref{NLWeu} to the mean-field SNLH \eqref{NLHu}.
As in the previous subsection, we drop the superscript $j$. To prove it, we need to show the convergence of $v_\eps$ satisfying 
\begin{align}
    v_\eps (t) = P_\eps (t) (u_0, u_1) - \mathcal{I}_\eps \big( 2 \E [ \Psi_\eps v_\eps ] \Psi_\eps + 2 \E [\Psi_\eps v_\eps] v_\eps + \E [(v_\eps)^2] \Psi_\eps + \E [(v_\eps)^2] v_\eps \big) (t)
\label{NLWveDuh2}
\end{align}

\noi
to $v$ satisfying
\begin{align}
    v (t) = P_0 (t) (u_0, u_1) - \mathcal{I}_0 \big( 2 \E [ \Psi_0 v ] \Psi_0 + 2 \E [\Psi_0 v] v + \E [v^2] \Psi_0 + \E [v^2] v \big) (t) ,
\label{NLHvDuh2}
\end{align}

\noi
where $P_\eps$ is defined in \eqref{defPe}, $\mathcal{I}_\eps$ is defined in \eqref{defIe}, $P_0$ is defined in \eqref{defP0} (see also \eqref{P001}), $\mathcal{I}_0$ is defined in \eqref{defI0}, and $(u_0, u_1)$ is the initial data for both equations. Our goal is to prove the following proposition, which, together with \eqref{expe}, \eqref{exp}, and the convergence $\Psi_\eps$ to $\Psi_0$ from Lemma~\ref{LEM:sto}, implies Theorem~\ref{THM:conv2}~(iii).
\begin{proposition}
Let $\frac 45 < s < 1$, $T > 0$, and $(u_0, u_1) \in L^4 (\Om; H^s (\T^2)) \times L^2 (\Om; H^{s - 1} (\T^2))$. Given any $0 < \eps \leq 1$, let $v_\eps \in L^2 (\Om; C ([0, T]; H^s (\T^2)))$ be the solution to \eqref{NLWveDuh2} with initial data $(u_0, u_1)$ guaranteed by Proposition~\ref{PROP:NLW}. Let $v \in L^2 (\Om; C([0, T]; H^s (\T^2)))$ be the solution to \eqref{NLHvDuh2} with initial data $u_0$ guaranteed by Proposition~\ref{PROP:NLH}. Then, for any $s' < s$, we have
\begin{align*}
    v_\eps \longrightarrow v \quad \text{in } L^2 (\Om; C ([0, T]; H^{s'} (\T^2)))
\end{align*}

\noi
as $\eps \to 0$.
\end{proposition}

\begin{proof}
We may assume that $\frac 45 < s' < s < 1$. From \eqref{vTbdd} in Proposition~\ref{PROP:NLH} and Proposition~\ref{PROP:NLW}, we have the following bound:
\begin{align}
    \| v (t) \|_{L^2_\om C_T H_x^s} + \sup_{\eps \in (0, 1]} \| v_\eps (t) \|_{L_\om^2 C_T H_x^s} \leq C (T, s, \| u_0 \|_{L_\om^4 H_x^{s}}, \| u_1 \|_{L_\om^2 H_x^{s - 1}})
\label{vvebdd}
\end{align}

\noi
for some constant $C (T, s, \| u_0 \|_{L_\om^4 H_x^{s}}, \| u_1 \|_{L_\om^2 H_x^{s - 1}}) > 0$ which also varies from line to line below.

We recall the $S_T^{\s, \ld}$-norm defined in \eqref{STsl}. From \eqref{NLWveDuh2}, \eqref{NLHvDuh2}, and \eqref{SlT_bdd}, we have
\begin{align}
\begin{split}
    \| &v_\eps - v \|_{L_\om^2 S_T^{s', \ld}} \\
    &\leq \| (P_\eps - P_0) (\cdot) (u_0, u_1) \|_{L_\om^2 C_T H_x^{s'}} \\
    &\quad + \big\| (\mathcal{I}_\eps - \mathcal{I}_0) \big( 2 \E [ \Psi_\eps v_\eps ] \Psi_\eps + 2 \E [\Psi_\eps v_\eps] v_\eps + \E [v_\eps^2] \Psi_\eps + \E [v_\eps^2] v_\eps \big) \big\|_{L_\om^2 C_T H_x^{s'}} \\
    &\quad + 2 \big\| \mathcal{I}_0 \big( \E [ \Psi_\eps v_\eps ] \Psi_\eps - \E [\Psi_0 v] \Psi_0 \big) \big\|_{L_\om^2 S_T^{s', \ld}} + 2 \big\| \mathcal{I}_0  \big( \E [\Psi_\eps v_\eps] v_\eps - \E [\Psi_0 v] v \big) \big\|_{L_\om^2 S_T^{s', \ld}} \\
    &\quad + \big\| \mathcal{I}_0 \big( \E [ (v_\eps)^2 ] \Psi_\eps - \E [v^2] \Psi_0 \big) \big\|_{L_\om^2 S_T^{s', \ld}} + \big\| \mathcal{I}_0  \big( \E [ (v_\eps)^2 ] v_\eps - \E [ v^2 ] v \big) \big\|_{L_\om^2 S_T^{s', \ld}} \\
    &\deff \1 + \II + \III_1 + \III_2 + \III_3 + \III_4 .
\end{split}
\label{conve0}
\end{align}

\noi
For $\1$, we use Lemma~\ref{LEM:PIconv} to obtain
\begin{align}
    \1 \les \eps^{\frac{s - s'}{2}} \| (u_0, u_1) \|_{L_\om^2 \H_x^{s}} .
\label{conve1}
\end{align}

\noi
For $\II$, by using Lemma~\ref{LEM:PIconv}, we have
\begin{align}
\begin{split}
    \II &\les T^{\frac 12} \eps^{\frac{s - s'}{2}} \big( \big\| \E [\Psi_\eps v_\eps] \Psi_\eps \big\|_{L_\om^2 C_T H_x^{s - 1}} + \big\| \E [\Psi_\eps v_\eps] v_\eps \big\|_{L_\om^2 C_T H_x^{s - 1}} \\
    &\qquad \qquad + \big\| \E [v_\eps^2] \Psi_\eps \big\|_{L_\om^2 C_T H_x^{s - 1}} + \big\| \E [v_\eps^2] v_\eps \big\|_{L_\om^2 C_T H_x^{s - 1}} \big) .
\end{split}
\label{conve2-0}
\end{align}

\noi
We need to estimate the four nonlinear terms on the right-hand side of \eqref{conve2-0}. As in \cite[Lemma~3.1]{SSZZ}, we write
\begin{align}
    \E [\Psi_\eps v_\eps] \Psi_\eps = \E [v_\eps' \Psi'_\eps \Psi_\eps | \Psi_\eps] ,
\label{conve2h}
\end{align}

\noi
where $(\Psi_\eps', v_\eps')$ is an independent copy of $(\Psi_\eps, v_\eps)$. By \eqref{conve2h}, Jensen's inequality, the product estimate in Lemma~\ref{LEM:prod}~(ii), and conditional H\"older's inequality, we get
\begin{align}
\begin{split}
    \big\| \E [\Psi_\eps v_\eps] \Psi_\eps \big\|_{L_\om^2 C_T H_x^{s - 1}} &\leq \big\| \E \big[ \| v_\eps' \Psi_\eps' \Psi_\eps \|_{C_T H_x^{s - 1}} | \Psi_\eps \big] \big\|_{L_\om^2} \\
    &\les \big\| \E \big[ \| v_\eps' \|_{C_T H_x^{1 - s}} \| \Psi_\eps' \Psi_\eps \|_{C_T W_x^{s - 1, \infty}} | \Psi_\eps \big] \big\|_{L_\om^2} \\
    &\leq \| v_\eps \|_{L_\om^2 C_T H_x^{s'}} \| \Psi_\eps' \Psi_\eps \|_{L_\om^2 C_T W_x^{s - 1, \infty}} ,
\end{split}
\label{conve2-1}
\end{align}

\noi
where we used $1 - s \leq s'$ given $\frac 45 < s' < s < 1$. By the product estimate in Lemma~\ref{LEM:prod}~(ii), Minkowski's integral inequality, Lemma~\ref{LEM:prod}~(ii) again, the Cauchy-Schwarz inequality in $\om$, and Sobolev's inequality, we get
\begin{align}
\begin{split}
    \big\| \E [\Psi_\eps v_\eps] v_\eps \big\|_{L_\om^2 C_T H_x^{s - 1}} &\les \big\| \E \big[ \| \Psi_\eps v_\eps \|_{C_T W_x^{s - 1, 4}} \big] \| v_\eps \|_{C_T W_x^{1 - s, 4}} \big\|_{L_\om^2} \\
    &\les \E \big[ \| \Psi_\eps \|_{C_T W_x^{s - 1, \infty}} \| v_\eps \|_{C_T W_x^{1 - s, 4}} \big] \| v_\eps \|_{L_\om^2 C_T W_x^{1 - s, 4}} \\
    &\les \| \Psi_\eps \|_{L_\om^2 C_T W_x^{s - 1, \infty}} \| v_\eps \|_{L_\om^2 C_T H_x^{s'}}^2 ,
\end{split}
\label{conve2-2}
\end{align}

\noi
where we used $\frac{s' - (1 - s)}{2} \geq \frac 14$ given $\frac 45 < s' < s < 1$. By the product estimate in Lemma~\ref{LEM:prod}~(ii), Minkowski's integral inequality, the product estimate in Lemma~\ref{LEM:prod}~(i), the Cauchy-Schwarz inequality in $\om$, and Sobolev's inequalities, we get
\begin{align}
\begin{split}
    \big\| \E [ v_\eps^2 ] \Psi_\eps \big\|_{L_\om^2 C_T H_x^{s - 1}} &\les \big\| \E \big[ \| v_\eps^2 \|_{C_T H_x^{1 - s}} \big] \| \Psi_\eps \|_{C_T W_x^{s - 1, \infty}} \big\|_{L_\om^2} \\
    &\les \E \big[ \| v_\eps \|_{C_T L_x^4} \| v_\eps \|_{C_T W_x^{1 - s, 4}} \big] \| \Psi_\eps \|_{L_\om^2 C_T W_x^{s - 1, \infty}} \\
    &\les \| v_\eps \|_{L_\om^2 C_T H_x^{s'}}^2 \| \Psi_\eps \|_{L_\om^2 C_T W_x^{s - 1, \infty}} ,
\end{split}
\label{conve2-3}
\end{align}

\noi
where we used $\frac{s' - (1 - s)}{2} \geq \frac 14$ given $\frac 45 < s' < s < 1$. By H\"older's inequality, Minkowski's integral inequality, and Sobolev's inequality, we get
\begin{align}
\begin{split}
    \big\| \E [v_\eps^2] v_\eps \big\|_{L_\om^2 C_T H_x^{s - 1}} &\leq \big\| \E [v_\eps^2] v_\eps \big\|_{L_\om^2 C_T L_x^2} \\
    &\leq \E \big[ \| v_\eps^2 \|_{C_T L_x^3} \big] \| v_\eps \|_{L_\om^2 C_T L_x^6} \\
    &\les \| v_\eps \|_{L_\om^2 C_T L_x^6}^3 \\
    &\les \| v_\eps \|_{L_\om^2 C_T H_x^{s'}}^3 ,
\end{split}
\label{conve2-4}
\end{align}

\noi
where we used $\frac{s'}{2} \geq \frac 13$ given $\frac 45 < s' < s < 1$. Combining \eqref{conve2-0}, \eqref{conve2-1}, \eqref{conve2-2}, \eqref{conve2-3}, and \eqref{conve2-4} and applying \eqref{vvebdd} and the moment bound in Lemma~\ref{LEM:sto} (which also applies to $\Psi_\eps' \Psi_\eps$), we obtain
\begin{align}
    \II \leq \eps^{\frac{s - s'}{2}} C (T, s, \| u_0 \|_{L_\om^4 H_x^{s}}, \| u_1 \|_{L_\om^2 H_x^{s - 1}}) .
\label{conve2}
\end{align}

We now estimate $\III_1$, $\III_2$, $\III_3$, and $\III_4$. 
By using \eqref{conve3h}, \eqref{SlT_bdd}, and similar steps in \eqref{conve2-1}, \eqref{conve2-2}, \eqref{conve2-3}, and \eqref{conve2-4}, we obtain
\begin{align}
\begin{split}
    \III_1 &\les \ld^{- \frac 12} \Big( \| v_\eps - v \|_{L_\om^2 S_T^{s', \ld}} \| \Psi_\eps' \Psi_\eps \|_{L_\om^2 C_T W_x^{s - 1, \infty}} \\
    &\qquad + \| v \|_{L_\om^2 C_T H_x^{s'}} \| \Psi_\eps' \Psi_\eps - \Psi_0' \Psi_0 \|_{L_\om^2 C_T W_x^{s' - 1, \infty}} \Big) , \\
    \III_2 + \III_3 &\les \ld^{- \frac 12} \Big( \| \Psi_\eps - \Psi_0 \|_{L_\om^2 C_T W_x^{s' - 1, \infty}} \| v_\eps \|_{L_\om^2 C_T H_x^{s'}}^2 \\
    &\qquad + \| \Psi_0 \|_{L_\om^2 C_T W_x^{s - 1, \infty}} \| v_\eps - v \|_{L_\om^2 S_T^{s', \ld}} \big( \| v_\eps \|_{L_\om^2 C_T H_x^{s'}} + \| v \|_{L_\om^2 C_T H_x^{s'}} \big) \Big) , \\
    \III_4 &\les \ld^{- \frac 12} \Big( \| v_\eps - v \|_{L_\om^2 S_T^{s', \ld}} \big( \| v_\eps \|_{L_\om^2 C_T H_x^{s'}}^2 + \| v \|_{L_\om^2 C_T H_x^{s'}}^2 \big)  \Big) 
\end{split}
\label{conve3-1}
\end{align}

\noi
provided that $1 - s' \leq s'$, $\frac{s' - (1 - s')}{2} \geq \frac 14$, and $\frac{s'}{2} \geq \frac 13$ which are satisfied given $\frac 45 < s' < 1$, where $\Psi_0'$ is an independent copy of $\Psi_0$. From \eqref{conve3-1}, \eqref{vvebdd}, and the moment bounds in Lemma~\ref{LEM:sto}, we obtain
\begin{align}
\begin{split}
    \III_1 + \III_2 + \III_3 + \III_4 
    &\leq C (T, s, \| u_0 \|_{L_\om^4 H_x^{s}}, \| u_1 \|_{L_\om^2 H_x^{s - 1}}) \Big( \eps^{\frac{1 - s}{2}} + \ld^{- \frac 12} \| v_\eps - v \|_{L_\om^2 S_T^{s', \ld}} \Big) .
\end{split}
\label{conve3}
\end{align}

Combining \eqref{conve0}, \eqref{conve1}, \eqref{conve2}, and \eqref{conve3}, we get
\begin{align*}
    \| v_\eps - v \|_{L_\om^2 S_T^{s', \ld}} &\leq C (T, s, \| u_0 \|_{L_\om^4 H_x^{s}}, \| u_1 \|_{L_\om^2 H_x^{s - 1}}) \Big( \eps^{\frac{s - s'}{2}} + \eps^{\frac{1 - s}{2}} + \ld^{- \frac 12} \| v_\eps - v \|_{L_\om^2 S_T^{s', \ld}} \Big)  .
\end{align*}

\noi
By choosing $\ld = \ld (T, s, \| (u_0, u_1) \|_{L_\om^2 H_x^s})$ to be sufficiently large and using \eqref{SlT_bdd}, we obtain
\begin{align*}
    \| v_\eps - v \|_{L_\om^2 C_T H_x^{s'}} \leq e^{\ld T} \| v_\eps - v \|_{L_\om^2 S_T^{s', \ld}} \leq e^{\ld T} C (T, s, \| u_0 \|_{L_\om^4 H_x^{s}}, \| u_1 \|_{L_\om^2 H_x^{s - 1}}) \big( \eps^{\frac{s - s'}{2}} + \eps^{\frac{1 - s}{2}} \big) ,
\end{align*}

\noi
which gives the desired convergence.
\end{proof}

\appendix

\section{Smoluchowski-Kramers approximation for invariant Gibbs dynamics}
\label{APP:inv}

In the appendix, we consider the commutative diagram in Figure~\ref{fig:diagram} for global-in-time dynamics at Gibbs equilibrium.

\subsection{Gibbs measures}

Let us first keep our discussion at a formal level (i.e. we disregard the renormalization for now) and start with the Gibbs measure $\vec \rho_\eps^N$ for the \HLSe \eqref{NLWNe0}, which is formally given by
\begin{align*}
    d \vec \rho_\eps^N (\uu^N, \dt \uu^N) = Z_{\eps, N}^{-1} e^{- E_\eps^N (\uu^N, \dt \uu^N)} d \uu^N d (\dt \uu^N) ,
\end{align*}

\noi
where $(\uu^N, \dt \uu^N) = ((u^{N, j}, \dt u^{N, j}))_{1 \leq j \leq N}$, $E_\eps^N$ is as defined in \eqref{defEeN}, and $Z_{\eps, N}$ (along with all the ``$Z$'''s below) is a normalizing factor. To be more precise, we write $\vec \rho_\eps^N$ as a weighted Gaussian measure as follows. Given $s \in \R$ and $0 < \eps \leq 1$, we define the fractional Gaussian field $\mu_{s, \eps}$ with scaling parameter $\eps$ as
\begin{align}
    d \mu_{s, \eps} (u) = Z_{s, \eps}^{-1} e^{- \frac 12 \eps^2 \| u \|_{H^s (\T^2)}^2} du \deff \prod_{n \in \Z^2} Z_{s, \eps, n}^{-1} e^{- \frac 12 \eps^2 \jb{n}^{2 s} |\ft u (n)|^2} d \ft u (n) .
\label{defmu}
\end{align}

\noi
From \cite[Lemma~B.1]{BTz08}, we see that $\mu_{s, \eps}$ is a Gaussian probability measure on $H^{s - 1 - \dl} (\T^2) \setminus H^{s - 1} (\T^2)$ for any $\dl > 0$. When $s = 1$ and $\eps = 1$, this measure is the massive Gaussian free field on $\T^2$, in which case we write $\mu_1 = \mu_{1, 1}$ for simplicity; when $s = 0$ and $\eps = 1$, it is the white noise measure on $\T^2$, in which case we write $\mu_0 = \mu_{0, 1}$ for simplicity. Given a measure $\mu$, we denote by $\mu^{\otimes N}$ the $N$-fold product measure $\mu \otimes \cdots \otimes \mu$. Then, defining the measure
\begin{align}
    d \rho^N (\uu^N) = Z_N^{-1} \exp \bigg( - \frac{1}{4N} \int_{\T^2} \Big( \sum_{j = 1}^N (u^{N, j})^2 \Big)^2 dx \bigg) d \mu_1^{\otimes N} (\uu^N) ,
\label{rhoN0}
\end{align}

\noi
we can write the joint measure as
\begin{align}
    d \vec \rho_\eps^N (\uu^N, \dt \uu^N) 
    = d \rho^N (\uu^N) \otimes d \mu_{0, \eps}^{\otimes N} (\dt \uu^N) .
\label{rhoeN0}
\end{align}

\noi
As mentioned in the introduction, the measure $\rho^N$ is called the $O(N)$ linear sigma model and was studied by Wilson in \cite{Wil}.
Using the notations from \eqref{Hamil}, we can write the \HLSe \eqref{NLWNe0} as
\begin{align*}
    \eps^2 \dt \begin{pmatrix}
        \uu^N \\ \dt \uu^N
    \end{pmatrix} 
    = \begin{pmatrix}
        0 & \textbf{Id}^N \\
        - \textbf{Id}^N & 0
    \end{pmatrix}
    \nb_{(L^2 (\T^2))^{\otimes 2N}} E_\eps^N (\uu^N, \dt \uu^N) 
    + \begin{pmatrix}
        0 \\ - \dt \uu^N + \sqrt{2} \pmb{\xi}^N
    \end{pmatrix} 
\end{align*}

\noi
with $\pmb{\xi}^N = (\xi^j)_{1 \leq j \leq N}$, which can be viewed as a superposition of the coupled NLW system \eqref{NLWsys} and an Ornstein-Uhlenbeck process (for the component $\dt \uu^N$). In particular, the coupled NLW system \eqref{NLWsys} formally preserves the Gibbs measure $\vec \rho_\eps^N$ and the Ornstein-Uhlenbeck process (for the component $\dt \uu^N$) preserves the scaled white noise measure $\mu_{0, \eps}^{\otimes N}$. Thus, in view of \eqref{rhoeN0}, we see that the \HLSe \eqref{NLWNe0} formally preserves the Gibbs measure $\vec \rho_\eps^N$. 
This invariance is also expected from the stochastic quantization point of view, as the \HLSe \eqref{NLWNe0} is the so-called canonical stochastic quantization or the Hamiltonian stochastic quantization for the Gibbs measure $\vec \rho_\eps^N$; see \cite{RSS}.

Let us now take into account the renormalization, which is required to make sense of the interaction potential in \eqref{rhoN0} due to the fact that the Gaussian free field $\mu_1$ is only supported on distributions of negative regularity. For the $O(N)$ linear sigma model $\rho^N$, we have
\begin{align}
    d \rho^N (\uu^N) = Z_N^{-1} \exp \bigg( - \frac{1}{4N} \int_{\T^2} \wick{\Big( \sum_{j = 1}^N (u^{N, j})^2 \Big)^2} dx \bigg) d \mu_1^{\otimes N} (\uu^N) .
\label{rhoN}
\end{align}

\noi
Here, the renormalized interaction potential is defined by
\begin{align*}
    \wick{\Big( \sum_{j = 1}^N (u^{N, j})^2 \Big)^2} \, = \sum_{k = 1}^N \sum_{j = 1}^N \wick{(u^{N, k})^2 (u^{N, j})^2}
\end{align*}

\noi
with
\begin{align*}
    \wick{(u^{N, k})^2 (u^{N, j})^2} \, 
    = \begin{cases}
        \wick{(u^{N, j})^4} & \text{if } k = j \\
        \wick{(u^{N, k})^2} \wick{(u^{N, j})^2} & \text{if } k \neq j ,
    \end{cases}
\end{align*}

\noi
where $\wick{(u^{N, j})^2}$ and $\wick{(u^{N, j})^4}$ denote the standard Wick powers. For a rigorous construction of $\rho^N$, one can adapt the procedure from the scalar case (see \cite{Nel, DPT, OTh18}) to the vector-valued setting here (see also \cite[Remark~5.8]{SSZZ}). Then, from \eqref{rhoeN0} and \eqref{rhoN}, the renormalized Gibbs measure $\vec \rho_\eps^N$ is given by
\begin{align}
\begin{split}
    d \vec \rho_\eps^N (\uu^N, \dt \uu^N) &= d \rho^N (\uu^N) \otimes d \mu_{0, \eps}^{\otimes N} (\dt \uu^N) \\
    &= Z_N^{-1} \exp \bigg( - \frac{1}{4N} \int_{\T^2} \wick{\Big( \sum_{j = 1}^N (u^{N, j})^2 \Big)^2} dx \bigg) d \mu_1^{\otimes N} (\uu^N) \otimes d \mu_{0, \eps}^{\otimes N} (\dt \uu^N) .
\end{split}
\label{rhoeN}
\end{align}

\subsection{On stochastic objects}

Before moving on to the dynamical problem, we need to first discuss the solutions to the linear equations just as the stochastic convolutions $\Psi_\eps^j$ satisfying \eqref{Psije} and $\Psi_0^j$ satisfying \eqref{Psij0}. However, in the setting of Gibbs dynamics, we need to incorporate the random initial data.

Let $\{ \phi_0^j, \phi_1^j \}_{j \in \N}$ be a set of independent Gaussian random distributions such that
\begin{align*}
\Law (\phi_0^j) = \mu_1 \quad \text{and} \quad \Law (\phi_1^j) = \mu_0
\end{align*}

\noi
for all $j \in \N$, where $\mu_1$ is the massive Gaussian free field on $\T^2$ and $\mu_0$ is the white noise measure on $\T^2$.
In fact, we can write
\begin{align}
\phi_0^j = \frac{1}{2 \pi} \sum_{n \in \Z^2} \frac{g_n^j}{\jb{n}} e^{in \cdot x} \quad \text{and} \quad \phi_1^j = \frac{1}{2 \pi} \sum_{n \in \Z^2} h_n^j e^{in \cdot x} ,
\label{rand_init}
\end{align}

\noi
where $\{ g_n^j, h_n^j \}_{n \in \Z^2, j \in \N}$ are independent standard complex-valued Gaussian random variables conditioned such that $g_{-n}^j = \cj{g_n^j}$ and $h_{-n}^j = \cj{h_n^j}$ for any $n \in \Z^2$ and $j \in \N$.
By scaling, we note that given $0 < \eps \leq 1$, the family $\{ \eps^{-1} \phi_1^j \}_{j \in \N}$ satisfies
\begin{align*}
    \Law (\eps^{-1} \phi_1^j) = \mu_{0, \eps} \quad \text{for all } j \in \N ,
\end{align*}

\noi
where the Gaussian measure $\mu_{0, \eps}$ is defined in \eqref{defmu}.

Given $j \in \N$ and $0 < \eps \leq 1$, we define $\Phi_\eps^j$ as the stochastic object satisfying the following linear stochastic damped wave equation with a parameter $\eps$ and Gaussian initial data:
\begin{align}
\begin{cases}
    (\eps^2 \dt^2 + \dt + 1 - \Dl) \Phi_\eps^j = \sqrt{2} \xi^j \\
    (\Phi_\eps^j, \dt \Phi_\eps^j) |_{t = 0} = (\phi_0^j, \eps^{-1} \phi_1^j) .
\end{cases}
\label{Phije}
\end{align}

\noi
Here, the family of Gaussian initial data $\{ \phi_0^j, \phi_1^j \}_{j \in \N}$ is assumed to be independent of the family of space-time white noises $\{ \xi^j \}_{j \in \N}$. 
When $\eps = 0$, we define $\Phi_0^j$ as the stochastic convolution satisfying the following linear stochastic heat equation with Gaussian initial data:
\begin{align}
\begin{cases}
(\dt + 1 - \Dl) \Phi_0^j = \sqrt{2} \xi^j \\
\Phi_0^j |_{t = 0} = \phi_0^j .
\end{cases}
\label{Phij0}
\end{align}

\noi
One can compare $\Phi_\eps^j$ with $\Psi_\eps^j$ satisfying \eqref{Psije} and $\Phi_0^j$ with $\Psi_0^j$ satisfying \eqref{Psij0}. 

From \eqref{Phije} and \eqref{Phij0}, given $0 < \eps \leq 1$, we may write
\begin{align*}
    \Phi_\eps^j (t) = P_\eps (t) (\phi_0^j, \eps^{-1} \phi_1^j) + \sqrt{2} \int_0^t \eps^{-2} \D_\eps (t - t') d W^j (t')
\end{align*}

\noi
and
\begin{align*}
    \Phi_0^j (t) = P_0 (t) \phi_0^j + \sqrt{2} \int_0^t P_0 (t - t') d W^j (t') ,
\end{align*}

\noi
where $P_\eps (t)$ is defined in \eqref{defPe}, $\D_\eps (t)$ is defined in \eqref{defDe}, $P_0 (t)$ is defined in \eqref{defP0}, and $W^j$'s are the cylindrical Wiener processes defined in \eqref{Wj}.

Given $M \in \N$, $j \in \N$, and $0 \leq \eps \leq 1$, we define $\Phi_{\eps, M}^j \deff P_{\leq M} \Phi_{\eps}^j$. Let us compute the variance of $\Phi_{\eps, M}^j$.
\begin{lemma}
\label{LEM:var}
Let $0 \leq \eps \leq 1$. Then, for each $n \in \Z^2$, we have
\begin{align*}
\E \Big[ | \ft{\Phi_\eps^j} (t, n) |^2 \Big] = \frac{1}{\jb{n}^2} 
\end{align*}

\noi
for any $t \geq 0$.
\end{lemma}

\begin{proof}
The case $\eps = 0$ is easier, and so we only focus on the case when $0 < \eps \leq 1$. In the low frequency case $\jb{n} \leq (2 \eps)^{-1}$, we use \eqref{rand_init}, independence, \eqref{defPe}, and \eqref{defDe} to compute that
\begin{align*}
\E &\big[ \big| \F \big( P_\eps (t) (\phi_0^j, \eps^{-1} \phi_1^j) \big) (n) \big|^2 \big] \\
&= e^{- \frac{t}{\eps^2}} \Big( \frac{e^{t \ld_\eps (n)} - e^{-t \ld_\eps (n)}}{4 \eps^2 \ld_\eps (n)} + \frac{e^{t \ld_\eps (n)} + e^{- t \ld_\eps (n)}}{2} \Big)^2 \frac{1}{\jb{n}^2} + e^{- \frac{t}{\eps^2}} \Big( \frac{e^{t \ld_\eps (n)} - e^{- t \ld_\eps (n)}}{2 \eps \ld_\eps (n)} \Big)^2 \\
&= e^{- \frac{t}{\eps^2}} \frac{e^{2 t \ld_\eps (n)} + e^{- 2 t \ld_\eps (n)} - 2}{16 \eps^4 \ld_\eps (n)^2 \jb{n}^2} + e^{- \frac{t}{\eps^2}} \frac{e^{2 t \ld_\eps (n)} - e^{- 2 t \ld_\eps (n)}}{4 \eps^2 \ld_\eps (n) \jb{n}^2} 
+ e^{- \frac{t}{\eps^2}} \frac{e^{2 t \ld_\eps (n)} + e^{- 2 t \ld_\eps (n)} + 2}{4 \jb{n}^2} \\
&\quad + e^{- \frac{t}{\eps^2}} \frac{e^{2 t \ld_\eps (n)} + e^{- 2 t \ld_\eps (n)} - 2}{4 \eps^2 \ld_\eps (n)^2}
\end{align*}

\noi
and
\begin{align*}
\E &\bigg[ \bigg| \F \bigg( \sqrt{2} \int_0^t \eps^{-2} \D_\eps (t - t') d W^j (t') \bigg) (n) \bigg|^2 \bigg] \\
&= \frac{1}{2 \eps^4 \ld_\eps (n)^2} \int_0^t  \big( e^{(2 \ld_\eps (n) - \frac{1}{\eps^2}) t'} + e^{(- 2 \ld_\eps (n) - \frac{1}{\eps^2}) t'} - 2 e^{- \frac{t'}{\eps^2}} \big) dt' \\
&= - e^{- \frac{t}{\eps^2}} \frac{e^{2 t \ld_\eps (n)} - e^{-2 t \ld_\eps (n)}}{4 \eps^2 \ld_\eps (n) \jb{n}^2} - e^{- \frac{t}{\eps^2}} \frac{e^{2 t \ld_\eps (n)} + e^{- 2 t \ld_\eps (n)}}{8 \eps^4 \ld_\eps (n)^2 \jb{n}^2} + e^{- \frac{t}{\eps^2}} \frac{1}{\eps^2 \ld_\eps (n)^2} + \frac{1}{\jb{n}^2} ,
\end{align*}

\noi
so that, together with the independence of $\{ \phi_0^j, \phi_1^j \}_{j \in \N}$ and $\{ \xi^j \}_{j \in \N}$, we get
\begin{align*}
\E \Big[ | \ft{\Phi_\eps^j} (t, n) |^2 \Big] &=  
\E \big[ \big| \F \big( P_\eps (t) (\phi_0^j, \eps^{-1} \phi_1^j) \big) (n) \big|^2 \big]
+ \E \bigg[ \bigg| \F \bigg( \sqrt{2} \int_0^t \eps^{-2} \D_\eps (t - t') d W^j (t') \bigg) (n) \bigg|^2 \bigg] \\
&= e^{- \frac{t}{\eps^2}} (e^{2 t \ld_\eps (n)} + e^{- 2 t \ld_\eps (n)}) \Big( - \frac{1}{16 \eps^4 \ld_\eps (n)^2 \jb{n}^2} + \frac{1}{4 \jb{n}^2} + \frac{1}{4 \eps^2 \ld_\eps (n)^2}  \Big) \\
&\quad + e^{- \frac{t}{\eps^2}} \Big( - \frac{1}{8 \eps^4 \ld_\eps (n)^2 \jb{n}^2} + \frac{1}{2 \jb{n}^2} + \frac{1}{2 \eps^2 \ld_\eps (n)^2} \Big) + \frac{1}{\jb{n}^2} \\
&= \frac{1}{\jb{n}^2} .
\end{align*}

\noi
Similarly, in the high frequency case $\jb{n} > (2 \eps)^{-1}$, we compute that
\begin{align*}
\E &\big[ \big| \F \big( P_\eps (t) (\phi_0^j, \eps^{-1} \phi_1^j) \big) (n) \big|^2 \big] \\
&= e^{- \frac{t}{\eps^2}} \Big( \frac{e^{i t \zeta_\eps (n)} - e^{-i t \zeta_\eps (n)}}{4 i \eps^2 \zeta_\eps (n)} + \frac{e^{i t \zeta_\eps (n)} + e^{- i t \zeta_\eps (n)}}{2} \Big)^2 \frac{1}{\jb{n}^2} + e^{- \frac{t}{\eps^2}} \Big( \frac{e^{i t \zeta_\eps (n)} - e^{- i t \zeta_\eps (n)}}{2 i \eps \zeta_\eps (n)} \Big)^2 \\
&= - e^{- \frac{t}{\eps^2}} \frac{e^{2 i t \zeta_\eps (n)} + e^{- 2 i t \zeta_\eps (n)} - 2}{16 \eps^4 \zeta_\eps (n)^2 \jb{n}^2} + e^{- \frac{t}{\eps^2}} \frac{e^{2 i t \zeta_\eps (n)} - e^{- 2 i t \zeta_\eps (n)}}{4 i \eps^2 \zeta_\eps (n) \jb{n}^2} 
+ e^{- \frac{t}{\eps^2}} \frac{e^{2 i t \zeta_\eps (n)} + e^{- 2 i t \zeta_\eps (n)} + 2}{4 \jb{n}^2} \\
&\quad - e^{- \frac{t}{\eps^2}} \frac{e^{2 i t \zeta_\eps (n)} + e^{- 2 i t \zeta_\eps (n)} - 2}{4 \eps^2 \zeta_\eps (n)^2}
\end{align*}

\noi
and
\begin{align*}
\E &\bigg[ \bigg| \F \bigg( \sqrt{2} \int_0^t \eps^{-2} \D_\eps (t - t') d W^j (t') \bigg) (n) \bigg|^2 \bigg] \\
&= - \frac{1}{2 \eps^4 \zeta_\eps (n)^2} \int_0^t  \big( e^{(2 i \zeta_\eps (n) - \frac{1}{\eps^2}) t'} + e^{(- 2 i \zeta_\eps (n) - \frac{1}{\eps^2}) t'} - 2 e^{- \frac{t'}{\eps^2}} \big) dt' \\
&= - e^{- \frac{t}{\eps^2}} \frac{e^{2 i t \zeta_\eps (n)} - e^{-2 i t \zeta_\eps (n)}}{4 i \eps^2 \zeta_\eps (n) \jb{n}^2} + e^{- \frac{t}{\eps^2}} \frac{e^{2 i t \zeta_\eps (n)} + e^{- 2 i t \zeta_\eps (n)}}{8 \eps^4 \zeta_\eps (n)^2 \jb{n}^2} - e^{- \frac{t}{\eps^2}} \frac{1}{\eps^2 \zeta_\eps (n)^2} + \frac{1}{\jb{n}^2} ,
\end{align*}

\noi
so that, together with the independence of $\{ \phi_0^j, \phi_1^j \}_{j \in \N}$ and $\{ \xi^j \}_{j \in \N}$, we get
\begin{align*}
\E \Big[ | \ft{\Phi_\eps^j} (t, n) |^2 \Big] &=  
\E \big[ \big| \F \big( P_\eps (t) (\phi_0^j, \eps^{-1} \phi_1^j) \big) (n) \big|^2 \big]
+ \E \bigg[ \bigg| \F \bigg( \sqrt{2} \int_0^t \eps^{-2} \D_\eps (t - t') d W^j (t') \bigg) (n) \bigg|^2 \bigg] \\
&= e^{- \frac{t}{\eps^2}} (e^{2 i t \zeta_\eps (n)} + e^{- 2 i t \zeta_\eps (n)}) \Big(  \frac{1}{16 \eps^4 \zeta_\eps (n)^2 \jb{n}^2} + \frac{1}{4 \jb{n}^2} - \frac{1}{4 \eps^2 \zeta_\eps (n)^2} \Big) \\
&\quad + e^{- \frac{t}{\eps^2}} \Big( \frac{1}{8 \eps^4 \zeta_\eps (n)^2 \jb{n}^2} + \frac{1}{2 \jb{n}^2} - \frac{1}{2 \eps^2 \zeta_\eps (n)^2} \Big) + \frac{1}{\jb{n}^2} \\
&= \frac{1}{\jb{n}^2} .
\end{align*}

\noi
This finishes the proof.
\end{proof}

From Lemma~\ref{LEM:var}, we have
\begin{align*}
    \al_M \deff \E \big[ |\Phi_{\eps, M}^j (t)|^2 \big] = \sum_{\substack{n \in \Z^2 \\ |n| \leq M}} \frac{1}{\jb{n}^2} \sim \log M ,
\end{align*}

\noi
which diverges as $M \to \infty$. We see that unlike $\sigma_{\eps, M}$ in \eqref{sigmaM}, $\al_M$ is independent of $\eps$ and $t$. 
This is due to the fact that the Gaussian measure $\mu_1 \otimes \mu_{0, \eps}$ (see \eqref{defmu}) is invariant under the dynamics of the linear stochastic damped wave equation \eqref{Phije} for $0 < \eps \leq 1$ and $\mu_1$ is invariant under the dynamics of the linear stochastic heat equation \eqref{Phij0} for $\eps = 0$.

Similar to \eqref{wick0}, given $k, j \in \N$, we define the Wick products 
\begin{align*}
\begin{split}
    \wick{ \Phi_{\eps, M}^k \Phi_{\eps, M}^j } &\deff
    \begin{cases}
        H_2 ( \Phi_{\eps, M}^j ; \al_{M} ) & \text{if } k = j \\
        \Phi_{\eps, M}^k \Phi_{\eps, M}^j & \text{if } k \neq j ,
    \end{cases} \\
    \wick{ (\Phi_{\eps, M}^k)^2 \Phi_{\eps, M}^j } &\deff
    \begin{cases}
        H_3 ( \Phi_{\eps, M}^j ; \al_{M} ) & \text{if } k = j \\
        H_2 ( \Phi_{\eps, M}^k ; \al_{M} ) \Phi_{\eps, M}^j & \text{if } k \neq j ,
    \end{cases}
\end{split}
\end{align*}

\noi
where $H_2 (\cdot ; \s)$ and $H_3 (\cdot; \s)$ denote Hermite polynomials of degree 2 and 3, respectively, with variance parameter $\s > 0$. Then, we define
\begin{align}
\begin{split}
    \wick{ \Phi_{\eps}^k \Phi_{\eps}^j } &\deff \lim_{M \to \infty} \wick{ \Phi_{\eps, M}^k \Phi_{\eps, M}^j } , \\
    \wick{ (\Phi_{\eps}^k)^2 \Phi_{\eps}^j } &\deff \lim_{M \to \infty} \wick{ (\Phi_{\eps, M}^k)^2 \Phi_{\eps, M}^j } .
\end{split}
\label{wick2}
\end{align}

\noi
We will see in Lemma~\ref{LEM:sto2} below that the limits in \eqref{wick2} exist almost surely in the space $C(\R_+; W^{- \s, \infty})$ for any $\s > 0$. Before showing the regularity properties of $\Phi_\eps^j$'s and their Wick products, we first prove the following useful estimates for $\ft \D_\eps (t, n)$ defined in \eqref{Den}.

\begin{lemma}
\label{LEM:Den}
Let $t \geq 0$ and $n \in \Z^2$.

\smallskip \noi
\textup{(i)} For any $0 < \eps \leq 1$ and $0 \leq \gamma \leq 1$, we have
\begin{align}
&| \eps^{-2} \ft \D_\eps (t, n) | \les 1 , \label{Dest1-1} \\
&| \dt \ft \D_\eps (t, n) | \les 1 , \label{Dest1-2} \\
&| \eps^{-1} \ft \D_\eps (t, n) | \les \eps^{\gamma} \jb{n}^{-1 + \gamma} , \label{Dest1-3}
\end{align}

\noi
where the implicit constants are independent of $t$ and $\eps$.

\smallskip \noi
\textup{(ii)} For any $0 < \eps \leq 1$ and $h \in \R$ such that $0 < \eps + h \leq 1$, we have
\begin{align}
&\big| (\eps^{-2} + \dt) \ft \D_\eps (t, n) - e^{- t \jb{n}^2} \big| \les \eps^{\gamma_1} \jb{n}^{\gamma_2} , \label{Dest2-1} \\
&\big| \big( (\eps + h)^{-2} + \dt \big) \ft{\mathcal{D}}_{\eps + h} (t, n) - ( \eps^{-2} + \dt ) \ft{\mathcal{D}}_\eps (t, n) \big| \les |h|^{\gamma_1} \jb{n}^{\gamma_2}, \label{Dest2-2} \\
&\big| (\eps + h)^{-1} \ft{\mathcal{D}}_{\eps + h} (t, n) - \eps^{-1} \ft{\mathcal{D}}_\eps (t, n) \big| \les |h|^{\gamma_1} \jb{n}^{- 1 + \gamma_2} \label{Dest2-3} 
\end{align}

\noi
for some $\gamma_1, \gamma_2 > 0$ arbitrarily small, where the implicit constants are independent of $t$ and $\eps$.

\smallskip \noi
\textup{(iii)} For any $0 < \eps \leq 1$ and $h \in \R$ such that $t + h \geq 0$, we have
\begin{align}
&\big| e^{- (t + h) \jb{n}^2} - e^{- t \jb{n}^2} \big| \les |h|^{\gamma_1} \jb{n}^{\gamma_2},  \label{Dest3-1} \\
&\big| (\eps^{-2} + \dt) \ft{\mathcal{D}}_{\eps} (t + h, n) - (\eps^{-2} + \dt) \ft{\mathcal{D}}_\eps (t, n) \big| \les |h|^{\gamma_1} \jb{n}^{\gamma_2}, \label{Dest3-2} \\
&\big| \eps^{-1} \ft{\mathcal{D}}_{\eps} (t + h, n) - \eps^{-1} \ft{\mathcal{D}}_\eps (t, n) \big| \les |h|^{\gamma_1} \jb{n}^{- 1 + \gamma_2}  \label{Dest3-3} 
\end{align}

\noi
for some $\gamma_1, \gamma_2 > 0$ arbitrarily small, where the implicit constants are independent of $t$ and $\eps$.
\end{lemma}

\begin{proof}
(i) All estimates \eqref{Dest1-1}, \eqref{Dest1-2}, and \eqref{Dest1-3} follow directly from \cite[Lemma~3.5]{Zine}. 

\smallskip \noi
(ii) 
For \eqref{Dest2-1}, we see from \cite[Lemma~3.6]{Zine} that the estimate holds if $\jb{n} \les \eps^{-1 + \ta}$ for some $\ta > 0$ small. If $\jb{n} \ges \eps^{-1 + \ta}$, the estimate follows directly from \eqref{Dest1-1}, \eqref{Dest1-2}, and the fact that $1 \les \eps^{1 - \ta} \jb{n}$.

We now look at \eqref{Dest2-2}, where we only consider $h > 0$. 
Following the proof of \cite[Lemma~2.1.5]{Zine23}, we write $\ft \D_\eps (t, n) = e^{- \frac{t}{2 \eps^2}} S_\eps (t, n)$, where we have
\begin{align}
S_\eps (t, n) = t \phi \big( t^2 (4 \eps^4)^{-1} - t^2 \eps^{-2} \jb{n}^2 \big)
\label{help1}
\end{align}

\noi
with $\phi$ being an analytic function on $\R$ given by 
\begin{align*}
\phi (x) = 
\begin{cases}
\frac{\sinh (\sqrt{x})}{\sqrt{x}} & \text{if } x \geq 0 \\
\frac{\sin (\sqrt{-x})}{\sqrt{-x}} & \text{if } x < 0 
\end{cases}
\end{align*}

\noi
whose derivatives satisfy
\begin{align}
\phi^{(k)} (x) \les_k e^{\sqrt{x}} \ind_{\{x \geq 0\}} +  \ind_{\{x < 0\}}
\label{help2}
\end{align}

\noi
for any $k \in \N \cup \{0\}$. 
Then, following the steps in \cite[Lemma~2.1.5]{Zine23} with \eqref{help1} and \eqref{help2}, we compute that
\begin{align}
| \partial_\eps \ft \D_\eps (t, n) | \les \eps^{-5} \jb{n}^2
\label{Desth1}
\end{align}

\noi
and
\begin{align}
| \partial_\eps \dt \ft \D_\eps (t, n) | \les \eps^{-9} \jb{n}^4 .
\label{Desth2}
\end{align}

\noi
Thus, from the mean value theorem, \eqref{Desth1}, and \eqref{Desth2}, we get
\begin{align}
\big| \big( (\eps + h)^{-2} + \dt \big) \ft{\mathcal{D}}_{\eps + h} (t, n) - ( \eps^{-2} + \dt ) \ft{\mathcal{D}}_\eps (t, n) \big| \les h \eps^{-9} \jb{n}^4 .
\label{Dest2-2-1}
\end{align}

\noi
If $h < \eps^{10}$, we interpolate \eqref{Dest2-2-1} and the bounds in \eqref{Dest1-1} and \eqref{Dest1-2} to obtain the desired estimate.
If $h \geq \eps^{10}$, we have $\eps \leq h^{\frac{1}{10}}$, and so from \eqref{Dest2-1}, we have
\begin{align*}
\big| \big( (\eps + h)^{-2} + \dt \big) \ft{\mathcal{D}}_{\eps + h} (t, n) - ( \eps^{-2} + \dt ) \ft{\mathcal{D}}_\eps (t, n) \big| \les (\eps + h)^{\gamma_1'} \jb{n}^{\gamma_2'} \les h^{\frac{\gamma_1'}{10}} \jb{n}^{\gamma_2'}
\end{align*}

\noi
for some $\gamma_1', \gamma_2' > 0$, giving the desired estimate.

For \eqref{Dest2-3}, we use the mean value theorem and \eqref{Desth1} to get
\begin{align}
\big| (\eps + h)^{-1} \ft{\mathcal{D}}_{\eps + h} (t, n) - \eps^{-1} \ft{\mathcal{D}}_\eps (t, n) \big| \les h \eps^{-6} \jb{n}^2 .
\label{Dest2-3-1}
\end{align}

\noi
We then interpolate \eqref{Dest2-3-1} and \eqref{Dest1-3} to obtain the desired estimate.

\smallskip \noi
(iii) 
The estimate \eqref{Dest3-1} follows directly from the mean value theorem.

We now look at \eqref{Dest3-2}, where we only consider $h > 0$. Following the steps in \cite[Lemma~2.1.5]{Zine23} with \eqref{help1} and \eqref{help2}, we compute that
\begin{align}
|\dt^2 \ft \D_\eps (t, n)| \les \eps^{-8} \jb{n}^4 .
\label{Desth3}
\end{align}

\noi
Thus, from the mean value theorem, \eqref{Dest1-2}, and \eqref{Desth3}, we get
\begin{align}
\big| (\eps^{-2} + \dt) \ft{\mathcal{D}}_{\eps} (t + h, n) - (\eps^{-2} + \dt) \ft{\mathcal{D}}_\eps (t, n) \big| \les h \eps^{-8} \jb{n}^4 .
\label{Dest3-2-1}
\end{align}

\noi
If $h < \eps^{10}$, we interpolate \eqref{Dest3-2-1} and the bounds in \eqref{Dest1-1} and \eqref{Dest1-2} to obtain the desired estimate.
If $h \geq \eps^{10}$, we have $\eps \leq h^{\frac{1}{10}}$, and so from \eqref{Dest2-1} and \eqref{Dest3-1}, we have
\begin{align*}
\big| (\eps^{-2} + \dt) \ft{\mathcal{D}}_{\eps} (t + h, n) - (\eps^{-2} + \dt) \ft{\mathcal{D}}_\eps (t, n) \big| \les \eps^{\gamma_1'} \jb{n}^{\gamma_2'} + h^{\gamma_1'} \jb{n}^{\gamma_2'} \les h^{\frac{\gamma_1'}{10}} \jb{n}^{\gamma_2'}
\end{align*}

\noi
for some $\gamma_1', \gamma_2' > 0$, giving the desired estimate.

For \eqref{Dest3-3}, we use the mean value theorem and \eqref{Dest1-2} to get
\begin{align}
\big| \eps^{-1} \ft{\mathcal{D}}_{\eps} (t + h, n) - \eps^{-1} \ft{\mathcal{D}}_\eps (t, n) \big| \les h \eps^{-1} .
\label{Dest3-3-1}
\end{align}

\noi
We then interpolate \eqref{Dest3-3-1} and \eqref{Dest1-3} to obtain the desired estimate.
\end{proof}

We are now ready to show that $\Phi_\eps^j$ and its Wick products and powers share the same regularity properties with those of $\Psi_\eps^j$. 

\begin{lemma}
\label{LEM:sto2}
Lemma~\ref{LEM:sto} holds with $\Psi_{\eps, M}^j$ replaced by $\Phi_{\eps, M}^j$ and $\Psi_\eps^j$ replaced by $\Phi_\eps^j$.
\end{lemma}

\begin{proof}
From Lemma~\ref{LEM:sto}, we know that we only need to study $P_\eps (t) (\phi_0^j, \eps^{-1} \phi_1^j)$ for $0 < \eps \leq 1$ and $P_0 (t) (\phi_0^j)$ with $P_\eps (t)$ defined in \eqref{defPe}, $P_0 (t)$ defined in \eqref{defP0}, and $\phi_0^j$ and $\phi_1^j$ defined in \eqref{rand_init}. From the proof of Lemma~\ref{LEM:sto}, we know that the desired properties for $\Phi_\eps^j$'s and their Wick products follow once we prove that for any $0 < \eps \leq 1$, $t \geq 0$, $n \in \Z^2$, and $h_1, h_2 \in \R \setminus \{0\}$ satisfying $0 < \eps + h_1 \leq 1$ and $t + h_2 \geq 0$,
\begin{align}
\begin{split}
&\E \Big[ \big| \F \big( P_\eps (t) (\phi_0^j, \eps^{-1} \phi_1^j) \big) (n) \big|^2 \Big] \les \jb{n}^{-2} , \\
&\E \Big[ \big| \F \big( P_{\eps + h_1} (t + h_2) (\phi_0^j, (\eps + h_1)^{-1} \phi_1^j) \big) (n) - \F \big( P_\eps (t) (\phi_0^j, \eps^{-1} \phi_1^j) \big) (n) \big|^2 \Big] \\ 
&\quad \les (|h_1|^{\gamma_1} + |h_2|^{\gamma_1}) \jb{n}^{-2 + \gamma_2} , \\
&\E \Big[ \big| \F \big( P_0 (t) (\phi_0^j) \big) (n) \big|^2 \Big] \les \jb{n}^{-2} , \\
&\E \Big[ \big| \F \big( P_{|h_1|} (t + h_2) (\phi_0^j, |h_1|^{-1} \phi_1^j) \big) (n) - \F \big( P_0 (t) (\phi_0^j) \big) (n) \big|^2 \Big]  
\les (|h_1|^{\gamma_1} + |h_2|^{\gamma_1}) \jb{n}^{-2 + \gamma_2} 
\end{split}
\label{Phin} 
\end{align}

\noi
for some $\gamma_1, \gamma_2 > 0$ arbitrarily small. 
By using \eqref{defPe}, \eqref{rand_init}, and the independence of Gaussian random variables, we see that the estimates in \eqref{Phin} are reduced to the estimates in Lemma~\ref{LEM:Den}. This finishes the proof.
\end{proof}

\subsection{Convergence of invariant Gibbs dynamics}

We now consider the Gibbs dynamics for $\text{HLSM}_{\eps, N}$:
\begin{align}
(\eps^2 \dt^2 + \dt + 1 - \Dl) u_\eps^{N, j} = - \frac{1}{N} \sum_{k = 1}^N \wick{ (u_\eps^{N, k})^2 u_\eps^{N, j} } + \, \sqrt{2} \xi^j, \qquad j = 1, \dots, N ,
\label{HLSMu}
\end{align}

\noi
and the mean-field SNLH:
\begin{align}
(\dt + 1 - \Dl) u^j = - \E [ (u^j)^2 - (\Phi_0^j)^2 ] u^j + \sqrt{2} \xi^j  ,
\label{mfSNLHu}
\end{align}

\noi
where $\Phi_0^j$ is as defined in \eqref{Phij0}.
We also write out the two intermediate equations in the two convergence regimes, namely $\text{PLSM}_N$:
\begin{align}
(\dt + 1 - \Dl) u^{N, j} = - \frac{1}{N} \sum_{k = 1}^N \wick{ (u^{N, k})^2 u^{N, j} } + \, \sqrt{2} \xi^j, \qquad j = 1, \dots, N ,
\label{PLSMu}
\end{align}

\noi
and the mean-field $\text{SdNLW}_\eps$:
\begin{align}
(\eps^2 \dt^2 + \dt + 1 - \Dl) u_\eps^j = - \E [ (u_\eps^j)^2 - (\Phi_\eps^j)^2 ] u_\eps^j + \sqrt{2} \xi^j  ,
\label{mfSdNLWu}
\end{align}

\noi
where $\Phi_\eps^j$ is as defined in \eqref{Phije}.
In the above four equations, we have already renormalized the nonlinearities, and as in the previous sections, we define the solutions to the above four equations as follows. By using the first order expansions:
\begin{align*}
\begin{split}
u_\eps^{N, j} &= \Phi_\eps^j + v_\eps^{N, j}, \qquad j = 1, \dots, N , \\
u^j &= \Phi_0^j + v^j, \qquad j \in \N , \\
u^{N, j} &= \Phi_0^j + v^{N, j}, \qquad j = 1, \dots, N , \\
u_\eps^j &= \Phi_\eps^j + v_\eps^j, \qquad j \in \N ,
\end{split}
\end{align*}

\noi
we say that $\uu_\eps^N = (u_\eps^{N, j})_{1 \leq j \leq N}$, $u^j$, $\uu^N = (u^{N, j})_{1 \leq j \leq N}$, or $u_\eps^j$ is a solution to \HLSe \eqref{HLSMu}, the mean-field SNLH \eqref{mfSNLHu}, \PLS \eqref{PLSMu}, or the mean-field \NLWe \eqref{mfSdNLWu}, respectively, if $\vv_\eps^N = (v_\eps^{N, j})_{1 \leq j \leq N}$, $v^j$, $\vv^N = (v^{N, j})_{1 \leq j \leq N}$, or $v_\eps^j$ satisfies the perturbed $\text{HLSM}_{\eps, N}$
\begin{align}
\begin{split}
(\eps^2 \dt^2 + \dt + 1 - \Dl) v_\eps^{N, j} &= - \frac 1N \sum_{k = 1}^N \big( \wick{ (\Phi_\eps^k)^2 \Phi_\eps^j } + \wick{ (\Phi_\eps^k)^2 } v_\eps^{N, j} + 2 v_\eps^{N, k} \wick{ \Phi_\eps^k \Phi_\eps^j } \\
    &\,\,\,\, + 2 \Phi_\eps^k v_\eps^{N, k} v_\eps^{N, j} + (v_\eps^{N, k})^2 \Phi_\eps^j + (v_\eps^{N, k})^2 v_\eps^{N, j} \big) , \qquad j = 1, \dots, N ,
\end{split}
\label{HLSMv}
\end{align}

\noi
the perturbed mean-field SNLH
\begin{align}
(\dt + 1 - \Dl) v^j = - 2 \E [\Phi_0^j v^j] \Phi_0^j - 2 \E [\Phi_0^j v^j] v^j - \E [(v^j)^2] \Phi_0^j - \E [(v^j)^2] v^j ,
\label{mfSNLHv}
\end{align}

\noi
the perturbed $\text{PLSM}_N$
\begin{align}
\begin{split}
(\dt + 1 - \Dl) v^{N, j} 
= - \frac 1N &\sum_{k = 1}^N \big( \wick{ (\Phi_0^k)^2 \Phi_0^j } + \wick{ (\Phi_0^k)^2 } v^{N, j} + 2 v^{N, k} \wick{ \Phi_0^k \Phi_0^j } \\
    &\, + 2 \Phi_0^k v^{N, k} v^{N, j} + (v^{N, k})^2 \Phi_0^j + (v^{N, k})^2 v^{N, j} \big) , \qquad j = 1, \dots, N ,
\end{split}
\label{PLSMv}
\end{align}

\noi
or the perturbed mean-field \NLWe
\begin{align}
(\eps^2 \dt^2 + \dt + 1 - \Dl) v_\eps^j 
= - 2 \E [\Phi_\eps^j v_\eps^j] \Phi_\eps^j - 2 \E [\Phi_\eps^j v_\eps^j] v_\eps^j - \E [(v_\eps^j)^2] \Phi_\eps^j - \E [(v_\eps^j)^2] v_\eps^j ,
\label{mfSdNLWv}
\end{align}

\noi
respectively, where the Wick products and powers of $\Phi_\eps^j$'s are defined in \eqref{wick2}.

From Lemma~\ref{LEM:sto2}, we know that $\Phi_\eps^j$'s and their Wick products and powers share the same regularity properties with those from the stochastic convolutions $\Psi_\eps^j$'s given by \eqref{Psije} and \eqref{Psij0}. Consequently, with $\Psi_\eps^j$'s replaced by $\Phi_\eps^j$'s, all the global well-posedness results stated in Subsection~\ref{SUB:intro} apply to the above four equations: 
Proposition~\ref{PROP:GWP_NLW} applies to the renormalized \HLSe \eqref{HLSMu}, 
Proposition~\ref{PROP:GWP_NLH} applies to the renormalized mean-field SNLH \eqref{mfSNLHu}, 
Theorem~\ref{THM:conv1}~(i) applies to the renormalized \PLS \eqref{PLSMu}, 
and Theorem~\ref{THM:conv2}~(i) applies to the renormalized mean-field \NLWe \eqref{mfSdNLWu}.

Let us also mention the following invariance results for the above four equations. For each $0 < \eps \leq 1$, the Gibbs measure $\vec \rho_\eps^N$ defined in \eqref{rhoeN} is invariant under the dynamics of the renormalized \HLSe \eqref{HLSMu} in the sense that for each $t \in \R_+$, we have
\begin{align*}
\Law \big( \uu_\eps^N (t), \dt \uu_\eps^N (t) \big) = \vec \rho_\eps^N .
\end{align*}

\noi
This follows directly from the invariance of the frequency-truncated Gibbs measure $\vec \rho_\eps^N$ under the flow of the frequency-truncated \HLSe \eqref{HLSMu} (see, for example, \cite{GKOT, ORTz}) and a PDE approximation argument. In particular, Bourgain's invariant measure argument in \cite{Bour94, Bour96} is not needed.
Also, the Gibbs measure $\rho^N$ defined in \eqref{rhoN} is invariant under the dynamics of the renormalized \PLS \eqref{PLSMu} in the sense that for each $t \in \R_+$, we have
\begin{align*}
\Law \big( \uu^N (t) \big) = \rho^N .
\end{align*}

\noi
This follows from \cite[Lemma~5.7 and Remark~5.8]{SSZZ}, where we note that the system \eqref{PLSMu} is the stochastic quantization of the Gibbs measure $\rho^N$ in \eqref{rhoN}.
Moreover, in view of Lemma~\ref{LEM:var}, we see that $u^j = \Phi_0^j$ is a solution to the mean-field SNLH \eqref{mfSNLHu} that preserves the Gaussian free field $\mu_1$ and that, for each $0 < \eps \leq 1$, $(u_\eps^j, \dt u_\eps^j) = (\Phi_\eps^j, \dt \Phi_\eps^j)$ is a solution to the mean-field \NLWe \eqref{mfSdNLWu} that preserves the Gaussian measure $\mu_1 \otimes \mu_{0, \eps}$.\footnote{A direct but tedious computation similar to that in Lemma~\ref{LEM:var} shows that $\E [ |\ft{\dt \Phi_\eps^j} (t, n)|^2 ] = \eps^{-2}$ for each $n \in \Z^2$ and $t \geq 0$, which shows that $\Law (\dt \Phi_\eps^j (t)) = \mu_{0, \eps}$. Note that each $\ft{\dt \Phi_\eps^j} (t, n)$ diverges as $\eps \to 0$, which is compatible with the fact that $\dt \Phi_0^j = - (1 - \Dl) \Phi_0^j + \sqrt{2} \xi^j$ is almost surely merely a distribution in time.}

\medskip
In order to discuss convergence of dynamics at Gibbs equilibrium, we need to prepare Gibbsian initial data corresponding to the Gibbs measure $\rho^N$. The construction of such Gibbsian initial data is provided by the following proposition, which was proved in \cite[Proposition~6.1]{LLO} based on the study of $\rho^N$ in \cite{DS}.
\begin{proposition}[Gibbsian initial data]
\label{PROP:Gibbs}
There exist a probability space $(\Om_1, \F_1, \PP_1)$ and random distributions $\{ \phi_{0}^j
\}_{j \in \N}$, $\{ \phi_{1}^j \}_{j \in \N}$, and $\vv_0^N = (v_0^{N, j})_{1 \leq j \leq N}$, $N \in \N$, on $(\Om_1, \F_1, \PP_1)$ such that the family $\{ \phi_0^j, \phi_1^j \}_{j \in \N}$ is independent, the random variables $v_0^{N, j}$'s satisfy
\begin{align}
    \sup_{1 \leq j \leq N} \big\| \| v_0^{N, j} \|_{H_x^1} \big\|_{L^2 (\Om_1)} \les N^{- \frac 12} ,
\label{v0init}
\end{align}

\noi
and we have
\begin{align*}
    &\Law (\phi_0^j) = \mu_1 \quad \text{and} \quad \Law (\phi_1^j) = \mu_0 \quad \text{for all } j \in \N , \\
    &\Law \big( (\phi_0^j + v_0^{N, j})_{1 \leq j \leq N} \big) = \rho^N \quad \text{for all } N \in \N ,
\end{align*}

\noi
where $\mu_1$ is the massive Gaussian free field on $\T^2$, $\mu_0$ is the white noise measure on $\T^2$, and $\rho^N$ is the $O(N)$ linear sigma model in \eqref{rhoN}. 
\end{proposition}

Here, we slightly abused notations, since $\phi_0^j$ and $\phi_1^j$ are already used as initial data for the stochastic convolutions $\Phi_\eps^j$ defined in \eqref{Phije} and $\Phi_0^j$ in \eqref{Phij0}. However, with the new initial data $\{\phi_0^j, \phi_1^j\}_{j \in \N}$ provided in Proposition~\ref{PROP:Gibbs}, the statistical properties of the stochastic convolutions $\Phi_\eps^j$ and $\Phi_0^j$ remain unchanged. In particular, the family of initial data $\{\phi_0^j, \phi_1^j\}_{j \in \N}$ provided in Proposition~\ref{PROP:Gibbs} is again assumed to be independent of the family of space-time white noises $\{\xi^j\}_{j \in \N}$, and so Lemma~\ref{LEM:var} and Lemma~\ref{LEM:sto2} remain true with the new data.

In view of Proposition~\ref{PROP:Gibbs}, we can equip the renormalized \HLSe \eqref{HLSMu} with initial data $((\phi_0^j + v_0^{N, j}, \eps^{-1} \phi_1^j))_{1 \leq j \leq N}$ and the renormalized \PLS \eqref{PLSMu} with initial data $(\phi_0^j + v_0^{N, j})_{1 \leq j \leq N}$.
Then, since each $\Phi_\eps^j$ in \eqref{Phije} has Gaussian initial data $(\phi_0^j, \phi_1^j)$, we see that the perturbed \HLSe \eqref{HLSMv} should be equipped with initial data $(\vv_0^N, \mathbf{0}) = ((v_0^{N, j}, 0))_{1 \leq j \leq N}$. Similarly, given that each $\Phi_0^j$ in \eqref{Phij0} has Gaussian initial data $\phi_0^j$, we equip the perturbed \PLS \eqref{PLSMv} with initial data $\vv_0^N = (v_0^{N, j})_{1 \leq j \leq N}$.
In view of the decay of $v_0^{N, j}$ in \eqref{v0init} as $N \to \infty$, we see that the solution $\vv_\eps^N = (v_\eps^{N, j})_{1 \leq j \leq N}$ to the perturbed \HLSe \eqref{HLSMv} and the solution $\vv^N = (v^{N, j})_{1 \leq j \leq N}$ to the perturbed \PLS \eqref{PLSMv} will converge to zero as $N \to \infty$. This is coherent with the mean-field convergence of these two systems, since $v^j \equiv 0$ is the unique solution to the perturbed mean-field SNLH \eqref{mfSNLHv} and $v^j_\eps \equiv 0$ is the unique solution to the perturbed mean-field \NLWe \eqref{mfSdNLWv}, both with zero initial data.

We are now ready to state the convergence results in Figure~\ref{fig:diagram} for invariant Gibbs dynamics. 
In view of the initial data set from Proposition~\ref{PROP:Gibbs} and the regularity properties of the stochastic objects in Lemma~\ref{LEM:sto2}, the proof follows from essentially the same steps for the convergence of general dynamics in Theorem~\ref{THM:conv1}~(ii) and (iii) and Theorem~\ref{THM:conv2}~(ii) and (iii).
\begin{theorem}[Convergence of invariant Gibbs dynamics]
\label{THM:Gibbs}
Let $\{\phi_0^j\}_{j \in \N}$, $\{\phi_1^j\}_{j \in \N}$, and $\vv_0^N = (v_0^{N, j})_{1 \leq j \leq N}$, $N \in \N$, be random distributions on a probability space $(\Om_1, \F_1, \PP_1)$ given by Proposition~\ref{PROP:Gibbs}. Given $N \in \N$ and $0 < \eps \leq 1$, let $\uu_\eps^N = (u_\eps^{N, j})_{1 \leq j \leq N}$ be the solution to the renormalized \HLSe \eqref{HLSMu} with initial data $((\phi_0^j + v_0^{N, j}, \eps^{-1} \phi_1^j))_{1 \leq j \leq N}$ and let $\uu^N = (u^{N, j})_{1 \leq j \leq N}$ be the solution to the renormalized \PLS \eqref{PLSMu} with initial data $(\phi_0^j + v_0^{N, j})_{1 \leq j \leq N}$.

\smallskip \noi
\textup{(i) (Convergence via regime I)} For each fixed $N \in \N$, $\{ \uu_\eps^N \}_{\eps \in (0, 1]}$ converges almost surely to $\uu^N$ in $C (\R_+; H^{- \s} (\T^2)^{\otimes N})$, endowed with the compact open topology in time, as $\eps \to 0$ for any $\s > 0$. 
Moreover, for each fixed $j \in \N$, $\{ u^{N, j} \}_{N \in \N}$ converges in probability to $\Phi_0^j$ in $C (\R_+; H^{- \s} (\T^2))$, endowed with the compact open topology in time, as $N \to \infty$ for any $\s > 0$.

\smallskip \noi
\textup{(ii) (Convergence via regime II)} For each fixed $0 < \eps \leq 1$ and $j \in \N$, $\{ u_\eps^{N, j} \}_{N \in \N}$ converges in probability to $\Phi_\eps^j$ in $C(\R_+; H^{- \s} (\T^2))$, endowed with the compact open topology in time, as $N \to \infty$ for any $\s > 0$.
Moreover, for each fixed $j \in \N$, $\{\Phi_\eps^j\}_{\eps \in (0, 1]}$ converges to $\Phi_0^j$ almost surely in $C(\R_+; H^{- \s} (\T^2))$, endowed with the compact open topology in time, as $\eps \to 0$ for any $\s > 0$.
\end{theorem}

\begin{ackno} \rm
The authors would like to thank Tadahiro Oh for suggesting this problem. 
R.L.~was funded by the Deutsche Forschungsgemeinschaft (DFG, German Research Foundation) -- Project-ID
211504053 -- SFB 1060.
R.L. and S.L. acknowledge support from the Deutsche Forschungsgemeinschaft (DFG, German Research Foundation) under Germany's Excellence Strategy -- EXC-2047/1 -- 390685813.
R.L. and S.L. were also funded by the Deutsche Forschungsgemeinschaft (DFG, German Research 
Foundation) -- Project-ID 539309657 -- SFB 1720.
\end{ackno}


\begin{thebibliography}{99}

\bibitem{BKKLSW}
G.~G.~Batrouni, G.~R.~Katz, A.~S.~Kronfeld, G.~P.~Lepage, B.~Svetitsky, and K.~G.~Wilson,
{\it Langevin simulations of lattice field theories}.
Phys. Rev. D 32 (1985), no. 10, 2736--2747.


\bibitem{BOZ}
\'A.~B\'enyi, T.~Oh, T.~Zhao,
{\it Fractional Leibniz rule on the torus}.
Proc. Amer. Math. Soc. 153 (2025), no. 1, 207--221.


\bibitem{Bla26}
N.~Blassel,
{\it Overdamped limits for Langevin dynamics with position-dependent coefficients via $L^2$-hypocoercivity}.
arXiv:2602.16924 [math.PR].


\bibitem{Bour94}
J.~Bourgain,
{\it Periodic nonlinear Schr\"odinger equation and invariant measures}.
Comm. Math. Phys. 166 (1994), no. 1, 1--26.


\bibitem{Bour96}
J.~Bourgain,
{\it Invariant measures for the 2D-defocusing nonlinear Schr\"odinger equation}.
Comm. Math. Phys. 176 (1996), no. 2, 421--445.


\bibitem{BL26}
C.-E.~Br\'ehier, Z.~Lei,
{\it Strong and weak rates of convergence in the Smoluchowski--Kramers approximation for stochastic partial differential equations}.
arXiv:2604.14752 [math.PR].


\bibitem{BTz08}
N.~Burq, N.~Tzvetkov,
{\it Random data Cauchy theory for supercritical wave equations. I. Local theory}.
Invent. Math. 173 (2008), no. 3, 449--475.


\bibitem{BTz14}
N.~Burq, N.~Tzvetkov,
{\it Probabilistic well-posedness for the cubic wave equation}.
J. Eur. Math. Soc. (JEMS) 16 (2014), no. 1, 1--30.


\bibitem{CB25}
S.~Cerrai, Z.~Brze\'zniak,
{\it Stochastic wave equations with constraints: well-posedness and Smoluchowski-Kramers diffusion approximation}.
Comm. Math. Phys. 406 (2025), no. 9, Paper No. 223, 59 pp.


\bibitem{CD25}
S.~Cerrai, A.~Debussche,
{\it Smoluchowski-Kramers diffusion approximation for systems of stochastic damped wave equations with nonconstant friction}.
Ann. Appl. Probab. 35 (2025), no. 6, 4106--4171.


\bibitem{CF06-1}
S.~Cerrai, M.~Freidlin,
{\it On the Smoluchowski-Kramers approximation for a system with an infinite number of degrees of freedom}.
Probab. Theory Related Fields 135 (2006), no. 3, 363--394.


\bibitem{CF06-2}
S.~Cerrai, M.~Freidlin,
{\it Smoluchowski-Kramers approximation for a general class of SPDEs}.
J. Evol. Equ. 6 (2006), no. 4, 657--689.


\bibitem{CFS17}
S.~Cerrai, M.~Freidlin, M.~Salins,
{\it On the Smoluchowski-Kramers approximation for SPDEs and its interplay with large deviations and long time behavior}.
Discrete Contin. Dyn. Syst. 37 (2017), no. 1, 33--76.


\bibitem{CG20}
S.~Cerrai, N.~Glatt-Holtz,
{\it On the convergence of stationary solutions in the Smoluchowski-Kramers approximation of infinite dimensional systems}.
J. Funct. Anal. 278 (2020), no. 8, 108421, 38 pp.


\bibitem{CS14}
S.~Cerrai, M.~Salins,
{\it Smoluchowski-Kramers approximation and large deviations for infinite dimensional gradient systems}.
Asymptot. Anal. 88 (2014), no. 4, 201--215.


\bibitem{CS16}
S.~Cerrai, M.~Salins,
{\it Smoluchowski-Kramers approximation and large deviations for infinite-dimensional nongradient systems with applications to the exit problem}.
Ann. Probab. 44 (2016), no. 4, 2591--2642.


\bibitem{CS17}
S.~Cerrai, M.~Salins,
{\it On the Smoluchowski-Kramers approximation for a system with infinite degrees of freedom exposed to a magnetic field}.
Stochastic Process. Appl. 127 (2017), no. 1, 273--303.


\bibitem{CX22}
S.~Cerrai, G.~Xi,
{\it A Smoluchowski-Kramers approximation for an infinite dimensional system with state-dependent damping}.
Ann. Probab. 50 (2022), no. 3, 874--904.


\bibitem{CX23}
S.~Cerrai, M.~Xie,
{\it On the small noise limit in the Smoluchowski-Kramers approximation of nonlinear wave equations with variable friction}.
Trans. Amer. Math. Soc. 376 (2023), no. 11, 7651--7689.


\bibitem{CX24}
S.~Cerrai, M.~Xie,
{\it On the small-mass limit for stationary solutions of stochastic wave equations with state dependent friction}.
Appl. Math. Optim. 90 (2024), no. 1, Paper No. 7, 48 pp.


\bibitem{CX25}
S.~Cerrai, M.~Xie,
{\it The small-mass limit for some constrained wave equations with nonlinear conservative noise}.
Electron. J. Probab. 30 (2025), Paper No. 25, 27 pp.


\bibitem{CKSTT02}
J.~Colliander, M.~Keel, G.~Staffilani, H.~Takaoka, T.~Tao,
{\it Almost conservation laws and global rough solutions to a nonlinear Schrödinger equation}.
Math. Res. Lett. 9 (2002), no. 5-6, 659--682.


\bibitem{DPD}
G.~Da Prato, A.~Debussche,
{\it Strong solutions to the stochastic quantization equations}.
Ann. Probab. 31 (2003), no. 4, 1900--1916.


\bibitem{DPT}
G.~Da Prato, L.~Tubaro,
{\it Wick powers in stochastic PDEs: an introduction}.
Quantum and stochastic mathematical physics, 1–15.
Springer Proc. Math. Stat., 377.


\bibitem{DPZ}
G.~Da Prato, J.~Zabczyk,
{\it Stochastic equations in infinite dimensions}.
Second edition.
Encyclopedia Math. Appl., 152.
Cambridge University Press, Cambridge, 2014. xviii+493 pp.


\bibitem{DS}
M.~G.~Delgadino, S.~A.~Smith,
{\it Mass generation for the two dimensional O(N) Linear Sigma Model in the large N limit}.
arXiv:2601.19630 [math.PR].


\bibitem{DKB}
S.~Duane, A.~D.~Kennedy, B.~J.~Pendleton, D.~Roweth,
{\it Hybrid Monte Carlo}.
Phys. Lett. B 195 (1987), no. 2, 216--222.


\bibitem{FHI}
R.~Fukuizumi, M.~Hoshino, T.~Inui,
{\it Non relativistic and ultra relativistic limits in 2D stochastic nonlinear damped Klein-Gordon equation}.
Nonlinearity 35 (2022), no. 6, 2878--2919.
{\it Corrigendum: Non relativistic and ultra relativistic limits in 2D stochastic nonlinear damped Klein-Gordon equation (2022 Nonlinearity 35 2878)}.
Nonlinearity 35 (2022), no. 10, C17--C19.


\bibitem{GJ}
J.~Glimm, A.~Jaffe,
{\it Quantum physics.
A functional integral point of view}.
Springer-Verlag, New York-Berlin, 1981. xx+417 pp.


\bibitem{GH19}
M.~Gubinelli, M.~Hofmanov\'a,
{\it Global solutions to elliptic and parabolic $\Phi^4$ models in Euclidean space}.
Comm. Math. Phys. 368 (2019), no. 3, 1201--1266.


\bibitem{GIP}
M.~Gubinelli, P.~Imkeller, N.~Perkowski,
{\it Paracontrolled distributions and singular PDEs}.
Forum Math. Pi 3 (2015), e6, 75 pp.


\bibitem{GKO}
M.~Gubinelli, H.~Koch, T.~Oh,
{\it Renormalization of the two-dimensional stochastic nonlinear wave equations}.
Trans. Amer. Math. Soc. 370 (2018), no. 10, 7335--7359.


\bibitem{GKOT}
M.~Gubinelli, H.~Koch, T.~Oh, L.~Tolomeo,
{\it Global dynamics for the two-dimensional stochastic nonlinear wave equations}.
Int. Math. Res. Not. IMRN. (2022), no. 21, 16954--16999. 


\bibitem{GV25}
B.~Guelmame, J.~Vovelle,
{\it A Smoluchowski-Kramers approximation for the stochastic variational wave equation}.
arXiv:2511.13567 [math.AP].


\bibitem{GOTW}
T.~Gunaratnam, T.~Oh, N.~Tzvetkov, H.~Weber,
{\it Quasi-invariant Gaussian measures for the nonlinear wave equation in three dimensions}.
Probab. Math. Phys. 3 (2022), no. 2, 343--379.


\bibitem{Kra}
H.~A.~Kramers,
{\it Brownian motion in a field of force and the diffusion model of chemical reactions}.
Physica 7 (1940), 284--304.


\bibitem{LLO}
R.~Liu, S.~Liu, T.~Oh,
{\it Hyperbolic $O(N)$ linear sigma model and its mean-field limit}.
arXiv:2511.21950 [math.AP].


\bibitem{LLX26}
S.~Liu, W.~Liu, L.~Xu,
{\it Long term convergence rate of Smoluchowski-Kramers approximation by Stein's method}.
arXiv:2602.00875 [math.PR].


\bibitem{McK}
H.~P.~McKean,
{\it Statistical mechanics of nonlinear wave equations. IV. Cubic Schr\"odinger}.
Comm. Math. Phys. 168 (1995), no. 3, 479--491.
{\it Erratum: ``Statistical mechanics of nonlinear wave equations. IV. Cubic Schr\"odinger''}.
Comm. Math. Phys. 173 (1995), no. 3, 675.


\bibitem{MW17}
J.-C.~Mourrat, H.~Weber,
{\it Global well-posedness of the dynamic $\Phi^4$ model in the plane}.
Ann. Probab. 45 (2017), no. 4, 2398--2476.


\bibitem{Neal}
R.~M.~Neal,
{\it MCMC using Hamiltonian dynamics}.
Handbook of Markov chain Monte Carlo, 113--162. 
Chapman \& Hall/CRC Handb. Mod. Stat. Methods.
CRC Press, Boca Raton, FL, 2011.


\bibitem{Nel}
E.~Nelson,
{\it A quartic interaction in two dimensions}.
Mathematical Theory of Elementary Particles (Proc. Conf., Dedham, Mass., 1965), pp. 69--73.
MIT Press, Cambridge, Mass.-London, 1966.




\bibitem{OOT1}
T.~Oh, M.~Okamoto, L.~Tolomeo,
{\it Focusing $\Phi^4_3$-model with a Hartree-type nonlinearity}.
Mem. Amer. Math. Soc. 304 (2024), no. 1529, vi+143 pp.


\bibitem{OOT2}
T.~Oh, M.~Okamoto, L.~Tolomeo,
{\it Stochastic quantization of the $\Phi_3^3$-model}.
Mem. Eur. Math. Soc., 16.
EMS Press, Berlin, 2025. viii+145 pp.


\bibitem{OOTz}
T.~Oh, M.~Okamoto, N.~Tzvetkov,
{\it Uniqueness and non-uniqueness of the Gaussian free field evolution under the two-dimensional Wick ordered cubic wave equation}.
Ann. Inst. Henri Poincar\'e Probab. Stat. 60 (2024), no. 3, 1684--1728.


\bibitem{ORTz}
T.~Oh, T.~Robert, N.~Tzvetkov,
{\it Stochastic nonlinear wave dynamics on compact surfaces}.
Ann. H. Lebesgue 6 (2023), 161--223.


\bibitem{OTh18}
T.~Oh, L.~Thomann,
{\it A pedestrian approach to the invariant Gibbs measures for the 2-d defocusing nonlinear Schr\"odinger equations}.
Stoch. Partial Differ. Equ. Anal. Comput. 6 (2018), no. 3, 397--445.


\bibitem{OTWZ}
T.~Oh, L.~Tolomeo, Y.~Wang, G.~Zheng,
{\it Hyperbolic $P(\Phi)_2$-model on the plane}.
Comm. Math. Phys. 407 (2026), no. 2, Paper No. 34, 84 pp.


\bibitem{PW}
G.~Parisi, Y.~S.~Wu,
{\it Perturbation theory without gauge fixing}.
Sci. Sinica 24 (1981), no. 4, 483--496.


\bibitem{RSS}
S.~Ryang, T.~Saito, K.~Shigemoto,
{\it Canonical stochastic quantization}.
Progr. Theoret. Phys. 73 (1985), no. 5, 1295--1298.


\bibitem{Sa19}
M.~Salins,
{\it Smoluchowski-Kramers approximation for the damped stochastic wave equation with multiplicative noise in any spatial dimension}.
Stoch. Partial Differ. Equ. Anal. Comput. 7 (2019), no. 1, 86--122.


\bibitem{Shen}
H.~Shen,
{\it A stochastic PDE approach to large $N$ problems in quantum field theory: a survey}.
J. Math. Phys. 63 (2022), no. 8, Paper No. 081103, 21 pp.


\bibitem{SSZZ}
H.~Shen, S.~A.~Smith, R.~Zhu, X.~Zhu,
{\it Large $N$ limit of the $O(N)$ linear sigma model via stochastic quantization}.
Ann. Probab. 50 (2022), no. 1, 131--202.


\bibitem{SZZ}
H.~Shen, R.~Zhu, X.~Zhu,
{\it Large $N$ limit of the $O(N)$ linear sigma model in 3D}.
Comm. Math. Phys. 394 (2022), no. 3, 953--1009.


\bibitem{Smo}
M.~Smoluchowski,
{\it Drei Vortage \"uber Diffusion Brownsche Bewegung und Koagulation von Kolloidteilchen}.
Phys. Zeit. 17, 557--585 (1916).


\bibitem{Tol}
L.~Tolomeo,
{\it Global well posedness of the two-dimensional stochastic nonlinear wave equation on an unbounded domain}.
Ann. Probab. 49 (2021), no. 3, 1402--1426.


\bibitem{Tol2}
L.~Tolomeo,
{\it Ergodicity for the hyperbolic $P(\Phi)_2$-model}.
arXiv:2310.02190 [math.PR].


\bibitem{TW18}
P.~Tsatsoulis, H.~Weber,
{\it Spectral gap for the stochastic quantization equation on the 2-dimensional torus}.
Ann. Inst. Henri Poincar\'e Probab. Stat. 54 (2018), no. 3, 1204--1249.


\bibitem{Tz10}
N.~Tzvetkov,
{\it Construction of a Gibbs measure associated to the periodic Benjamin-Ono equation}.
Probab. Theory Related Fields 146 (2010), no. 3-4, 481--514.


\bibitem{Wil}
K.~G.~Wilson,
{\it Quantum field-theory models in less than 4 dimensions}.
Phys. Rev. D 7 (1973), 2911--2926.


\bibitem{XZ26}
L.~Xie, X.~Zhang,
{\it Uniform-in-time diffusion approximations for multiscale stochastic systems}.
arXiv:2604.00692 [math.PR].


\bibitem{Zine23}
Y.~Zine,
{\it Convergence problems for singular stochastic dynamics}.
Ph.D. thesis (2023).


\bibitem{Zine}
Y.~Zine,
{\it Smoluchowski-Kramers approximation for singular stochastic wave equations in two dimensions}.
Electron. J. Probab. 30 (2025), Paper No. 88, 49 pp.

\end{thebibliography}
\end{document}